\documentclass[onefignum,onetabnum]{siamonline190516}

\usepackage{amsfonts}
\usepackage{graphicx}
\usepackage{epstopdf}
\ifpdf
  \DeclareGraphicsExtensions{.eps,.pdf,.png,.jpg}
\else
  \DeclareGraphicsExtensions{.eps}
\fi

\usepackage{enumitem}
\setlist[enumerate]{leftmargin=.5in}
\setlist[itemize]{leftmargin=.5in}

\newsiamremark{remark}{Remark}
\newsiamremark{hypothesis}{Hypothesis}
\crefname{hypothesis}{Hypothesis}{Hypotheses}
\newsiamthm{claim}{Claim}

\headers{Inverse Scattering in the high-frequency limit}{S. Chen, Z. Ding, Q. Li, and L. Zepeda-N\'u\~nez}

\author{Dianne Doe\thanks{Imagination Corp., Chicago, IL
  (\email{ddoe@imag.com}, \url{http://www.imag.com/\string~ddoe/}).}
\and Paul T. Frank\thanks{Department of Applied Mathematics, Fictional University, Boise, ID
  (\email{ptfrank@fictional.edu}, \email{jesmith@fictional.edu}).}
\and Jane E. Smith\footnotemark[3]}

\usepackage{amsopn}

 \usepackage[margin=1in]{geometry}
 \usepackage{graphicx}
 \usepackage{hyperref}
\usepackage[mathscr]{euscript}
\usepackage{algpseudocode}
\usepackage{algorithm}
\usepackage{subfig}
\usepackage{mathabx}
\usepackage{mathtools}

\graphicspath{{./figures/}}

\newcommand{\Rb}{\mathbb{R}}

\newcommand{\Sb}{\mathbb{S}}

\newcommand{\Dc}{\mathcal{D}}
\newcommand{\Ec}{\mathcal{E}}
\newcommand{\Fc}{\mathcal{F}}
\newcommand{\Lc}{\mathcal{L}}
\newcommand{\Oc}{\mathcal{O}}

\newcommand{\wk}{{W^k}}

\newcommand{\im}{\mathrm{Im}}

\newcommand{\rd}{\mathrm{d}}
\newcommand{\rH}{\mathrm{H}}
\newcommand{\rL}{\mathrm{L}}

\newcommand{\eps}{\varepsilon}

\newcommand{\supp}{\text{supp}}

\newcommand{\algrule}[1][.2pt]{\hspace*{-.6in}\hrulefill}
\algdef{SE}[SUBALG]{Indent}{EndIndent}{}{\algorithmicend\ }%
\algtext*{Indent}
\algtext*{EndIndent}

\newcommand{\Wk}{W^k}
\newcommand{\uk}{u^k}
\newcommand{\fk}{f^k}
\newcommand{\Gk}{G^k}
\newcommand{\Hk}{H^k}
\newcommand{\uik}{u^{\mathrm{i},k}}
\newcommand{\usk}{u^{\mathrm{s},k}}

\newcommand{\rmd}{\mathrm{d}}

\newcommand{\rmi}{\mathrm{i}}

\newcommand{\rmo}{\mathrm{o}}
\newcommand{\rmr}{\mathrm{r}}
\newcommand{\rms}{\mathrm{s}}

\newcommand{\ri}{{\mathrm{i}}}

\begin{document}

\title{High-frequency limit of the inverse scattering problem: asymptotic convergence from inverse Helmholtz to inverse Liouville\thanks{Submitted to the editors DATE.
\funding{The work of Q.~L. is supported in part by the UW-Madison
Data Initiative, Vilas Young Investigation Award, National Science Foundation under the grant DMS-1750488 and the Office of Naval Research under the grant ONR-N00014-21-1-2140. The work of L.~Z.-N. is supported in part by the National Science Foundation under the grant DMS-2012292. In addition, Q.~L. and L.~Z.-N. are supported by the NSF TRIPODS award 1740707. The views expressed in the article do not necessarily represent the views of the any funding agencies. The authors are grateful for the support.}}}

\author{Shi Chen\thanks{Department of Mathematics, University of Wisconsin-Madison, Madison WI 53706
  (\email{schen636@wisc.edu}, \url{https://simonchenthu.github.io/}.}
  \and Zhiyan Ding\thanks{Department of Mathematics, University of California-Berkeley, Berkeley CA 94720
  (\email{zding.m@math.berkeley.edu}, \url{https://math.berkeley.edu/\string~zding.m/}).}
  \and Qin Li\thanks{Department of Mathematics, University of Wisconsin-Madison, Madison WI 53706
  (\email{qinli@math.wisc.edu}, \url{https://people.math.wisc.edu/\string~qinli/}).}
  \and Leonardo Zepeda-N\'u\~nez\thanks{Department of Mathematics, University of Wisconsin-Madison, Madison WI 53706
  (\email{zepedanunez@wisc.edu}, \url{https://people.math.wisc.edu/\string~lzepeda/}). Now at Google Research.}}
\date{\today}
\maketitle
\begin{abstract}
We investigate the asymptotic relation between the inverse problems relying on the Helmholtz equation and the radiative transfer equation (RTE) as physical models, in the high-frequency limit. In particular, we evaluate the asymptotic convergence of a generalized version of inverse scattering problem based on the Helmholtz equation, to the inverse scattering problem of the Liouville equation (a simplified version of RTE). The two inverse problems are connected through the Wigner transform that translates the wave-type description on the physical space to the kinetic-type description on the phase space, and the Husimi transform that models data localized both in location and direction. The finding suggests that impinging tightly concentrated monochromatic beams can indeed provide stable reconstruction of the medium, asymptotically in the high-frequency regime. This fact stands in contrast with the unstable reconstruction for the classical inverse scattering problem when the probing signals are plane-waves.

\end{abstract}



\section{Introduction}

The wave-particle duality of light has been one of the greatest enigmas in the natural sciences, dating back to Euclid's treatise in light, {\it Catoptrics} ($280$ B.C.) and spanning more than two millennia. In a nutshell, light can be either described as an electromagnetic (EM) wave governed by the Maxwell's equations, or as a stream of particles, called photons, governed by the radiative transport equation (RTE).

Although the advent of quantum mechanics at the onset of the last century partially solved the riddle, due to computational considerations, light continues to be modeled either as a particle or as a wave depending on the target application. Among those applications, inverse problems are perhaps the ones that have gained the most attention in the last decades, which in return have fueled many breakthroughs in telecommunications~\cite{Wh:2001electromagnetic,Wh:2003electromagnetic}, radar~\cite{Ch:2001mathematical}, biomedical imaging~\cite{Sc:1978improved,BaUh:2010inverse} and, more recently, in chip manufacturing~\cite{King:1981principles}. In this context, inverse problems can be roughly described as reconstructing unknown parameters within a domain of interest by data comprised of observations on its boundary.

Unfortunately, the properties of the inverse problems are highly dependent on the specific modeling of the underlying physical phenomena, even though, in principle, they share the same microscopic description. In particular, the stability of the inverse problem, i.e., how sensitive is the reconstruction of the unknown parameter to perturbations in the data, is surprisingly disparate~\cite{NaUhWa:2013increasing,ChLi:2021semiclassical}, thus creating an important gap between the wave and particle descriptions, which we seek to bridge in this paper. We point out that understanding this gap is not only of theoretical importance, it would also play an important role in designing new reconstruction algorithms with improved stability applicable to a broader set of wave-based inverse problems, which are ubiquitous in science~\cite{Ta:1984inversion,RaPoFi:2010seismic,Ol:1906construction} and engineering~\cite{PeWaLo:2015improved,AtAp:1997inverse,CoPrFrSeVeCoBaDeCa:2015use}.

For simplicity, we consider a time-harmonic wave-like description governed by the Helmholtz equation, which can be derived from the time-harmonic Maxwell's equations after some simplifications. Alternatively, the Helmholtz equation can also be obtained by computing the Fourier transform of the constant-density acoustic wave equation at frequency $k$, and is given by\footnote{The domain of definition, source, and boundary conditions will be specified in Section~\ref{sec:setups}.}
\begin{equation}\label{eqn:helmholtz_no_bdy}
\left(\Delta  +k^2n \right) u(x) = 0\,,
\end{equation}
where $u$ is the wave field, and $n(x)$ is the refractive index of the medium. We point out that even if this is a simplified model, it retains the core difficulty of more complex physics.

We also consider a particle-like description governed by the Liouville equation, which is a simplified RTE, given by:
\begin{equation}\label{eqn:RTE}
v\cdot\nabla_xf - \nabla_xn \cdot\nabla_vf = 0\,,
\end{equation}
where $f(x,v)$ is the distribution of photon particles, and $n$ is still the refractive index. The Liouville equation describes the trajectories of photons via its characteristics: $\dot{x}=v$ and $\dot{v} = -\nabla_xn$. For simplicity we neglect the photon interactions which are usually encoded by the collision operator.

Following the wave and photon descriptions, we define the forward problem as calculating either the wave-field, or the photon distribution from the refractive index by solving either the Helmholtz or the Liouville equations. The wave-particle duality, when translated to mathematical language, corresponds to the fact that the solutions obtained by the Helmholtz and Liouville equations are asymptotically close when $k \rightarrow \infty$, see~\cite{BaPaRy:2002radiative}.

For the sake of conciseness, we consider a simplified inverse problem consisting of reconstructing an unknown environment within a domain of interest by probing it with tightly concentrated monochromatic beams originated from the the boundary of the domain, in which the response of the unknown medium to the impinging beam is measured at its boundary. This measurement is performed by a measurement operator that is model-specific and it will play an important role in what follows. For simplicity, we consider the full aperture regime, i.e., we can probe the medium from any direction, and we sample its impulse response in all possible directions. When the beam is modeled as a wave, i.e., using the Helmholtz equation as a forward model, this process can be considered as a {\it generalized} version of the inverse scattering wave problem (which we, for the sake of clarity, just refer to as the {\it generalized Helmholtz scattering problem}. When the beam is modeled as a flux of photons, i.e., using the Liouville equation as a forward model, this process is often referred to as the {\it optical tomography problem}, but we will refer to as the  {\it Liouville scattering problem} in this manuscript.

Although the two different formulations seek to solve the same underlying physical problem, our understanding of the two inverse problems seems to suggest different stability properties. The traditional inverse scattering problem, using either near-field or far-field data is ill-posed: small perturbations in the measurements usually lead to large deviations in the reconstructions~\cite{CoKr:2019inverse,HaHo:2001new}. Thus, sophisticated algorithms~\cite{LiDe:2016full,DeKr:2013new,DeDa:2017numerical,Pr:1999seismic,Ch:1997inverse,BoGiGr:2017high,BaLiLiTr:2015inverse} have been designed to artificially stabilize the process by appropriately restricting the class of possible unknown environments, usually in the form of band-limited environments. Conversely, the inverse Liouville equation is well-conditioned: a small perturbation is reflected by a small error in the reconstruction~\cite{No:1999small}.

Thus the observation that the stability for both problems is different seems to be at odds with the fact that the Liouville equation and the Helmholtz equation are asymptotically close in the high-frequency regime. Fortunately, as what we will see, this somewhat contradictory property stems from the inability of {\it traditional} formulations of the inverse problems to agree in the high-frequency limit. When the measurement operators are accordingly adjusted, we show that the new formulations, which we call the {\it generalized} inverse scattering, are equivalent in the limit as $k \rightarrow \infty$, producing a stable inverse problem. The convergence from the Helmholtz equation to the Liouville equation is conducted through the Wigner transform~\cite{GeMaMaPo:1997homogenization,RyPaKe:1996transport,BaPaRy:2002radiative}, and the convergence of the measuring operators is achieved through the Husimi transform~\cite{BeCaKaPe:2002high}. Both convergences are obtained asymptotically in the $k\to\infty$ limit. This convergence allows us to conclude the following:

\medskip
\emph{The inverse Liouville scattering problem is asymptotically equivalent to the generalized inverse Helmholtz scattering problem in the high-frequency regime.}

\medskip

The current manuscript is dedicated to formulating the statement above in a mathematically precise manner, while providing extensive numerical evidence supporting the statement.

On the mathematical level, the current paper carries the following important features:
\begin{itemize}
    \item The result connects the two seemingly distinct inverse problems, and suggests that in the high-frequency regime, probing an unknown object with a single frequency is already enough for its reconstruction, with properly prepared data in the generalized inverse scattering setting. This partially answers the stability question regarding the inverse scattering.

    \item The result can be viewed as the counterpart of the asymptotic multiscale study conducted in the forward setting. In particular, semi-classical limit is a theory that connects quantum mechanical and the classical mechanical description: the proposed formulation for the inverse scattering problem can be regarded as taking the (semi-)classical limit in the inverse setting, and thus the work carries conceptual merits. This is in line with ~\cite{NaUhWa:2013increasing,ChLi:2021semiclassical}. See also~\cite{LaLiUh:2019inverse} for a different setting.
\end{itemize}

These mathematical understandings also naturally bring numerical and practical benefits. The new inverse wave scattering formulation coupled with PDE-constrained optimization seems to be empirically less prone to cycle-skipping, i.e., convergence to spurious local minima~\cite{ViOp:2009overview}, than its standard counterparts~\cite{ViAsBrMeRiZh:2017introduction,BoGiGr:2017high}, thus potentially opening the way to more robust algorithmic pipelines for inverse problems.

We point out that even though this current study is motivated by the wave-particle duality of light, the current results are also applicable to other oscillatory phenomena, see~\cite{ChLi:2021semiclassical} for a discussion on inverse Schr\"odinger problem in the classical limit.

\subsection*{Organization}
In Section~\ref{sec:setups}, we briefly review the Helmholtz equation and present the corresponding inverse problem that fits the particular experimental setup that allows passing the system to the $k\to\infty$ limit. In Section~\ref{sec:limit}, we discuss the limiting Liouville equation and the inverse Liouville scattering problem, by conducting the Wigner and Husimi transforms. The connections between the two inverse problems will thus be immediate. Finally, we present our numerical evidences that justify the convergence in Section~\ref{sec:numer} and we showcase the stability of the inverse problem in Section~\ref{sec:numer_inverse}.

\section{Experimental setup and inverse problem formulation}\label{sec:setups}
Suppose we use tightly concentrated monochromatic beams, or laser beams, to probe the medium. Each beam impinges in the area of interest, thus producing a scattered field which is then measured by directional receivers\footnote{Experimentally, this is often achieved by placing a collimator before the receiver, and changing the orientation of the collimator.} placed on a manifold around the domain of interest. The data, which is used to reconstruct the optical properties of the medium, is the intensity captured by each receiver for each incoming beam. Thus, the data is indexed by the position and direction of the impinging beam, and the location and direction of the receivers.

\begin{figure}[htbp]
  \centering
  \includegraphics[width=0.3\textwidth]{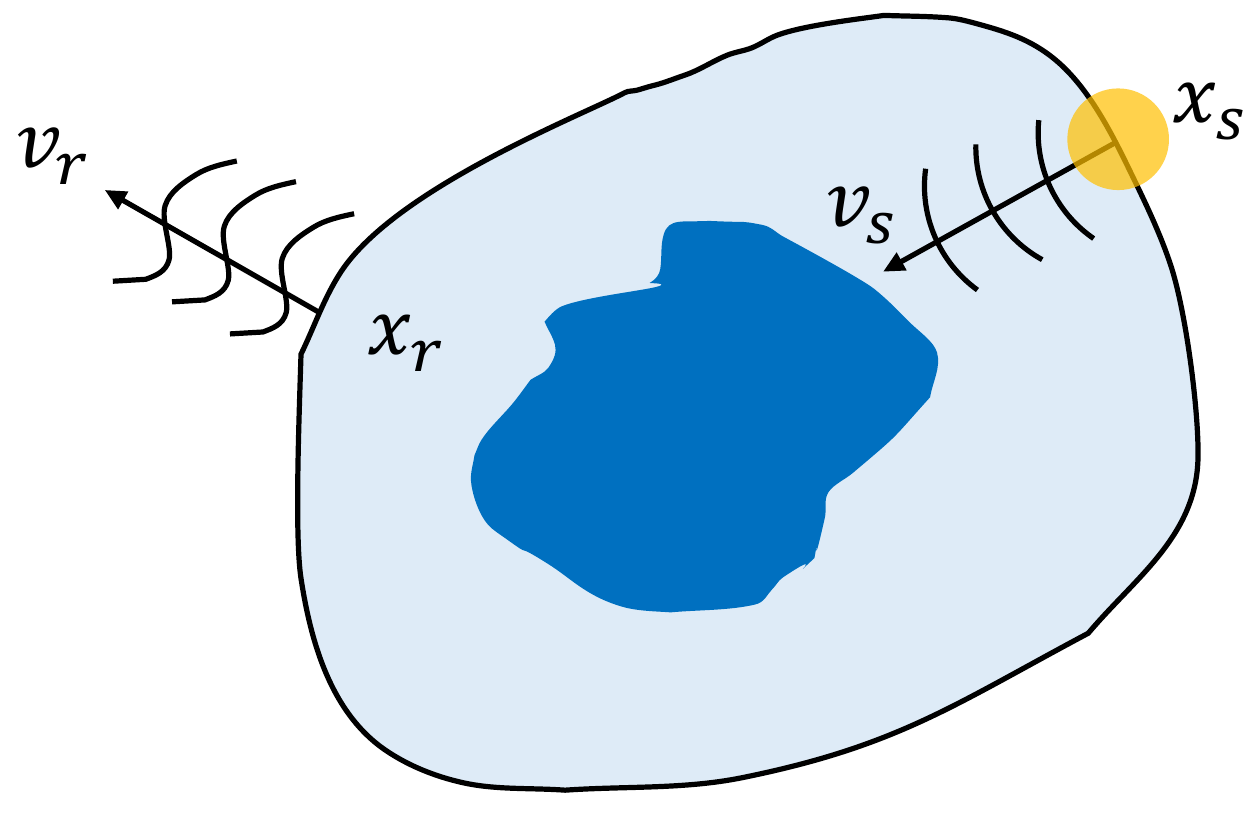}
  \caption{Illustration of the setup.}
  \label{fig:setup_new}
\end{figure}

In this section, we set up the experiment and provide the mathematical formulation, using both the wave and the particle forms for the forward model. This prepares us to link the two problems in Section~\ref{sec:limit}.
\subsection{Helmholtz equation and inverse wave scattering problem}\label{sec:inverse_wave}

The Helmholtz equation is a model equation for time-harmonic wave propagation. After some approximations, both the constant-density acoustic wave equation and the Maxwell equations for the EM waves, can be recast, through the Fourier transform in $t$, into the Helmholtz equation. It writes as:
\begin{equation}\label{eqn:Helmholtz_direct}
\Delta \uk + k^2 n(x) \uk = S^k(x)\,.
\end{equation}
In the equation, $\uk$ is the wave-field, with the superscript, $k>0$ represents the wave number (that carries the frequency information, and thus in the paper we use the two words interchangeably). $n(x)$ is a complex-valued refractive index having non-negative imaginary part, $\im(n(x))\geq 0$, reflecting the heterogeneity of the medium. We assume $n(x)$ is the constant one in all $\mathbb{R}^d$ except in a convex bounded open set $\Omega\subset\Rb^d$, meaning $\supp (n-1) \subset \Omega$. In order to streamline the notation, we let $\Omega = B_1$, the ball with radius $1$ centered around the origin. The right-hand side $S^k(x)$ is the source term, which is wave-number dependent.

The classical setup for the scattering problem is to probe the medium with an incident wave-field $u^{i,k}$ that triggers the response from the medium. Noting that the total field, which  satisfies~\eqref{eqn:Helmholtz_direct}, is the sum of the incident and the scattered wave-fields, we can write:
\[
\uk = u^{i,k} + u^{s,k}\,,
\]
and derive the equation for the scattered wave-field $u^{s,k}$. Suppose the incident wave is designed so that it absorbs all the external source information:
\begin{equation}\label{eqn:Helmholtz_incident}
\Delta \uik + k^2 \uik = S^k(x)\,,
\end{equation}
then by simply subtracting it from~\eqref{eqn:Helmholtz_direct}, we have the equation for $u^{s,k}$:
\begin{equation}\label{eqn:scattering}
\begin{aligned}
\Delta \usk + k^2 n(x) \usk &= k^2(1-n(x)) \uik \quad x \in \Rb^d \,, \\
\frac{\partial \usk}{\partial r} - \ri k \usk &= \Oc(r^{-(d+1)/2}) \text{ as } r=|x|\to\infty \,.
\end{aligned}
\end{equation}
In this equation, we can view the incident wave $\uik$ impinging in the perturbation $n-1$ as the source term for $\usk$. Clearly, this source term $k^2(1-n(x)) \uik$ is zero outside $B_1$, the support of $n-1$. The Sommerfeld radiation condition is imposed at infinity to ensure the uniqueness for $\usk$.

When $d=3$, a typical approach is to set $S(x) = \delta_y$, a point Dirac delta, then the solution $\uik$ to~\eqref{eqn:Helmholtz_incident} becomes the fundamental solution to the homogeneous Helmholtz equation in $\mathbb{R}^3$
\[
\Phi(x;y)=-\frac{1}{4\pi}\frac{\exp(ik|x-y|)}{|x-y|},\quad x,y\in\mathbb{R}^3, x\neq y\,,
\]
for any given $y$. We can clearly observe that the function is radially symmetric centered in $y$ thus it is often termed a spherical wave. If $|y| \gg |x|$, i.e., $y$ is far away from the origin, we have the far-field regime, in which the fundamental solution is approximately a plain wave: $\Phi(x;y)\approx -\frac{e^{ik|y|}}{4\pi|y|}\exp(-ik\hat{y}\cdot x)$.

In this case, however, instead of using the Dirac delta, we handcraft a specially designed source term, which will be crucial for the re-scaling proposed in this article. In particular, we choose $S^k_\rH(x)$ to be the following:
\begin{equation}\label{eqn:source}
S^k_\rH(x;x_s,{v}_s) = -k^{\frac{3+d}{2}} S_{{v}_s}(k(x-x_s)) \quad x \in \Rb^d \,,
\end{equation}
where the subscript $\rH$ stands for Helmholtz, and
\begin{equation}\label{eqn:S}
S_{{v}_s}(x) = C(\sigma,d)\exp\left(-\sigma^2\frac{|x|^2}{2}
+ \ri v_s\cdot x\right)\,.
\end{equation}
Here $C(\sigma,d)$ is the normalization constant $C(\sigma,d) = \sqrt{2} \left(\frac{\sigma}{\sqrt{\pi}}\right)^{\frac{d+1}{2}}$.

Physically this source term can be understood as the source generating a tight beam being shone onto the medium from the location $x_s$ in the direction of ${v}_s$. The profile of this tight beam, or ``laser beam'', is a Gaussian centered around the light-up location $x_s$ and the width of the Gaussian is characterized by $(k\sigma)^{-1}$. With $\sigma$ fixed, as $k\to\infty$, the beam is more and more concentrated.

Following the explanation above we incorporate the source term in~\eqref{eqn:source} into~\eqref{eqn:Helmholtz_direct}-\eqref{eqn:Helmholtz_incident}, to probe the medium from the positions, $x_s$, in the direction of $v_s$, that are physically pertinent. In particular, we let $(x_s,v_s)\in\Gamma_-$ where
\[
\Gamma_\pm = \{ (x,v)\in\partial B_1\times\Sb^{d-1}: \pm v \cdot \nu(x) > 0 \}\,.
\]
In which, $\nu(x)$ denotes the outer-normal direction at $x\in\partial B_1$. This means the laser beams shine from the boundary of $B_1$ in the direction $v$ pointing inward the interior of the domain.

From~\eqref{eqn:source} we can observe that as $k\to\infty$, the laser beam becomes increasingly concentrated. In particular, in the $k\to\infty$ limit, the incident wave $\uik$ becomes a ray, propagating through a straight line\footnote{The incoming ray propagates in a straight line due to the assumption that the background is constant. Otherwise, the ray would bend if a smooth non-constant background is considered.}.

As usual in inverse problems (in particular, in non-intrusive experimental setups), we take measurements of $\uk$ near the boundary $\partial B_1$. To take such measurement we design a family of test functions of the form:
\begin{equation}\label{eqn:phi_form}
\phi_v^k(x) = k^{d/4}\chi\left( \sqrt{k} x \right) e^{-\ri k v\cdot x} \,,
\end{equation}
where $\chi:\Rb^d\to\Rb$ is a smooth radially symmetric function that vanishes as $|x|\to\infty$.

We define the measurement of $\uk$ as its Husimi transform
\begin{equation}\label{eqn:Husimi}
\Hk \uk(x,v) = \left( \frac{k}{2\pi}\right)^d \left|\uk\ast\phi_{v}^k\right|^2\,\qquad \text{for } (x, v) \in \Gamma_{+}.
\end{equation}
The measurement then consists of the intensity of the field that convolves with the test function. This measurement is conducted only on the boundary, and only in the directions pointing outside the domain.

This measurement has a clear physical interpretation: it measures the intensity of the wave-field at location $x$ propagating in direction $v$, using $\chi$ as the impulse response of the receiver, or test function.

One typical choice for the family of test functions is to set $\chi$ as a  Gaussian (normalized in $L^2$ norm)
\begin{equation}\label{eqn:measure_profile}
\chi(x) = \left(\frac{1}{\pi}\right)^{d/4} \exp\left( -\frac{|x|^2}{2}\right)\,.
\end{equation}
It is straightforward to see that as $k\to\infty$, the test function $\phi^k_v$ concentrates around zero due to the $\sqrt{k}$ scaling. As such, the measurement $\uk\ast\phi^k_v$ at a location $x_s$ only takes value of $\uk$ in a very small neighborhood around $x_s$.

\begin{remark}
We note that the choice of $\chi$ in~\eqref{eqn:measure_profile} is not essential. We use this specific form to make the calculation explicit, as it will be shown in~\cref{prop:limit}. Other forms of $\chi$ would also work well as long as the corresponding $G^k=W^k[\phi_0^k]$ converges to a Dirac delta when $k\to\infty$, as it will be explained in~\cref{rmk:chi_convergence}.
\end{remark}

\medskip

\noindent \textbf{Forward Map:} now we have all the elements to define the forward map. For any $(x_s,v_s)\in\Gamma_-$, we shine laser beam into ${B_1}$ according to the format in~\eqref{eqn:source}, then the solution to the Helmholtz equation~\eqref{eqn:Helmholtz_direct}, $\uk$ is tested by $\phi_v^k(x)$ and evaluated on $\Gamma_+$:
\begin{equation}\label{eqn:forwardmap}
\Lambda^k_n:\quad S^k_\rH(x;x_s,v_s) \to H^k \uk(x_r,v_r)|_{\Gamma_+}\,.
\end{equation}
As a consequence, the dataset generated by this forward map is the collection of:
\begin{equation}\label{eqn:dataset}
\Dc^k[n] = \left\{\left(S^k_\rH(x;x_s,v_s),\Lambda_n^k[S^k_\rH](x_r, v_r)\right):\, (x_s,v_s)\in\Gamma_-, (x_r,v_r)\in\Gamma_+\right\} \,.
\end{equation}
We now formulate the generalized inverse scattering problem as
\begin{equation}\label{eqn:inverse_s_near1}
    \emph{to reconstruct $n$ using the information in $\Dc^k[n]$.}
\end{equation}

\subsubsection{Traditional inverse scattering problem}\label{rmk:classical} Given that we use a non-standard formulation of the inverse scattering problem, we will stress a couple of similarities and differences between the generalized and classical inverse scattering problems.

In particular, the form of the forward map introduced in our setting differs from the classical one, where the incident wave is typically a plane wave, meaning $u^{i,k}(x;v_s)=\exp(ikv_s\cdot x)$, where $v_s\in\mathcal{S}^{d-1}$, see~\cite{Ki:2011introduction}.

So the forward map is given by the far field map, $\widetilde{\Lambda}^k_n$:
\[
\widetilde{\Lambda}^k_n: u^{i,k}(x;v_s)\rightarrow u^{\infty,k}(\hat{x};v_s)\,,
\]
where $u^{\infty,k}:\mathcal{S}^{d-1}\rightarrow\mathbb{C}$ is defined as
\[
u^{\infty,k}(\hat{x};v_s)=\lim_{r\rightarrow\infty} ru^{s,k}(r\hat{x};v_s)\exp(-ikr)|_{\hat{x}\in\mathcal{S}^{k-1}},\quad \forall \hat{x}\in\mathcal{S}^{d-1}\,,
\]
with $u^{s,k}$ being the solution of \eqref{eqn:scattering}, where we leverage that $u^{i,k}(x;v_s)$ satisfies \eqref{eqn:Helmholtz_incident} with $S = 0$. Therefore in this setting, the data set induced by the forward map is defined as:
\[
\widetilde{\Dc}^k[n] = \left\{\left(u^{i,k}(x;v_s),\widetilde{\Lambda}^k_n\left[u^{i,k}\right](\hat{x})\right):\, v_s\in\mathcal{S}^{d-1}, \hat{x}\in\mathcal{S}^{d-1}\right\} \,.
\]
The well-posedness and stability of the inverse scattering problem in this context has been studied in \cite[Theorem 1.2]{HaHo:2001new}.

The differences from the classical inverse scattering formulation is two fold: i) we use a richer set of probing functions, instead of using incident waves that are directionally localized (as plane waves) or whose sources are localized (as Green's functions), we use tight beams that combine these two properties, and ii) instead of measuring the scattered wave-field on a manifold around the domain of interest, we multiply it  with a set of directional filters localized on the same manifold, and we compute its intensity. We should emphasize that this difference is significant. Take the plane-wave as the probing wave, as an example, it is only the direction of the incoming wave that can be tuned, and this composes $2$ dimensions of degrees of freedom in 3D with $v_s\in \mathcal{S}^{d-1}$. The way our source term is designed automatically carries $4$ dimensions of degrees of freedom with $(x_s,v_s)\in\Gamma_-$. Similarly, the way data gets taken also expands the degrees of freedom the measuring operator can access. It is a widely accepted fact that more data leads to more stable reconstruction. This will be indeed demonstrated in the later sections.

\begin{remark}
We note that even though the conventional inverse scattering problem has been shown to be ill-conditioned, a couple of strategies have been introduced in the literature to stabilize the problem. The most prominent strategy is to add the phase information (microlocally)~\cite{BaLeRa:1992sharp,StUhVaZh:2019travel,ChQiUhZh:2007new}. At the first look, the Husimi data~\eqref{eqn:Husimi} also extracts the phase information, by integrating the scattered wave with an oscillatory test function~\eqref{eqn:phi_form} that is localized in position and direction. In very simple cases, we can even show that the two sets of information is equivalent. For example, suppose the wave field is of the simple form of $u^k(x) = A(x)e^{\ri k p\cdot v}$ with $p\in\Sb^{d-1}$ and $A(x)\geq 0$, for all $x\in\Rb^d$. Then in the limit $k\to\infty$, we can fully recover $u^k(x)$, both the amplitude and the phase, on the boundary $\partial B_1$ using the Husimi data~\eqref{eqn:Husimi}
\[
\lim_{k\to\infty} H^k u^k(x,v) = |A(x)|^2 \delta(v-p)\,, \quad \forall (x,v)\in\partial B_1 \times \Sb^{d-1} \,.
\]
However, in general cases, we are not aware of results that translate Husimi data to the phase data. Indeed, according to~\cite{GrKoRa:2020phase,GrRa:2019stable,AlDaGrYi:2019stable}, this might be a very complicated phase retrieval problem that is beyond the scope of the current paper.
\end{remark}

\begin{remark}
Another strategy to stabilize the inverse scattering problem is to transform the Helmholtz equation back to the time-domain, and solve the inverse acoustic wave problem, with either full or partial data available for all time $T\geq0$. In various settings~\cite{BaZh:2014sensitivity,Is:2017inverse,Su:1990continuous,YuOlYa:2001global}, it is proved that the time-domain data is sufficient to reconstruct the medium. The wave equation and Helmholtz equation are Fourier transform of each other in time. Roughly speaking, the temporal data collected on the boundary translates to the boundary information for all frequency $k$. As such, the temporal data has wide-band information instead of being monochromatic, and thus is expected to be more stable. In our setting, though we require $k\gg 1$, we still use monochromatic information, and thus the data does not directly translate.
\\
\indent We should note, however, that though the time-domain data is expected to be more informative in theory, in practice, however, especially within the optimization-based reconstruction algorithm framework, the typical $\ell^2$ misfit loss function results in an extremely non-linear problem that often leads to cycle-skipping, and convergence to spurious, non-physical, local minima. The numerical problem is usually attenuated by using the time/frequency duality and localizing the frequency content of the data, which is then processed in a hierarchical fashion \cite{Ch:1997inverse,Pr:1999seismic}. These are beyond the focus of the paper.
\end{remark}

\subsection{High-frequency limit and inverse Liouville scattering problem}\label{sec:inverse_liouville}
The Liouville equation is a well studied classical model for describing particle propagation. Any system with a large number of identical particles can be described by the Liouville equation, or its variants, which is often written as:
\begin{equation}\label{eqn:liouville}
    v\cdot\nabla_xf + \frac{1}{2}\nabla_x n\cdot\nabla_vf = S_\rL(x,v)\,,
\end{equation}
where $f(x,v)$ characterizes the number of particles on the phase space $(x,v)$. Following the characteristics, we see that the particles follow Newton's second law:
\[
\dot{x} = v\,,\quad\dot{v} = \frac{1}{2}\nabla_x n\,.
\]
As usual in classical mechanics, we can define the Hamiltonian for each particle to be
\[
H(x(t),v(t)) = 2|v(t)|^2 - n(x(t))\,,
\]
which is preserved along the characteristics of the particles, i.e., $\frac{\rd H}{\rd t} = 0$.

We use \eqref{eqn:liouville} to describe photon propagation, and use the same setup as that in Section~\ref{sec:inverse_wave}. The source term $S_\rL(x,v)$ on the right-hand side of~\eqref{eqn:liouville} describes how laser beams are shone into the medium, and takes the form of:
\begin{equation}\label{eqn:source_l}
S_\rL(x,v;x_s,v_s) = \phi(x-x_s)\psi(v-v_s)\,,\quad\text{with}\quad (x_s,v_s)\in\Gamma_-\,,
\end{equation}
where both $\phi:\mathbb{R}^d\to\mathbb{R}$ and $\psi:\mathbb{R}^d\to\mathbb{R}$ are radially symmetric smooth functions that concentrate at the origin. By setting $(x_s,v_s)\in\Gamma_-$, we have the laser beam shining from the boundary $\partial{B_1}$ inward to the domain. The concentration of the beam is determined by $\phi$ and $\psi$ in physical- and velocity-space respectively.

Similar to the previous section, we take the measurements of the light intensity at the boundary pointing outside of the domain. To do so, we set the test function $\zeta(x,v)$ and the measurements would be its convolution with the solution to~\eqref{eqn:liouville}:
\begin{equation}\label{eqn:defL}
Lf(x,v)=f\ast\zeta(x,v)\,.
\end{equation}
The physical setup is clear. Imaging $\zeta$ a blob centers around $(x,v)=(0,0)$, then $Lf(x_r,v_r)$ essentially represents a measuring equipment that takes in light intensity concentrated around $(x_r,v_r)$ with the concentration determined by the size of the blob. The specific format of $\zeta$ will be specified in Section~\ref{sec:limit}.

\noindent \textbf{Forward Map:} we define the forward map in a similar fashion as in Section \ref{sec:inverse_wave}. For any $(x_s,v_s)\in\Gamma_-$, we solve~\eqref{eqn:liouville} with $S_\rL$ defined in \eqref{eqn:source_l}, and test the solution on $\zeta(x,v)$ evaluated on $\Gamma_+$:
\[
\Lambda_n:\quad S_\rL(x,v;x_s,v_s) \to Lf(x_r,v_r)|_{\Gamma_+}\,.
\]
As a consequence, the dataset generated by this forward map is the collection of:
\begin{equation}\label{eqn:ldataset}
\Dc[n] = \left\{\left(S_\rL(x,v;x_s,v_s),\Lambda_n[S_\rL](x_r, v_r)\right):\, (x_s,v_s)\in\Gamma_-, (x_r,v_r)\in\Gamma_+\right\} \,.
\end{equation}
While the forward problem is to compute and construct this $\Dc[n]$ for any given $n$, the inverse problem amounts to inferring $n$ using the information in $\Dc[n]$.

\section{Relation between the two problems in the high-frequency regime}\label{sec:limit}

In this section we discuss the connection between the forward maps for the wave- and particle-like descriptions introduced in the section above. We start introducing the Wigner transform, and we use it to present the equivalence of the two descriptions for the forward maps in the high-frequency regime. Then we introduce the Husimi transform to take the limit of the measuring operator, and this is used to show the equivalence of the two inverse problems. Finally, we briefly introduce the stability of the inverse Liouville problem.

\subsection{High-frequency limit of the forward problem}\label{sec:limit_f}
We first present their connection in the forward setting. We discuss the  derivation of the Liouville equation as the limiting equation for the Helmholtz. This process is typically called taking the ``classical"-limit, to reflect the passage from quantum mechanics to classical mechanics by linking the Schr\"odinger equation to the Liouville equation in the small $\hbar$ regime.

Among the multiple techniques to derive the classical limit we utilize the Wigner transform~\cite{GeMaMaPo:1997homogenization,RyPaKe:1996transport,BaPaRy:2002radiative,ChLiYa:2021classical}. Compared to other techniques, such as WKB expansion~\cite{EnRu:2003computational} and Gaussian beam expansion~\cite{TaQiRa:2007mountain,LuYa:2010frozen,QiYi:2010gaussian}
, Wigner transform presents the equation on the phase space, and avoids the emerging singularities during the evolution. Let $\uk_1$ and $\uk_2$ be two functions, then the corresponding Wigner transform is defined as
\begin{equation}\label{eqn:Wigner_transform}
\wk[\uk_1,\uk_2](x,v) = \frac{1}{(2\pi)^d} \int_{\Rb^d} e^{\ri v\cdot y} \uk_1 \left(x - \frac{y}{2k}\right) \overline{\uk_2}\left(x +  \frac{y}{2k} \right)\rmd y\,.
\end{equation}
Here $\overline{\uk_2}$ is the complex conjugate of $\uk_2$. We furthermore abbreviate $\wk[\uk_1,\uk_2]$ to be $\wk[\uk]$.

The Wigner transform $\wk[\uk]$ is defined on the phase space, is always real-valued, and the moments in $v$ of $\wk[\uk]$ carry interesting physical meanings. In particular, the first moment
recovers the energy density $\Ec^k$:
\begin{equation}\label{eqn:energy_density}
\Ec^k(x) = \int_{\Rb^d} \wk[\uk] (x,v) \rmd v=\left|\uk(x)\right|^2\,,
\end{equation}
and its second moment expresses the energy flux $\Fc^k$:
\begin{equation}\label{eqn:energy_flux}
\Fc^k(x) = \int_{\Rb^d} v \wk[\uk] (x,v) \rmd v= \frac{1}{k} \im \left( \overline{\uk(x)} \nabla_x \uk(x) \right)\,.
\end{equation}

Most importantly, if $\uk$ solves the Helmholtz equation~\eqref{eqn:Helmholtz_direct}, one can show that $\wk[\uk]$ solves an equation in the form of the radiative transfer equation, and in the $k\to\infty$ limit, this degenerates to the Liouville equation~\eqref{eqn:liouville}. In what follows we seek to make this statement more precise by defining the functional space and an appropriate metric.

Let $\lambda>0$, we define $X_\lambda$ a space that contains all scalar real valued functions defined on the phase-space $\mathbb{R}^3\times\mathbb{R}^3$:
\begin{equation}\label{eqn:Xlambda}
X_\lambda=\left\{\phi(x,y)\,\, \middle| \, \,\int_{\mathbb{R}^3}\sup_{x\in\mathbb{R}^3}(1+|x|+|\xi|)^{1+\lambda}|\hat{\phi}(x,\xi)|\rmd \xi<\infty\right\}\,,
\end{equation}
with associated norm given by
\[
\|\phi\|_{X_\lambda}= \int_{\mathbb{R}^3}\sup_{x\in\mathbb{R}^3}(1+|x|+|\xi|)^{1+\lambda}|\hat{\phi}(x,\xi)|\rmd \xi\,,
\]
where $\hat{\phi}(x,\xi)=\frac{1}{(2\pi)^d} \int_{\Rb^d} \phi(x,y)e^{-i\xi\cdot y}\rmd y$ is the Fourier transform in velocity-space. Now we cite a result from ~\cite[Theorem 3.11, 3.12]{BeCaKaPe:2002high}.
\begin{theorem}\label{thm:formal}
Let $n(x)$ be a $C^2(\mathbb{R}^d;\mathbb{R}_+)$  function that satisfies certain conditions (see Remark~\ref{rmk:condition_n}). Let $u^k$ be the solution to~\eqref{eqn:Helmholtz_direct} with radiation conditions, where the source term $S^k_\rH$ is defined in~\eqref{eqn:source}. Then the Wigner transform of $\uk$, denoted by $\fk(x,v) = \wk[\uk](x,v)$ solves
\begin{equation}\label{eqn:wigner_eqn}
\begin{aligned}
v\cdot\nabla_x \fk + \frac{1}{2}\Lc_n^k[\fk] = -\frac{1}{k} \im \left( \wk[u^k,S^k] \right)  \,,\quad (x,v)\in\Rb^{2d}\,,
\end{aligned}
\end{equation}
with the operator $\Lc_n^k$ defined as
\begin{equation}\label{eqn:wigner_Lc}
\Lc_n^k[\fk] := \frac{\ri}{(2\pi)^d} \int_{\Rb^{2d}} \delta^k [n](x,y) \fk(x,p) e^{\ri y(v-p)}  \, \rmd y \, \rmd p\,.
\end{equation}
Here $\delta^k [n](x,y) = k\left[ n\left(x+\frac{y}{2k}\right) - n\left(x-\frac{y}{2k}\right) \right]$. Furthermore, when $k\rightarrow\infty$, $\fk$ converges in weak-$\star$ sense to $f(x,v)$ in $\left(X_\lambda\right)^\star$, the solution to the Liouville equation~\eqref{eqn:liouville} with the radiation condition $\lim_{|x|\rightarrow\infty}f(x,v)=0$ for all $x\cdot v<0$, and the source $S_\rL(x,v)$ is:
\begin{equation}\label{eqn:source_limit}
S_\rL(x,v) = (2\pi)^d \frac{\pi}{2}\delta(x-x_s) |\hat{S}_{v_s}(v)|^2\delta\left(|v|^2=n(x_s)\right) \,.
\end{equation}
Here $\hat{S}_{v_s}$ denotes the Fourier transform, and the delta function $\delta\left(|v|^2=n(x_s)\right) \in \mathcal{D}'(\mathbb{R}^d)$ means
\[
\langle \delta\left(|v|^2=n(x_s)\right) \,,g\rangle =\int_{|v|^2=n(x_s)} g(v)dS_v,\quad \forall g\in \mathcal{S}(\mathbb{R}^d)\,.
\]
\end{theorem}
Suppose $S_v$ takes the form of~\eqref{eqn:source}, we can explicitly calculate its Fourier transform:
\[
|\hat{S}_{v_s}(v)|^2 = C(\sigma,d)^2 \frac{1}{(2\pi)^d\sigma^{2d}} e^{-\frac{|v-{v_s}|^2}{\sigma^2}} \,.
\]

\begin{remark}\label{rmk:condition_n}
The formal derivation of the limit is shown in Appendix~\ref{appendix:formalderivation}. To prove it rigorously, we refer to~\cite[Theorem 3.11, 3.12]{BeCaKaPe:2002high} and~\cite{CaPeRu:2002high}. The conditions for a rigorous proof are rather complicated to obtain. However, we mention that if $n$ is radially symmetric, i.e., $n(x) = n(|x|)$, the statement of the theorem holds true rigorously.

\end{remark}

Theorem~\ref{thm:formal} suggests that the wave model and the particle model are asymptotically equivalent in the high-frequency regime. According to~\eqref{eqn:source_limit}, the source term concentrates at $(x_s,v_s)$, the source location and the source velocity, when $k\to\infty$. The concentration on $x$ is already achieved by taking to limit as $k\to\infty$, but the concentration profile in $v$ still needs to be tuned by $\sigma$. Smaller $\sigma$ results in a more concentrated source in this limiting regime. Let $\sigma\to 0$, we have the source term $S_\rL$ turning into:
\begin{equation}\label{eqn:limit_sigma}
\begin{aligned}
(2\pi)^d\frac{\pi}{2} \delta(x-x_s) |\hat{S}_{v_s}(v)|^2\delta(|v|=1)
&= \delta(x-x_s)\left(\frac{1}{\sigma\sqrt{\pi}}\right)^{d-1} e^{-\frac{|v-{v_s}|^2}{\sigma^2}} \delta(|v|=1)\\
&\rightarrow \delta(x-x_s)\delta(v-{v_s})\,,
\end{aligned}
\end{equation}
where we used $n(x_s) = 1$, given that $x_s$ is out of the domain interest $B_1$.

In this specific limit, we have the explicit solution to the Liouville equation~\eqref{eqn:liouville}:
\begin{equation}\label{eqn:limitoff}
f(x,v) = \delta_{(x(s;(x_s,{v_s})),v(s;(x_s,{v_s})))}\,, \quad {k\to\infty} \,,
\end{equation}
where $(x(s;(x_s,{v_s})),v(s;(x_s,{v_s})))$ are the location and velocity of a particle at time $s$ that starts off at $(x_s,v_s)$, meaning $(x(0;(x_s,{v_s})),v(0;(x_s,{v_s})))=(x_s,v_s)$ and
\begin{equation}\label{eqn:chareqn}
\left\{
\begin{aligned}
&\frac{\rmd x(s;(x_s,{v_s}))}{\rmd s}=v(s;(x_s,{v_s}))\,,\\
&\frac{\rmd v(s;(x_s,{v_s}))}{\rmd s}=\frac{1}{2}\nabla_xn(x(s;(x_s,{v_s})))\,.
\end{aligned}
\right.
\end{equation}
The formulation in~\eqref{eqn:limitoff} means in this limit, with $k\to\infty$ and $\sigma\ll 1$, the wave becomes a curved ray that follows the trajectory of the particle that is governed by Newton's laws. As a consequence, recall the definition of energy and energy flux in~\eqref{eqn:energy_density}-\eqref{eqn:energy_flux}:
\[
\lim_{\sigma\to 0}\lim_{k\to\infty}\Ec^k(x)={\bf 1}_{s>0}\delta_{x(s;(x_s,{v_s}))},\quad \lim_{\sigma\to0}\lim_{k\to\infty}\Fc^k(x)={\bf 1}_{s>0}\delta_{x(s;(x_s,{v_s}))}v(s;(x_s,{v_s}))\,,
\]
suggesting that $\Ec^k$ and $\Fc^k$ respectively show approximately the location and velocity of the trajectory.

\subsection{High-frequency limit of the inverse problem}\label{sec:limit_i}
In the prequel we linked the two forward problems. We now proceed to connect the two inverse problems, by evaluating the convergence of the measurements. To do so, we first introduce Lemma~\ref{lem:connection} from \cite[Section 2.5]{CoRo:2012coherent} that connects the Husimi and Wigner transforms.

\begin{lemma}\label{lem:connection}
Assume $u\in L^2(\mathbb{R}^d;\mathbb{R})$ , and let $\Hk u$ be the Husimi transform defined in~\eqref{eqn:Husimi} with $\phi^k_v$ being the test function (defined in~\eqref{eqn:phi_form}). Denote $f^k = \wk[u]$, and $\Gk=\wk[\phi_0^k]$, the Wigner transform of $\uk$ and $\phi_{0}^k$ respectively. Here $\phi_0^k = \phi^k_{v=0}$. Then
\begin{equation}
\Hk u (x,v) = \fk \ast \Gk (x,v)\,, \quad \forall (x,v)\in\Rb^{2d} \,.
\end{equation}
\end{lemma}
\begin{proof}
This theorem is a directly result of the Moyal identity
\begin{equation}\label{eqn:p1}
(\Wk[h_1],\Wk[h_2])_{L^2(\Rb^{2d})} = \left( \frac{k}{2\pi} \right)^d |(h_1,h_2)_{L^2(\Rb^d)}|^2 \,,\quad \forall h_1,h_2\in L^2(\mathbb{R}^d;\mathbb{R})\,,
\end{equation}
and the fact that
\begin{equation}\label{eqn:p2}
\Wk[\phi_{v}^k(x-\cdot)](y,p) = \Wk[\phi_{0}^k](x-y,v-p) \,.
\end{equation}

Using \eqref{eqn:Husimi}, we have
\[ \begin{aligned}
\Hk u(x,v) &= \left( \frac{k}{2\pi}\right)^d \left|u\ast\phi_{v}^k\right|^2\\
&=\left( \frac{k}{2\pi}\right)^d \left| \left(u(\cdot),\phi_{v}^k(x-\cdot)\right)_{L^2(\Rb^d)}\right|^2\\
&=
\left(W^k[u],W^k[\phi_{v}^k(x-\cdot)]\right)_{L^2(\Rb^{2d})}\\
&=
\left(W^k[u],W^k[\phi_{0}^k](x-\cdot,v-\cdot)\right)_{L^2(\Rb^{2d})}\\
&=f^k* G^k\,,
\end{aligned} \]
where we use \eqref{eqn:p1} in the third equality, \eqref{eqn:p2} in the fourth equality, and the definitions of $f^k$ and $G^k$ in the last equality.
\end{proof}

This lemma connects the measurement of $\uk$ with the measurement on the phase space. Testing $\uk$ using the test function $\phi_0^k$ is translated to testing $f^k$ using the test function $\Gk$. This allows us to pass to the limit on the phase space. Combining with Theorem \ref{thm:formal}, we have the following proposition:

\begin{proposition}\label{prop:limit}
Let the assumption in Theorem \ref{thm:formal} hold true. Denote $f^k=\wk[\uk]$, with $\uk$ solving the Helmholtz equation~\eqref{eqn:Helmholtz_direct} with the source term $S_{\rH}$ defined in~\eqref{eqn:source}, and denote $f$ the solution to the Liouville equation~\eqref{eqn:liouville} with source term $S_\rL$  defined in~\eqref{eqn:source_limit}. If $\chi$ takes the form of~\eqref{eqn:measure_profile}, so that $\Gk$ takes the form of:
\begin{equation}\label{eqn:Gk}
\Gk(x,v) = \left( \frac{k}{\pi} \right)^d \exp\left( -k\left(|x|^2+|v|^2\right) \right) \,,
\end{equation}
as $k\rightarrow\infty$, we have:
\[
\fk \ast \Gk (x,v)\rightarrow f(x,v)
\]
weak-$\star$ in $\left(X_\lambda\right)^\star$.
\end{proposition}
\begin{proof}
Given the form of $G^k$ in~\eqref{eqn:Gk}, for any $\phi\in X_\lambda$, as $k\to\infty$:
\[
G^k*\phi(x,v)\longrightarrow \phi(x,v)\quad \text{in}\quad X_\lambda\,.
\]
Thus,
\[
\begin{aligned}
\lim_{k\rightarrow\infty}\int_{\mathbb{R}^3\times\mathbb{R}^3}\left(\fk \ast \Gk (x,v)\right)\phi(x,v)\rmd x\, \rmd v
=&\lim_{k\rightarrow\infty}\int_{\mathbb{R}^3\times\mathbb{R}^3}\fk(x,v)\left(\Gk\ast\phi(x,v)\right)\rmd x\, \rmd v\\
=&\lim_{k\rightarrow\infty}\int_{\mathbb{R}^3\times\mathbb{R}^3}f^k(x,v)\phi(x,v)\rmd x\, \rmd v\\
=&\int_{\mathbb{R}^3\times\mathbb{R}^3}f(x,v)\phi(x,v)\rmd x\, \rmd v\,,
\end{aligned}
\]
where we use $\|f^k\|_{(X_\lambda)^*}$ being bounded in the second equality, and $f^k\rightarrow f$ in the weak-$\star$ sense in the last equality.
\end{proof}
\begin{remark}\label{rmk:chi_convergence} We note that the statement of the proposition indeed uses the explicit form of $\chi$ as defined in~\eqref{eqn:measure_profile}, but the use only lies in the fact that $G^k*\phi(x,v)\longrightarrow \phi(x,v)$ in the high frequency limit. Other forms of $\chi$ works equally well as long as this $G^k$ serves as a delta measure when $k\to\infty$.
\end{remark}

\begin{theorem}
Let the assumptions in Theorem~\ref{thm:formal} and  Lemma~\ref{lem:connection} hold true, then:
\[
\lim_{k\rightarrow\infty}\Hk \uk (x,v)=\lim_{k\to\infty}\fk \ast \Gk (x,v)\xrightarrow[\text{weak}-\star]{} f(x,v),
\]
in $\left(X_\lambda\right)^*$. Furthermore, if $H^ku^k$ and $f$ are continuous, then each element in $\Dc^k[n]$ has a limit in $\Dc[n]$. More specifically:
\begin{equation}\label{eqn:inverseconvergence}
(S^k_{\rH}(x;x_s,v_s)\,,\Lambda^k_n[S^k_\rH](x_r,v_r))\to (S_{\rL}(x,v;x_s,v_s)\,, \Lambda_n[S_{\rL}](x_r,v_r))
\end{equation}
where $S_{\rL}$ takes the form of~\eqref{eqn:source_limit}, and $\Lambda_n[S_{\rL}](x_r,v_r)=f(x_r,v_r)$. In particular, if $\sigma\rightarrow0$,
\begin{equation}\label{eqn:measure_outgoing_conv}
\Lambda_n[S_\rL](x_r, v_r)=f*\delta_{(\vec{0},\vec{0})}|_{\Gamma_+}=f(x_r,v_r)|_{\Gamma_+} = \delta(x-x_{r_s})\delta(v-v_{r_s})\,,
\end{equation}
with $(x_{r_s},{v}_{r_s})$ being the outgoing location and velocity when the photon particle leaves the domain, namely:
\begin{equation}\label{eqn:Liouville_out}
x_{r_s}=x(S;(x_{s},{v}_{s})),\quad {v}_{r_s}=v(S;(x_{s},{v}_{s}))\,,
\end{equation}
where $S=\sup_{s\geq 0}\left\{s\middle|x(s;(x_{s},{v}_{s}))\in {B_1}\right\}$ and $\left\{x(s;(x_{s},{v}_{s}),v(s;(x_{s},{v}_{s}))\right\}$ solves~\eqref{eqn:chareqn}.
\end{theorem}
This theorem naturally links the two inverse problems. In the $k\to\infty$ limit, the two datasets \eqref{eqn:dataset},\eqref{eqn:ldataset} are asymptotically close with $\zeta=\delta_{(\vec{0},\vec{0})}(x,v)$ in \eqref{eqn:defL}. In the limit of $k\rightarrow\infty$ and $\sigma\rightarrow0$, the dataset \eqref{eqn:dataset} is asymptotically approximately equivalent to
\begin{equation}\label{eqn:limitdataset}
\Dc^\infty[n] = \left\{((x_s,v_s),(x_r, v_r)):\,(x_s,v_s)\in\Gamma_-,\ (x_r,v_r)\,\text{from \eqref{eqn:Liouville_out}}\right\} \,.
\end{equation}

\subsection{Stability of Liouville inverse problem}\label{sec:stabilityLiouville}
In this section, we consider the stability of Liouville inverse problem. In particular, we focus on the stability of \eqref{eqn:limitdataset}. We will show that when $n$ is close enough to $1$, $D^\infty_n$ almost contains the information of the $X$-ray transforms of $n(x)$ and $\nabla_xn(x)$, while the inverse of $X$-ray transform is a well-posed inverse problem.

We first introduce the $X$-ray transform. Define
\[
TS^{d-1}=\left\{(x,{v})\,\middle| \, x\in\mathbb{R}^d,\, {v}\in\mathcal{S}^{d-1}, \,\left\langle{v},x\right\rangle=0\right\}\,.
\]
Assuming that $n(x)$ is continuous, we introduce the $X$-ray transform $P$, which maps $n(x),\nabla_xn(x)$ into functions $Pn\in C(TS^{d-1},\mathbb{R})$ and $P(\nabla_xn)\in C(TS^{d-1},\mathbb{R}^d)$, such that
\[
Pn({v},x)=\int^\infty_{-\infty}n(t{v}+x)\rmd t,\quad P(\nabla_xn)({v},x)=\int^\infty_{-\infty}\nabla_xn(t{v}+x)\rmd t.\quad
\]
To connect $D^\infty_n$ with $X$-ray transform, we define a projection map $\mathcal{P}:\partial B_1\times \mathcal{S}^{d-1}\rightarrow\mathbb{R}^d\times \mathcal{S}^{d-1}$
\[
\mathcal{P}((x,{v}))=\left(x-\left\langle x,{v}\right\rangle{v},{v}\right)
\]
that projects $x$ to the plane with normal vector ${v}$. We also define in-out map $\mathcal{L}:\Gamma_-\rightarrow\Gamma_+$ corresponding to \eqref{eqn:Liouville_out}:
\[
\mathcal{L}((x_{s},{v}_{s}))=(x_{r},{v}_{r})\,.
\]

\begin{remark}
We remark that the in-out map may not be well-defined for arbitrarily given $n$.  Suppose $n(x)\geq c_0$ for all $x\in\mathbb{R}^d$ and some $c_0>0$, then according to the conservation of Hamiltonian
\begin{equation}\label{eqn:conservH}
    H(x,v) = \frac{1}{2}|v|^2 - \frac{1}{2} n(x) = \frac{1}{2} - \frac{1}{2} = 0 \,,
\end{equation}
the velocity of the particle satisfies
\[
|v(s)| = \sqrt{n(x(s))} \geq \sqrt{c_0} >0 \,,
\]
for all time $s\geq0$. This by no means suggests the non-trapping property, but it at least ensures that the potential is not a sink. In the general case, we do assume that $n$ is non-trapping, so that any incoming particle can eventually be expelled out of the domain again, making the map $\mathcal{L}$ well-defined. Such non-trapping condition is closely related to geodesic X-ray transforms, and we list references~\cite{StUhVaZh:2019travel,ChQiUhZh:2007new,MonardStefanovUhlmann_geodesic_ray} for interested readers. In our numerical examples, we choose the media to be locally repulsive in the sense that
\begin{equation}\label{eqn:loc_repul}
    n(x) + x \cdot \nabla n(x) \geq c_1 >0 \,, \quad \forall x\in \Rb^d \,.
\end{equation}
Let $(x(s), v(s))$ be any particle trajectory that solves~\eqref{eqn:chareqn}. Given~\eqref{eqn:loc_repul}, we obtain the inequality
\begin{equation}\label{eqn:mom_ine}
    \frac{\rmd^2 }{\rmd s^2} \left( \frac{1}{2}|x(s)|^2 \right) = |v(s)|^2 + x(s) \cdot \frac{\rmd v}{\rmd s} = n(x(s)) + x(s) \cdot \nabla n(x(s)) \geq c_1 >0 \,.
\end{equation}
In the last equality, we have used~\eqref{eqn:conservH}. By making use of~\eqref{eqn:mom_ine}, the particle is non-trapped since $|x(s)| \geq s \sqrt{\frac{1}{2}c_0}$ for sufficiently large $s>0$.
\end{remark}

We note that $\mathcal{P}((x,{v}))\in TS^{d-1}$ for any $(x,{v})\in \partial B_1\times \mathcal{S}^{d-1}$, and $\mathcal{P}|_{\Gamma_-}:\Gamma_-\rightarrow\mathbb{R}^d\times \mathcal{S}^{d-1} , \mathcal{P}|_{\Gamma_+}:\Gamma_+\rightarrow\mathbb{R}^d\times \mathcal{S}^{d-1}$ are invertible. Now, we are ready to introduce the following approximation theorem~\cite[Theorem 4.1]{No:1999small}:
\begin{theorem}\label{thm:liuinverse} Assume
\[
\|\nabla n(x)\|_{L^\infty}\leq \Delta,\quad \|\|H n(x)\|_F\|_{L^\infty}\leq \Delta\,
\]
for some $\Delta>0$, then for any $({v},x)\in TS^{d-1}$, we have
\[
\left|\left(Pn({v},x),P(\nabla_xn)({v},x)\right)-\mathcal{P}|_{\Gamma_+}\circ\mathcal{L}\circ \left(\mathcal{P}|_{\Gamma_-}\right)^{-1}({v},x)\right|\leq C\Delta^2\,,
\]
where $C>0$ is a constant only depends on $d$.
\end{theorem}
According to Theorem \ref{thm:liuinverse}, if $n$
is almost a constant (close enough to $1$), then we can use the data set to recover $X$-ray transform of $n,\nabla n(x)$. Thus, we can separate \eqref{eqn:limitdataset} into two inverse problems
\[
\mathcal{D}^\infty[n]\Longrightarrow \left(Pn({v},x),P(\nabla_xn)({v},x)\right)\Longrightarrow n(x)\,,
\]
where the first one can be approximately calculated if $n$
is almost constant $1$ and the second one is the inverse of $X$-ray transform that is well-posed according to~\cite[Theorem 5.1]{Na:2001mathematics}.

\section{Numerical experiments}\label{sec:numer}

In this section we provide numerical evidence showcasing the theory developed above. In particular, we would like to demonstrate that as $k$ increases, the measurement taken on the solution to the Helmholtz equation through the Husimi transform converges to the pointwise evaluation of the solution to the Liouville equation, and that the data becomes more and more sensitive to the perturbation in media, making the inverse problem more and more stable.

We first summarize the numerical setup and unify the notations, and then present a class of numerical results.

\subsection{Numerical setup}
We set up our experiment in a two dimensional domain that takes the form of:
\begin{equation}\label{eqn:Helmholtz_numerics}
\Delta \uk + k^2 n(x) \uk = -k^{\frac{5}{2}} S_{v_\rms}(k(x-x_\rms)), \quad x \in \Rb^2 \,.
\end{equation}
The Sommerfeld radiation condition is imposed at infinity as well. The source term is given by
\begin{equation}\label{eqn:numer_source}
S_{v_\rms}(x) =
\sqrt{2} \left(\frac{\sigma}{\sqrt{\pi}}\right)^{\frac{3}{2}} \exp\left(-\sigma^2\frac{|x|^2}{2}
+ \ri v_\rms\cdot x\right)\,,
\end{equation}
for $(x_\rms,v_\rms)\in\Gamma_-$. We denote the solution to~\eqref{eqn:Helmholtz_numerics} by $\uk_{x_\rms,v_\rms}$ whenever the source center and the incident direction are relevant for the discussion. The Husimi transform defined in \eqref{eqn:Husimi} takes the form
\begin{equation}\label{eqn:husimi_receive_num}
\Hk \uk(x_\rmr,v_\rmr) = \left( \frac{k}{2\pi}\right)^d \left|\uk*\phi_{v_\rmr}^k(x_\rmr)\right|^2\,,
\end{equation}
with $(x_\rmr,v_\rmr)\in\Gamma_+$\,. We let the refractive index $n(x)$ set to be $n(x) = 1 + q(x)$ with the support of the heterogeneity $q(x)\subset B(r)$. The measurement is taken on $\partial B(R)$ with $R>r$. See Figure~\ref{fig:setup_computation} for an illustration of the configuration.

Computationally we set the domain $D = \left[-L/2,L/2\right]^2$, with $L$ significantly bigger than $R$, and choose the spatial mesh size $h = 1/N$ with $N$ being an even integer. For simplicity of representation, we use the angles ${\theta}_\rms$ and ${\theta}_\rmr$ to denote the center of the sources and the center of the receivers, respectively, and the angles ${\theta}_\rmi$ and ${\theta}_\rmo$ are used to denote the incident and outgoing direction of the sources and receivers, respectively, so that:
\begin{equation}
\begin{aligned}
&x_\rms = (R\cos{\theta}_\rms,R\sin{\theta}_\rms)\,,\\
&v_\rms = (-\cos({\theta}_\rms+{\theta}_\rmi),-\sin({\theta}_\rms+{\theta}_\rmi)) \,,
\end{aligned}
\end{equation}
and
\begin{equation}
\begin{aligned}
&x_\rmr = (R\cos({\theta}_\rms+{\theta}_\rmr),R\sin({\theta}_\rms+{\theta}_\rmr)) \\
&v_\rmr = (\cos({\theta}_\rms+{\theta}_\rmr+{\theta}_\rmo),\sin({\theta}_\rms+{\theta}_\rmr+{\theta}_\rmo)) \,.
\end{aligned}
\end{equation}
The angles $\theta_\rmi$ and $\theta_\rmo $ take values in $[0,2\pi)$, whereas the angles $\theta_\rmi$ and $\theta_\rmo$ take values in $(-\frac{\pi}{2},\frac{\pi}{2})$. An illustration of the angles can be found in Figure~\ref{fig:setup_computation}. Since the mapping between $({\theta}_\rms, {\theta}_\rmi, {\theta}_\rms, {\theta}_\rmo)$ and the corresponding $(x_\rms, v_\rms,x_\rmr, v_\rmr)$ is one-to-one, we present the quantities $u^k$ and $H^ku^k$ on the $\theta$ coordinate system whenever there is no confusion.

The angles are discretized with step size $\Delta{\theta}$ and the angular grids are denoted by ${\theta}_\rms^j\,, {\theta}_\rmr^j = j\Delta {\theta}$ for all $j = 0,\cdots,2\pi/\Delta {\theta} -1$, and
${\theta}_\rmi^j\,, {\theta}_\rmo^j = -\frac{\pi}{2} + j\Delta {\theta}$ for all $j = 1,\cdots, \pi/\Delta {\theta} - 1$.
\begin{figure}[htbp]
  \centering
  \includegraphics[width=0.3\textwidth]{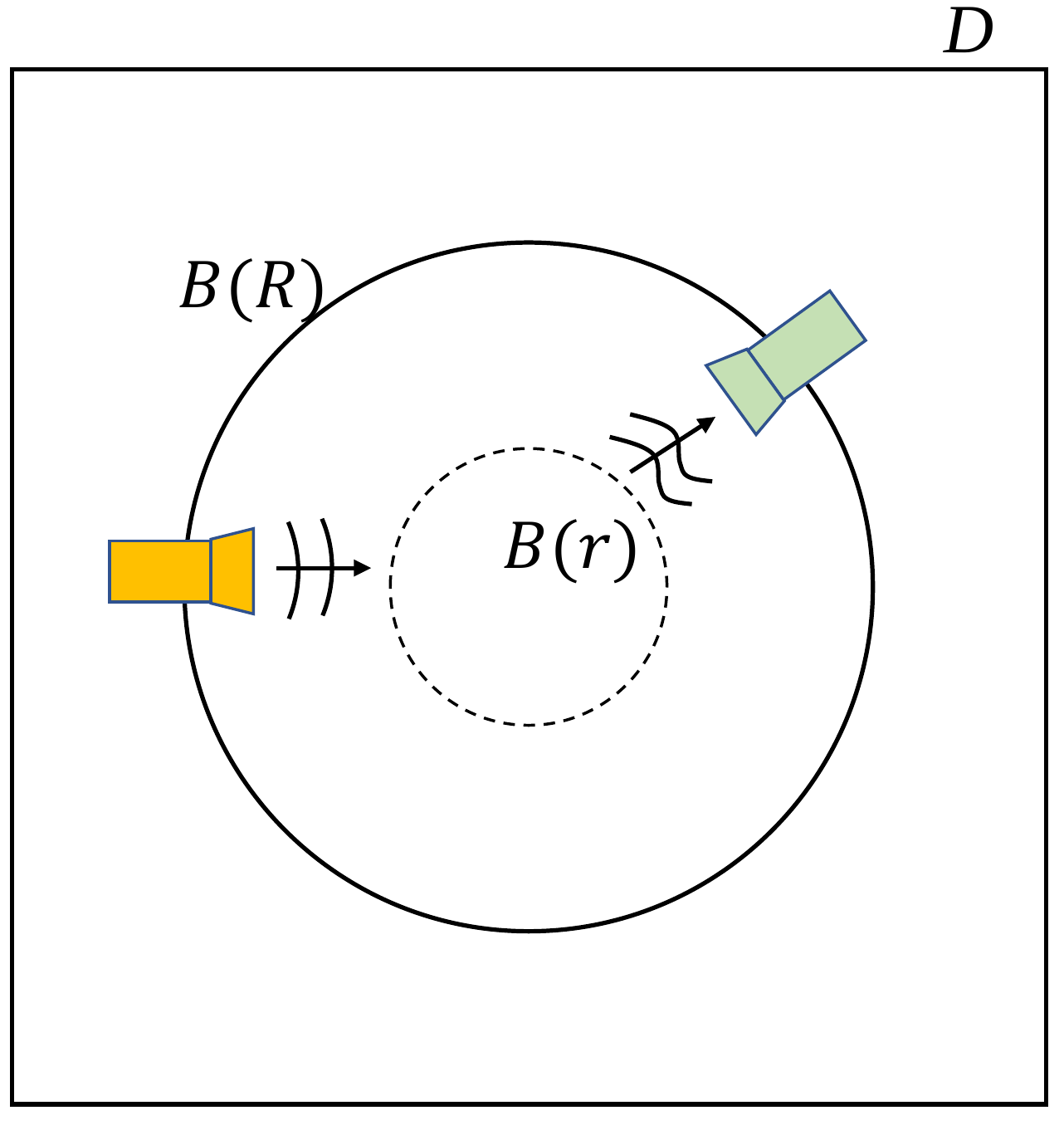}\hspace{2.5cm}
  \includegraphics[width=0.3\textwidth]{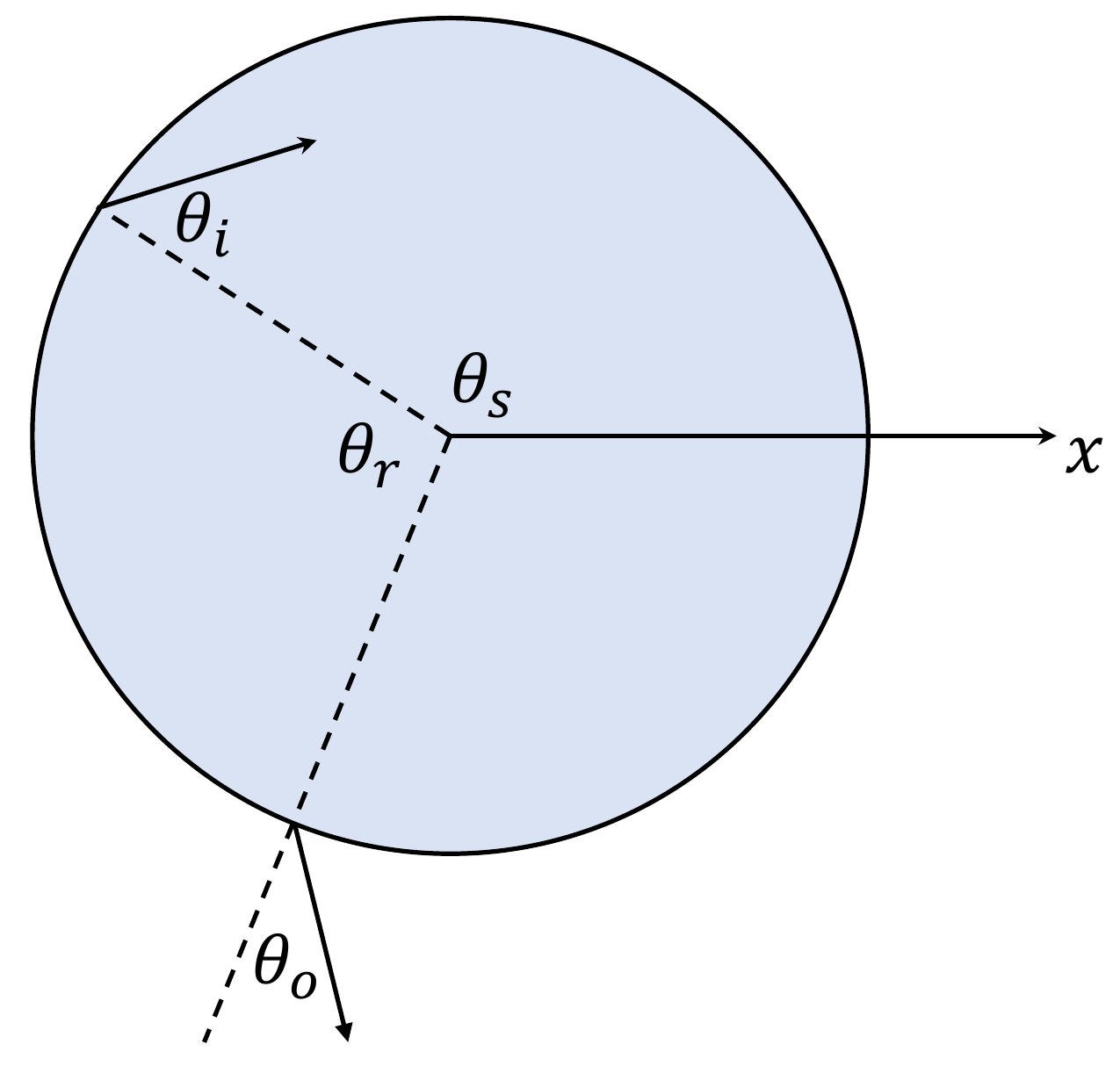}
  \caption{ (left) illustration of the setup for numerical experiments, (right) sketch of the definition of the angles on the circle $\partial B(R)$ used to parameterize the data.}
  \label{fig:setup_computation}
\end{figure}

To compare the Husimi transform of the solutions, we further define two quantities. The first quantity is the Husimi transform integrated in the outgoing direction
\begin{equation}\label{eqn:Hu_o_averaged}
M^k_\rmo(x_\rms,{v}_\rms,{x}_\rmr) := \int_{\Sb_{x_\rmr}^+} H^k u^k_{x_\rms,v_\rms}(x_\rmr,v_\rmr) \rmd v_\rmr = \int_{-\pi/2}^{\pi/2} H^k \uk_{\theta_\rms,{\theta}_\rmi} ({\theta}_\rmr,{\theta}_\rmo) \rmd {\theta}_\rmo \,,
\end{equation}
where $\Sb_{x_\rmr}^\pm = \{v\in\Sb^1: \pm\nu(x_\rmr)\cdot v>0\}$ and $\nu(x)$ is the unit outer normal vector at $x\in\partial\Omega$. Similarly, we also define the Husimi transform integrated along the outgoing boundary
\begin{equation}\label{eqn:Hu_r_averaged}
M^k_\rmr(x_\rms,{v}_\rmi,{v}_{\rmr}) := \int_{\partial\Omega_{v_\rmr}^+} H^k u^k_{x_\rms,v_\rmi}(x_\rmr,v_\rmr) \rmd x_\rmr
 =  \int_{(-\pi/2+{\theta}_{\rmo\rmr},\pi/2+{\theta}_{\rmo\rmr})} H^k \uk_{\theta_\rms,{\theta}_\rmi} ({\theta}_\rmr,{\theta}_{\rmo\rmr}-{\theta}_{\rmr}) \rmd {\theta}_\rmr \,,
\end{equation}
where we denote ${\theta}_{\rmo\rmr} = {\theta}_\rmo + {\theta}_\rmr\in[0,2\pi)$, and define $\partial\Omega_{v_\rmr}^\pm = \{x\in\partial\Omega: \pm\nu(x)\cdot v_\rmr>0\}$.

To solve the Helmholtz equation~\eqref{eqn:Helmholtz_numerics}, we use the truncated kernel method~\cite{ViGeFe:2016fast}, and solve for the Lippmann-Schwinger equation to obtain the scattered field $\usk$. This allows us to push for high-frequency without suffering from the numerical pollution that Finite Differences or Finite Elements methods often have. The scattered field is then combined with the incident field $\uik$ to yield $\uk$.

 \subsection{Numerical examples}
In the first example, we set $L = 1$, $R = 0.3$ and $r = 0.25$. For the medium, we set the heterogeneity to be the radially symmetric smooth function
\begin{equation}\label{eqn:medium_ex1}
q(x) =
\begin{cases}
A\exp\left( -\frac{1}{1-|x|^2/r^2} \right)\,, \quad |x|<r \,,\\
0\,, \quad \text{otherwise} \,.
\end{cases}
\end{equation}
Clearly, the support of $q(x)$ is contained in $B(r)$; see Figure~\ref{fig:medium_source_ex1}. We note that with $-1<A\leq 0$, the media is locally repulsive, and the incident wave is guaranteed to be expelled out of the domain.
\begin{figure}[htbp]
  \centering
  \includegraphics[width=0.4\textwidth]{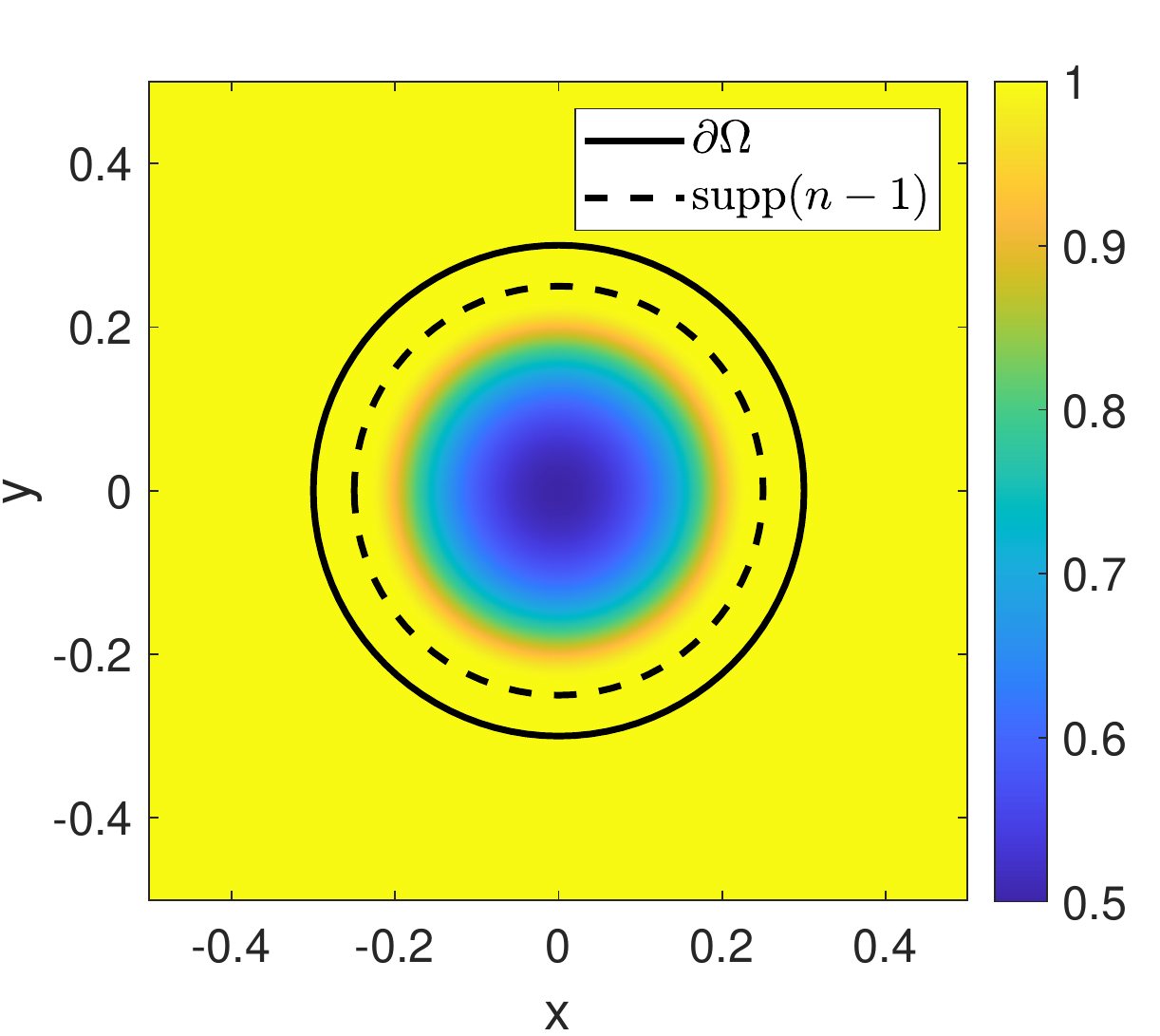}
  \includegraphics[width=0.4\textwidth]{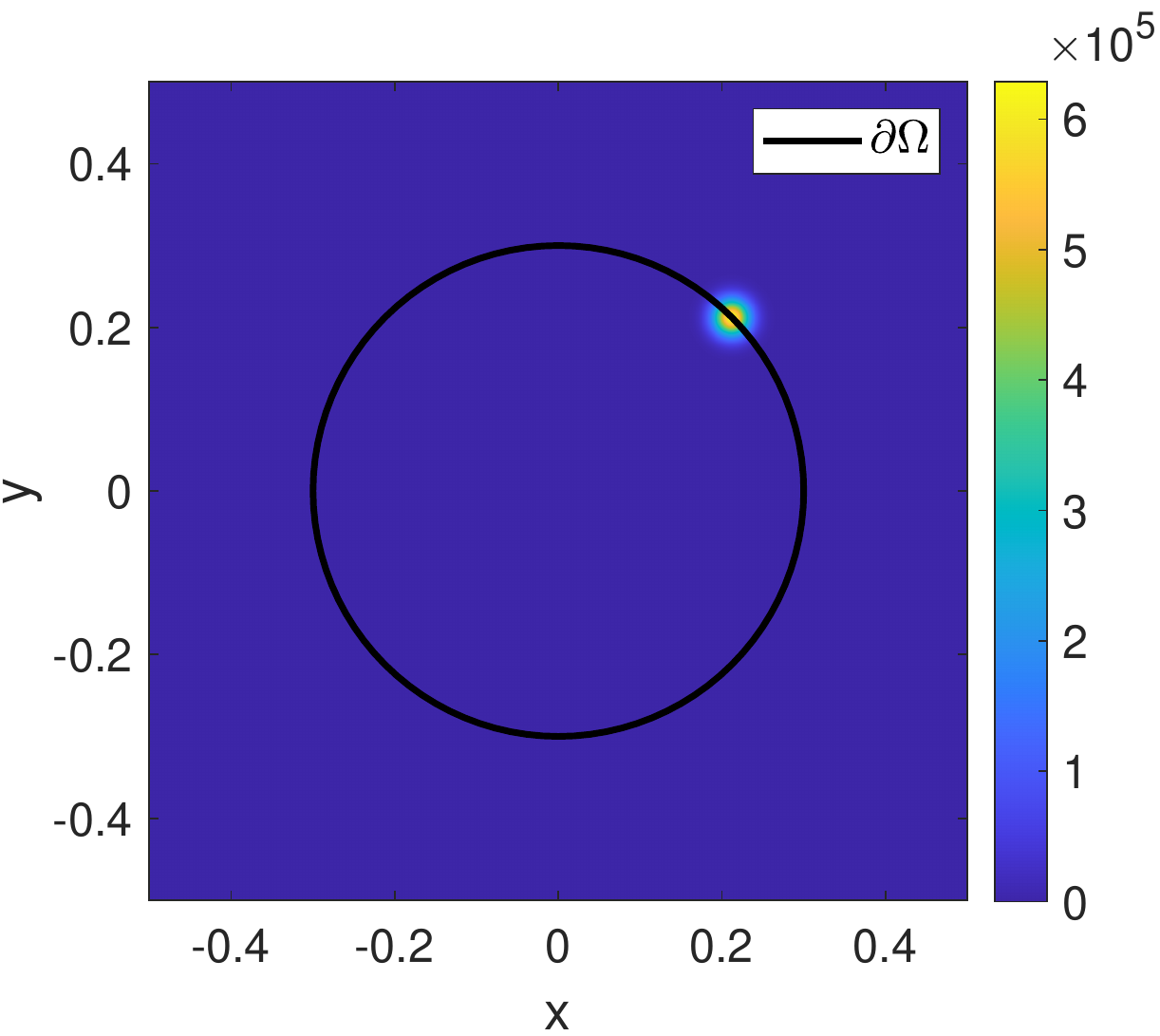}
  \caption{The left plot illustrates the medium $n(x) = 1 + q(x)$ in~\eqref{eqn:medium_ex1} with $A = -0.5$. The right plot shows the amplitude of source $|S_{v_s}(k(x-x_s))|$ with $k=2^{11}$, $\sigma=2^{-5}$ and $\theta_\rms=\pi/4$.}
  \label{fig:medium_source_ex1}
\end{figure}
For the source term, we fix $\sigma = 2^{-5}$ in the following experiments. Noting that the medium $n(x)$ is radially symmetric, one can study the scattered data for a fixed source location. We choose $\theta_\rms = \pi/4$; see Figure~\ref{fig:medium_source_ex1}.
For discretization, we choose spatial step size $h = 1/(2k)$ in the truncated kernel solver, and $\Delta\theta = \pi/30$ for the angular grids.

We first show the solution's behavior as $k$ increases in Figure~\ref{fig:Reu}. As $k$ increases, the solution converges to a narrow beam that follows the characteristic equation~\eqref{eqn:chareqn}.
\begin{figure}[htbp]
  \centering
  \includegraphics[width=0.3\textwidth]{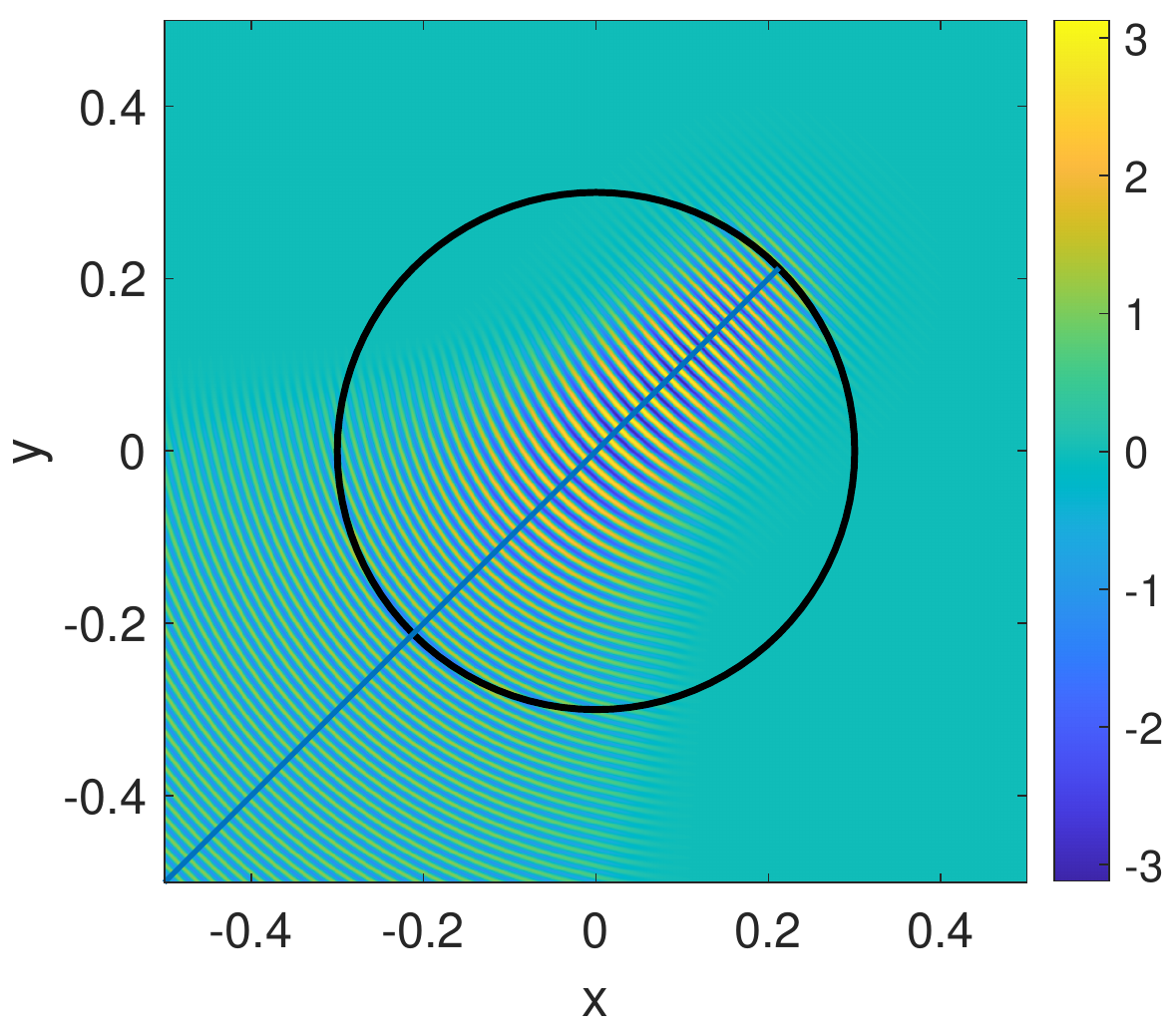}
  \includegraphics[width=0.3\textwidth]{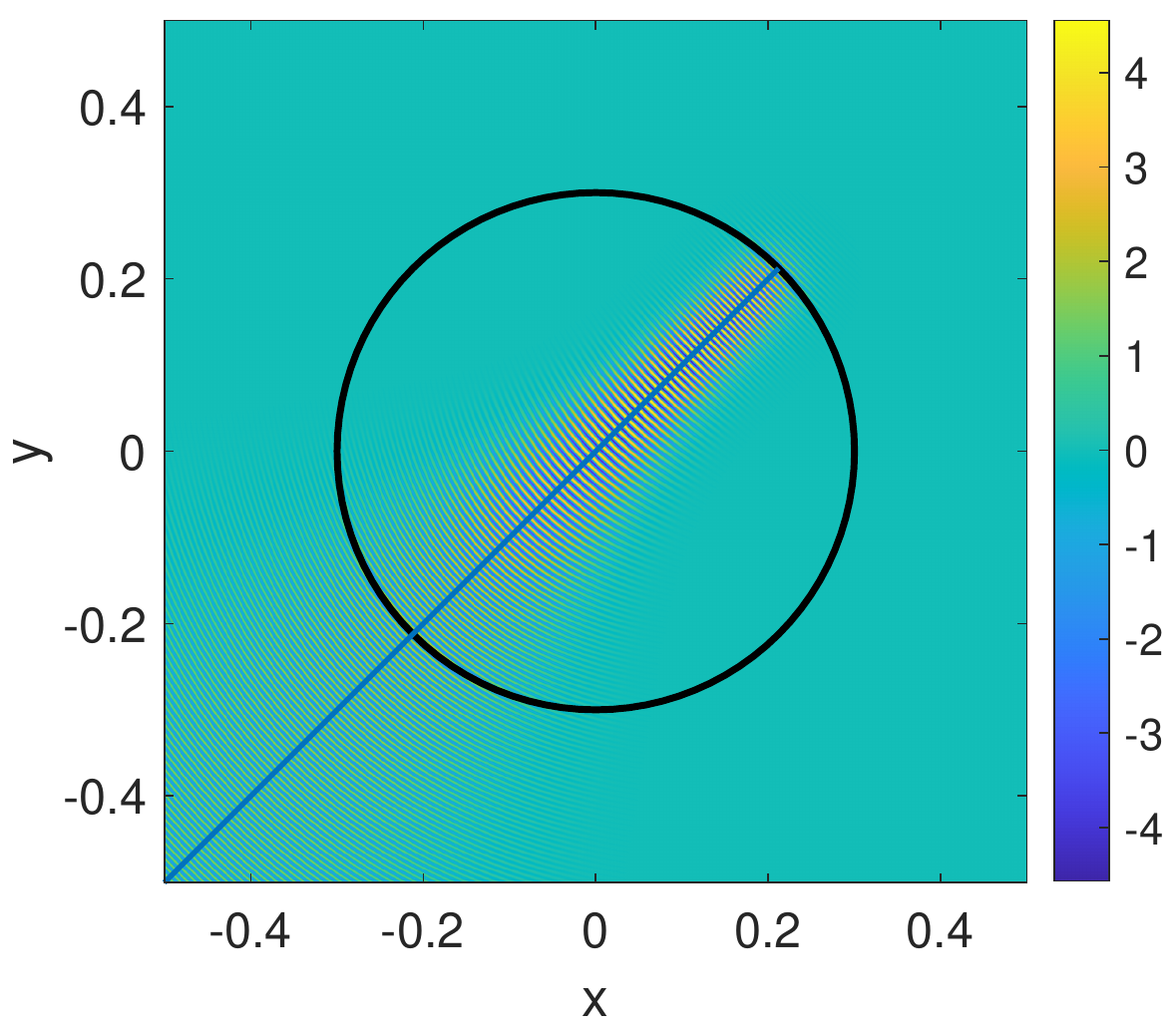}
  \includegraphics[width=0.3\textwidth]{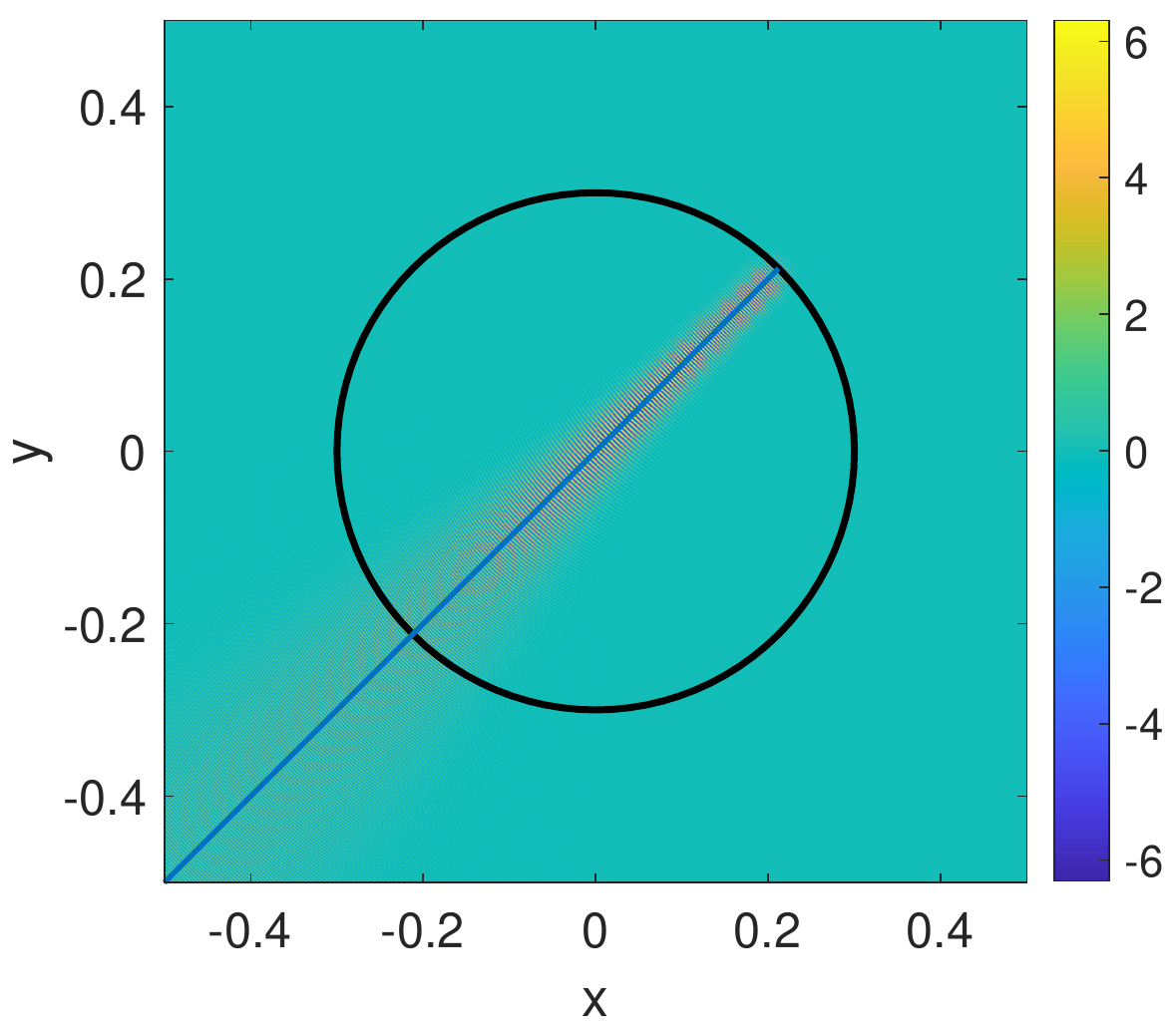}
  \\
  \includegraphics[width=0.3\textwidth]{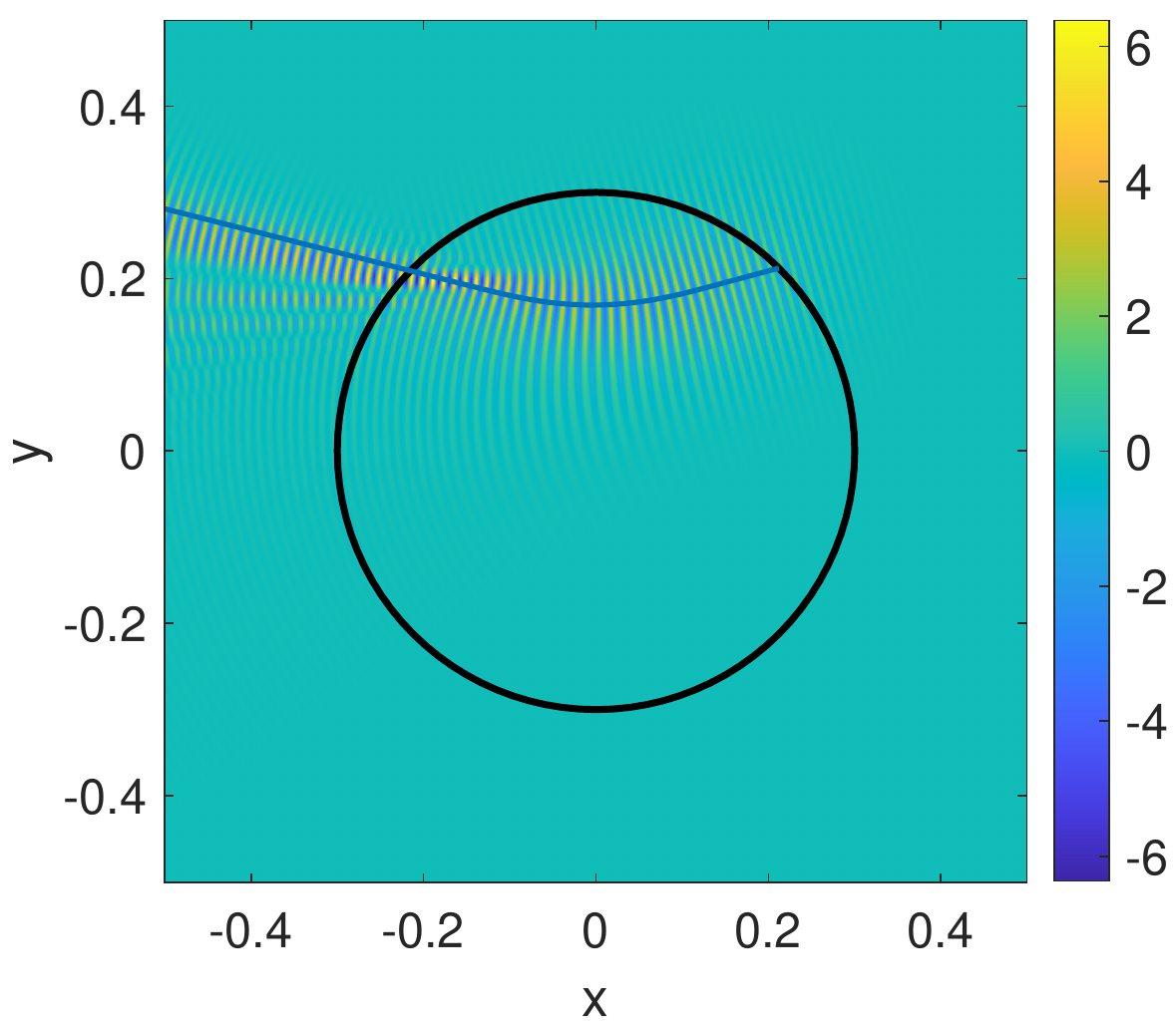}
  \includegraphics[width=0.3\textwidth]{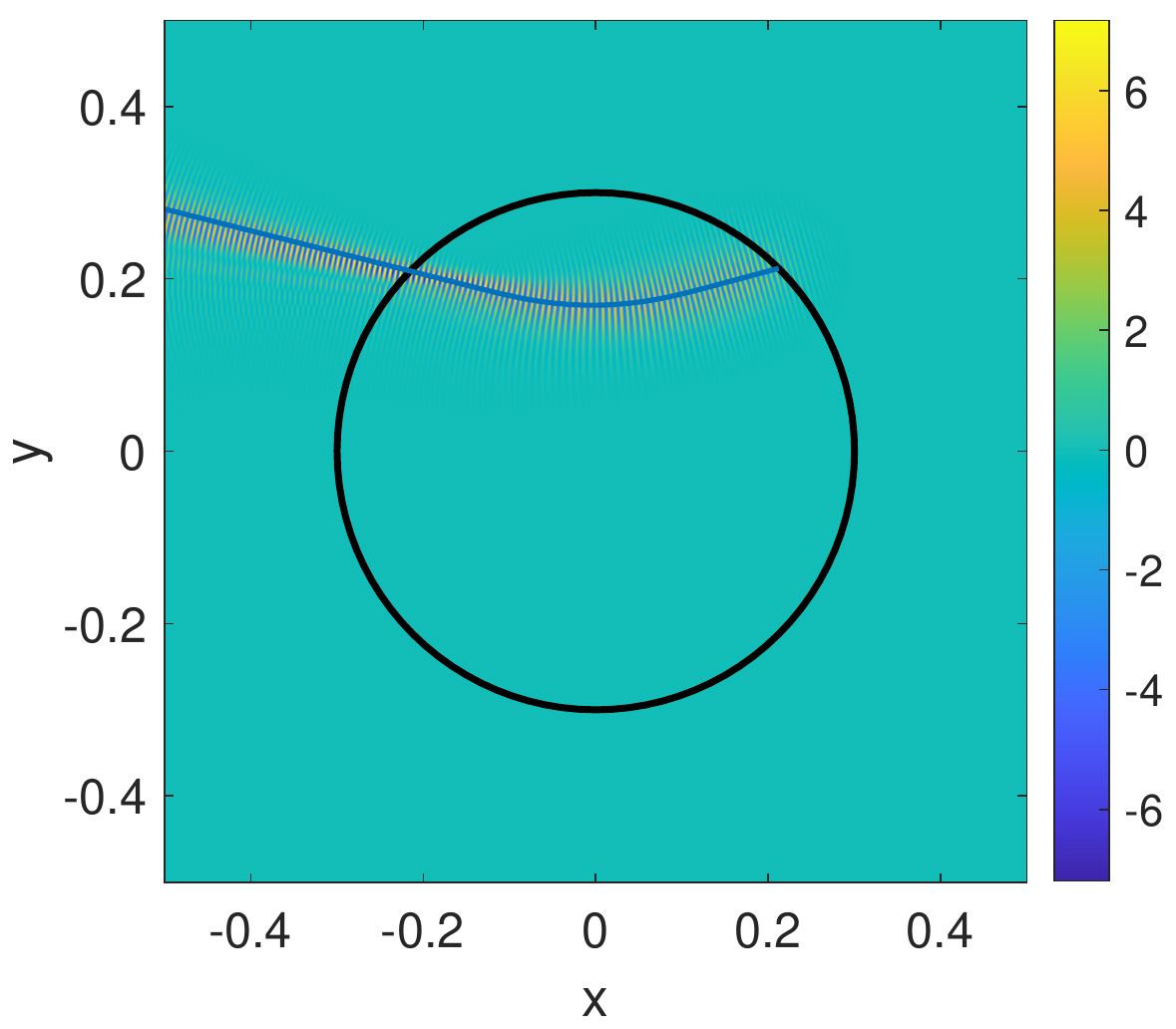}
  \includegraphics[width=0.3\textwidth]{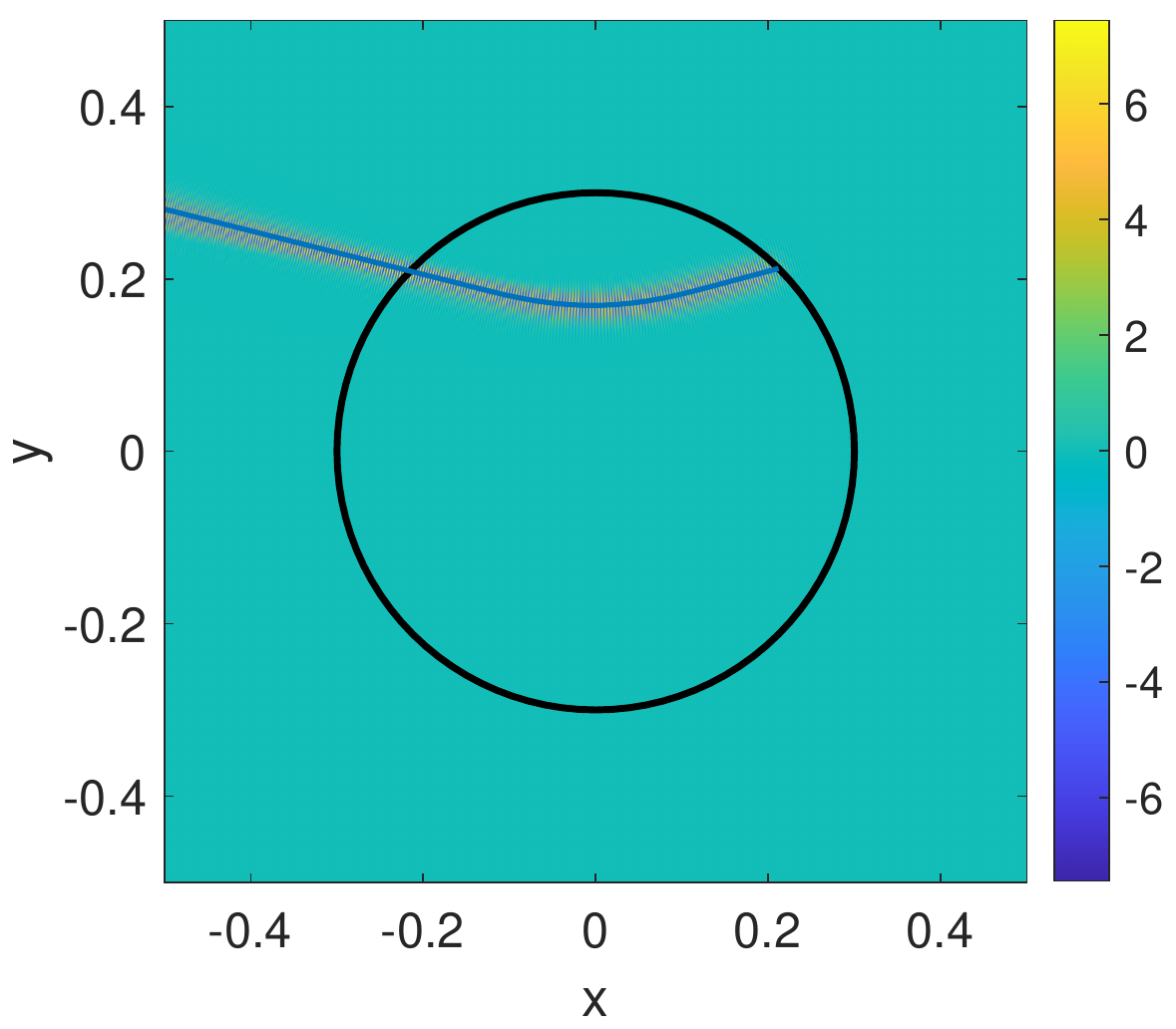}
  \caption{The real part of $\uk$ for $k=2^9$ (left), $k=2^{10}$ (middle) and $k=2^{11}$ (right). The blue lines show the Liouville trajectory that solves~\eqref{eqn:chareqn}. The medium~\eqref{eqn:medium_ex1} has amplitude $A = -0.5$. The incident direction $\theta_\rmi = 0$ (upper) and $\theta_\rmi = -\pi/6$ (lower).}
  \label{fig:Reu}
\end{figure}

We compute the Husimi transform $H^k\uk$ for different $k$ and we compare them with the trajectories of the Liouville equation. The results are shown in Figure~\ref{fig:Hu_ex1}, where we can observe that for a fixed ${\theta}_\rmi$, $H^k\uk$ converges to a delta function on the ${\theta}_\rmr$-${\theta}_\rmo$ plane, as $k$ increases. This agrees with the statement in Theorem~\ref{thm:formal}, especially equation~\eqref{eqn:measure_outgoing_conv}.

We then compare the integrated Husimi transform defined in~\eqref{eqn:Hu_o_averaged} and~\eqref{eqn:Hu_r_averaged}. In Figure~\ref{fig:Hu_o} and Figure~\ref{fig:Hu_r}, we demonstrate the convergence of $M^k_\rmo$ and $M^k_\rmr$ as $k$ increases.

As $k$ increases, the outgoing data becomes more and more sparse, and fewer and fewer detectors can receive outgoing light, leading to the sparser matrix presentation of $\Lambda_n^k$ (see definition in~\eqref{eqn:forwardmap}. This is shown in Figure~\ref{fig:sparsity} for different $k$.
\begin{figure}[htbp]
  \centering
  \includegraphics[width=0.3\textwidth]{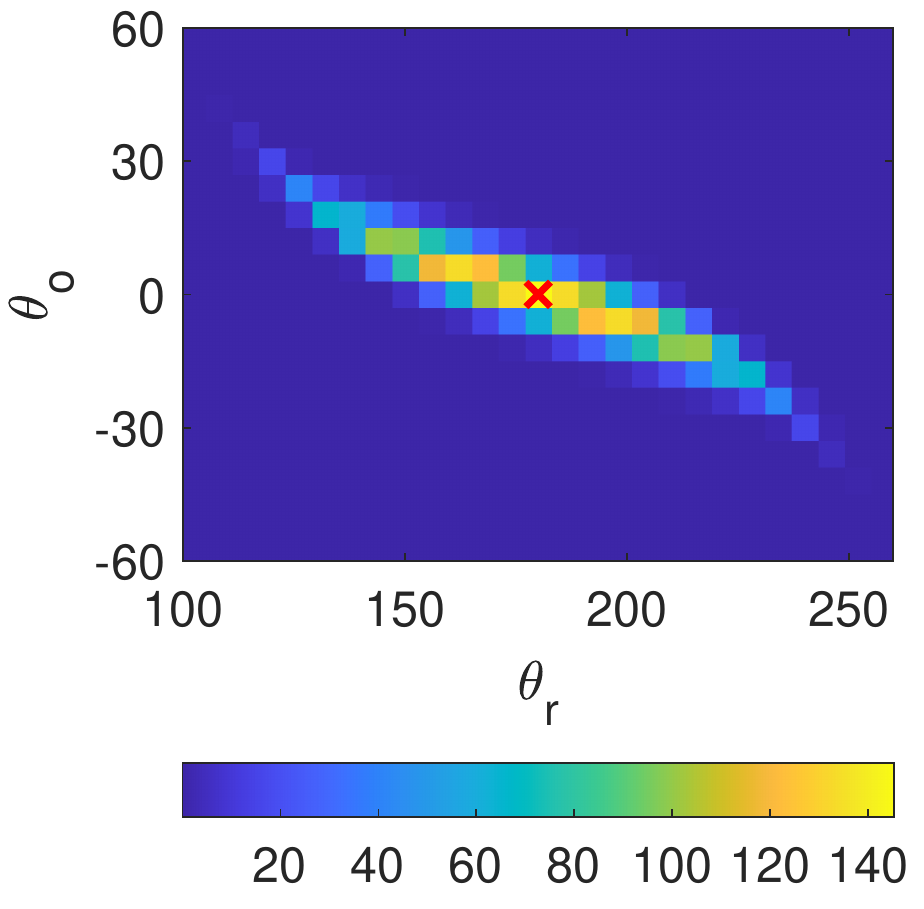}
  \includegraphics[width=0.3\textwidth]{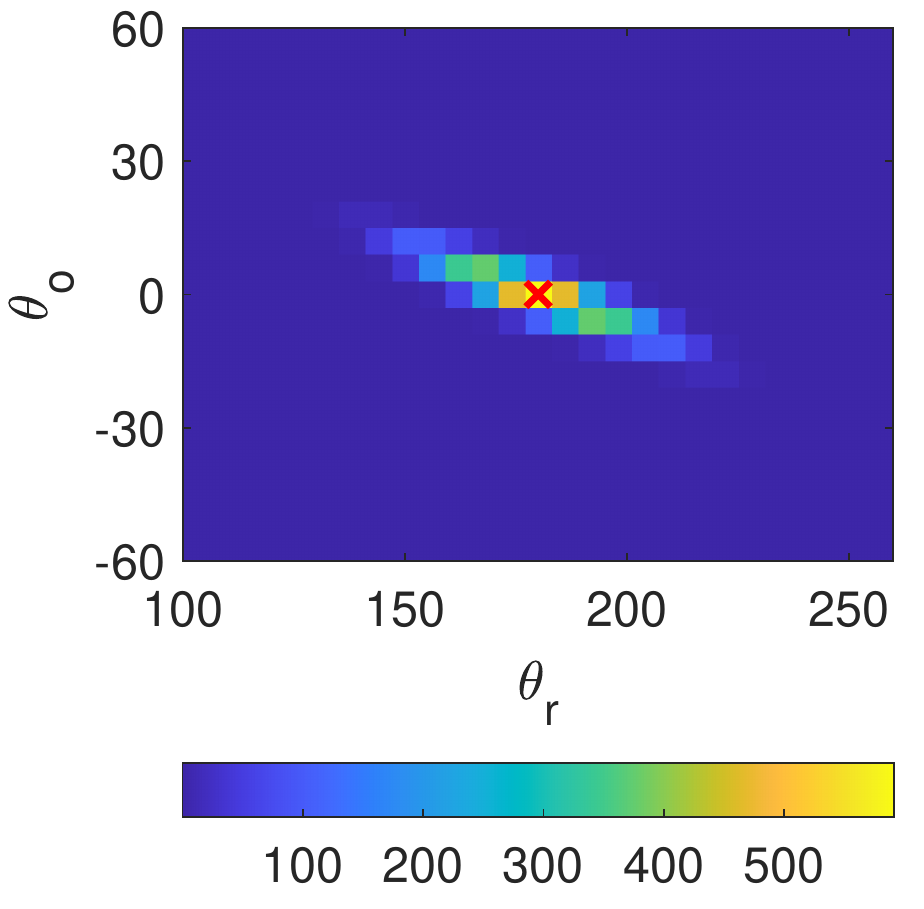}
  \includegraphics[width=0.3\textwidth]{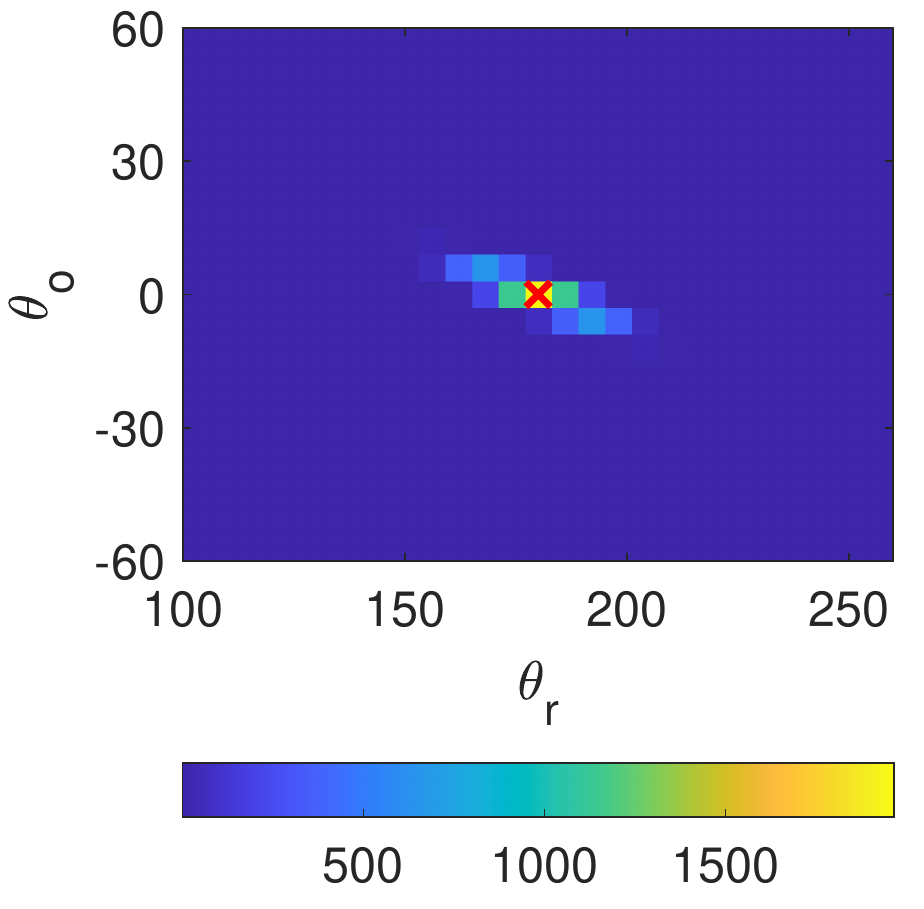}
  \\
  \includegraphics[width=0.3\textwidth]{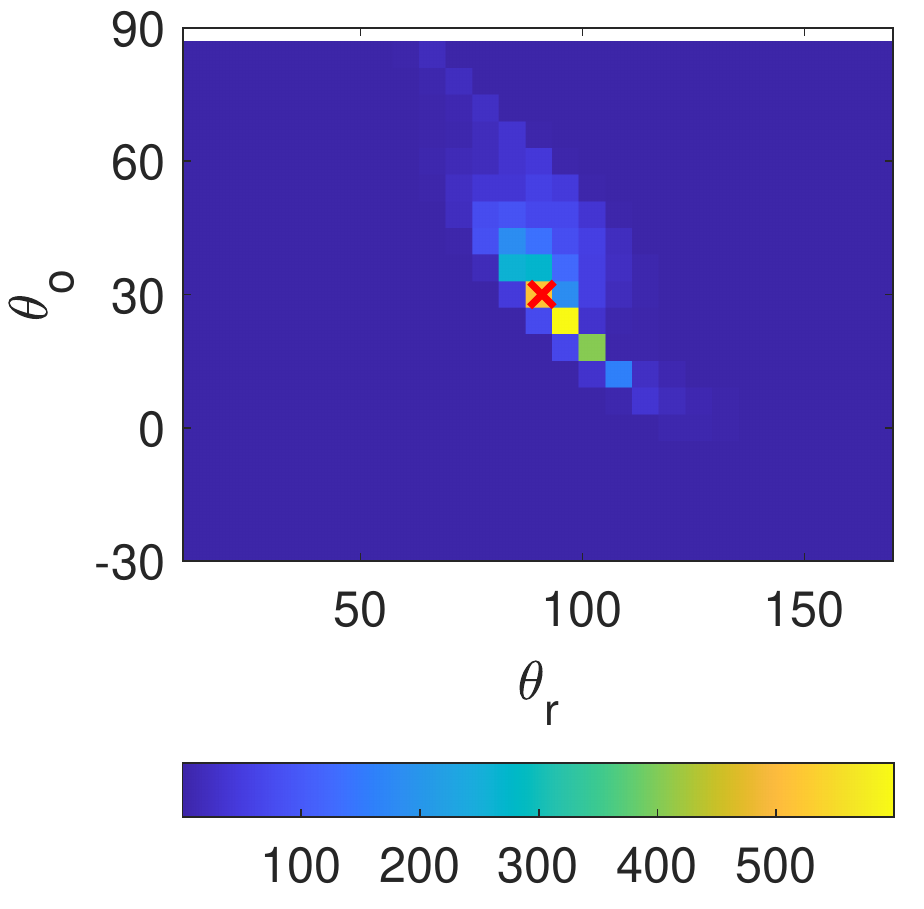}
  \includegraphics[width=0.3\textwidth]{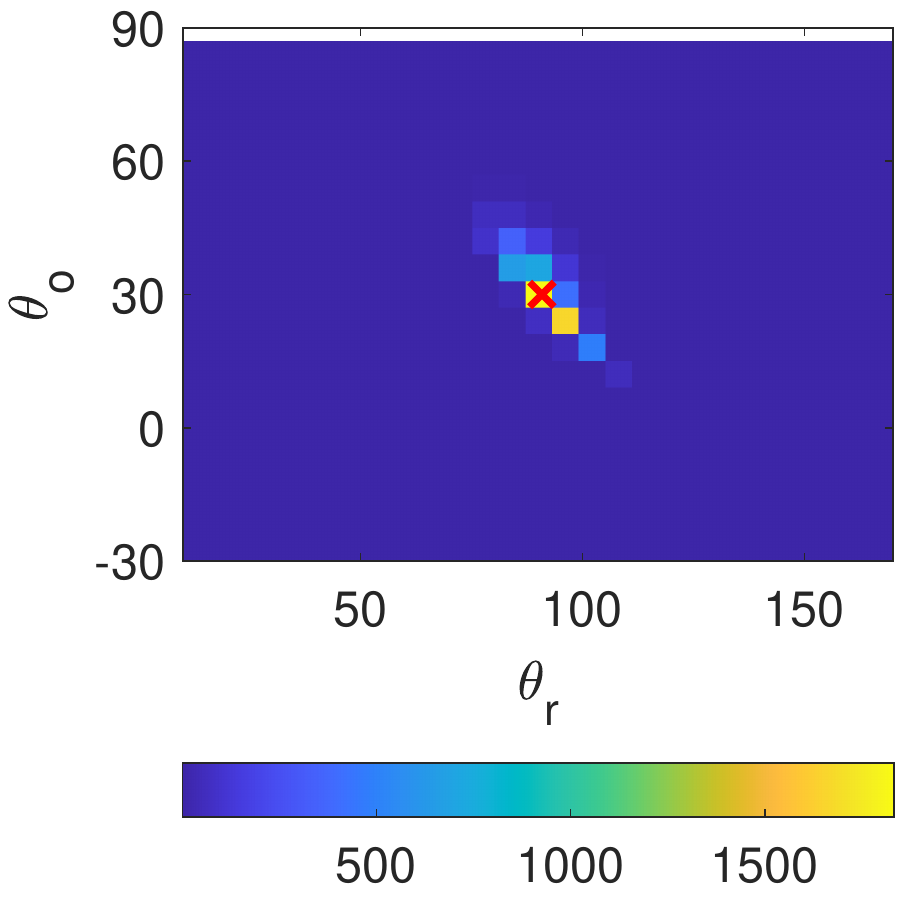}
  \includegraphics[width=0.3\textwidth]{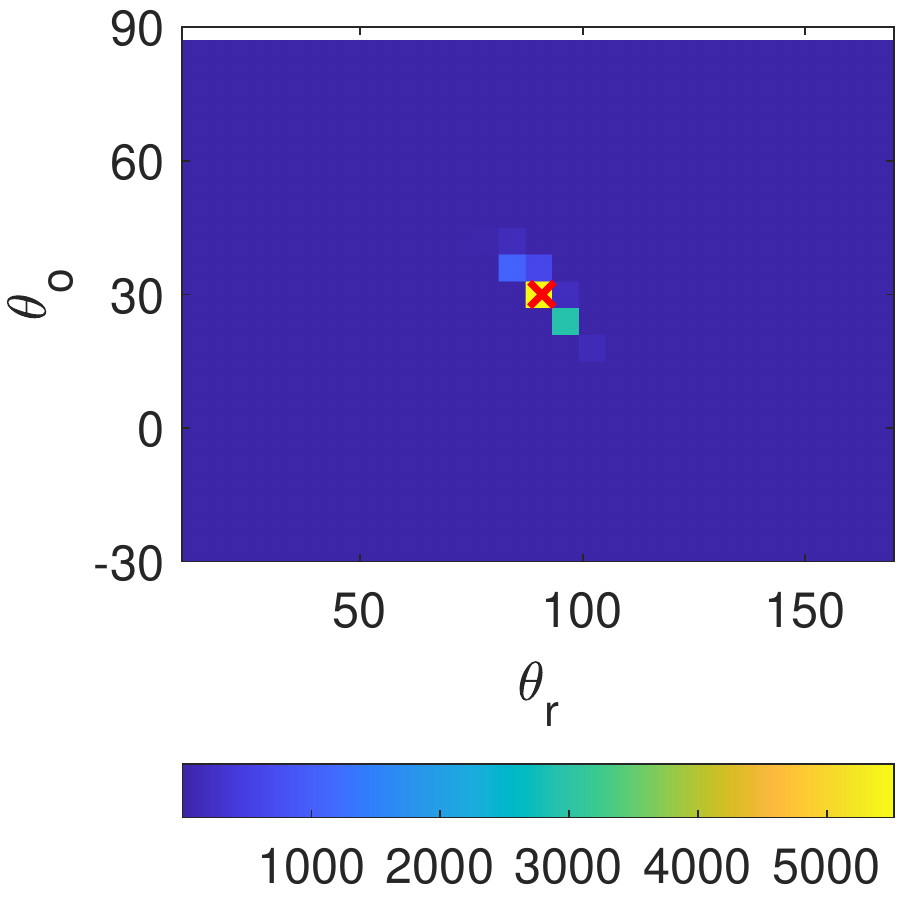}

  \caption{The Husimi transform $H^k\uk$ for $k=2^9$ (left), $k=2^{10}$ (middle) and $k=2^{11}$ (right). The upper row shows the results with ${\theta}_\rmi = 0$, and the lower row shows the results with ${\theta}_\rmi = -\pi/6$. The red crosses show the outgoing position and direction~\eqref{eqn:Liouville_out} of the Liouville trajectory. The medium~\eqref{eqn:medium_ex1} has amplitude $A = -0.5$.
  }
  \label{fig:Hu_ex1}
\end{figure}

\begin{figure}[htbp]
  \centering
  \includegraphics[width=0.4\textwidth]{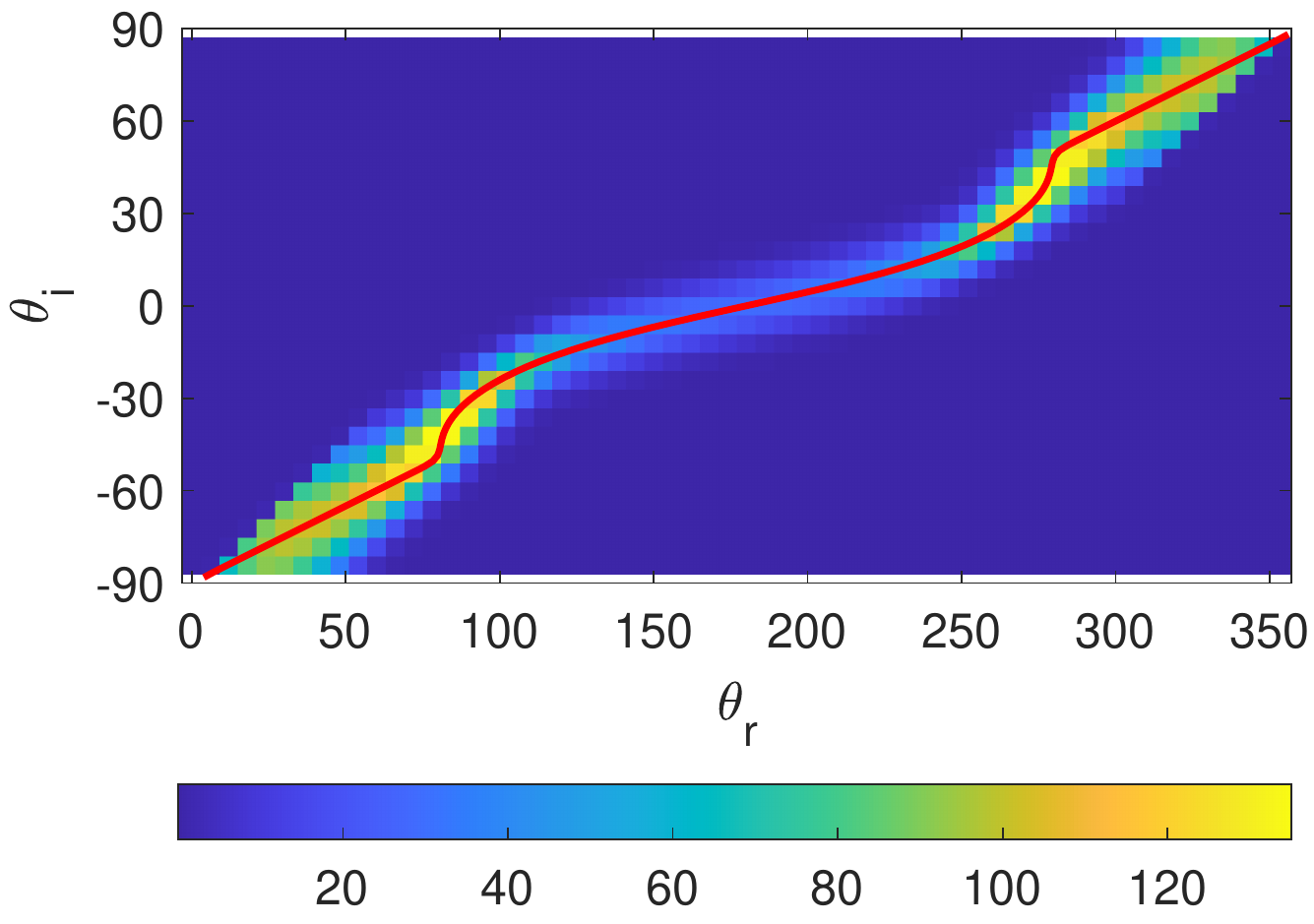}
  \includegraphics[width=0.4\textwidth]{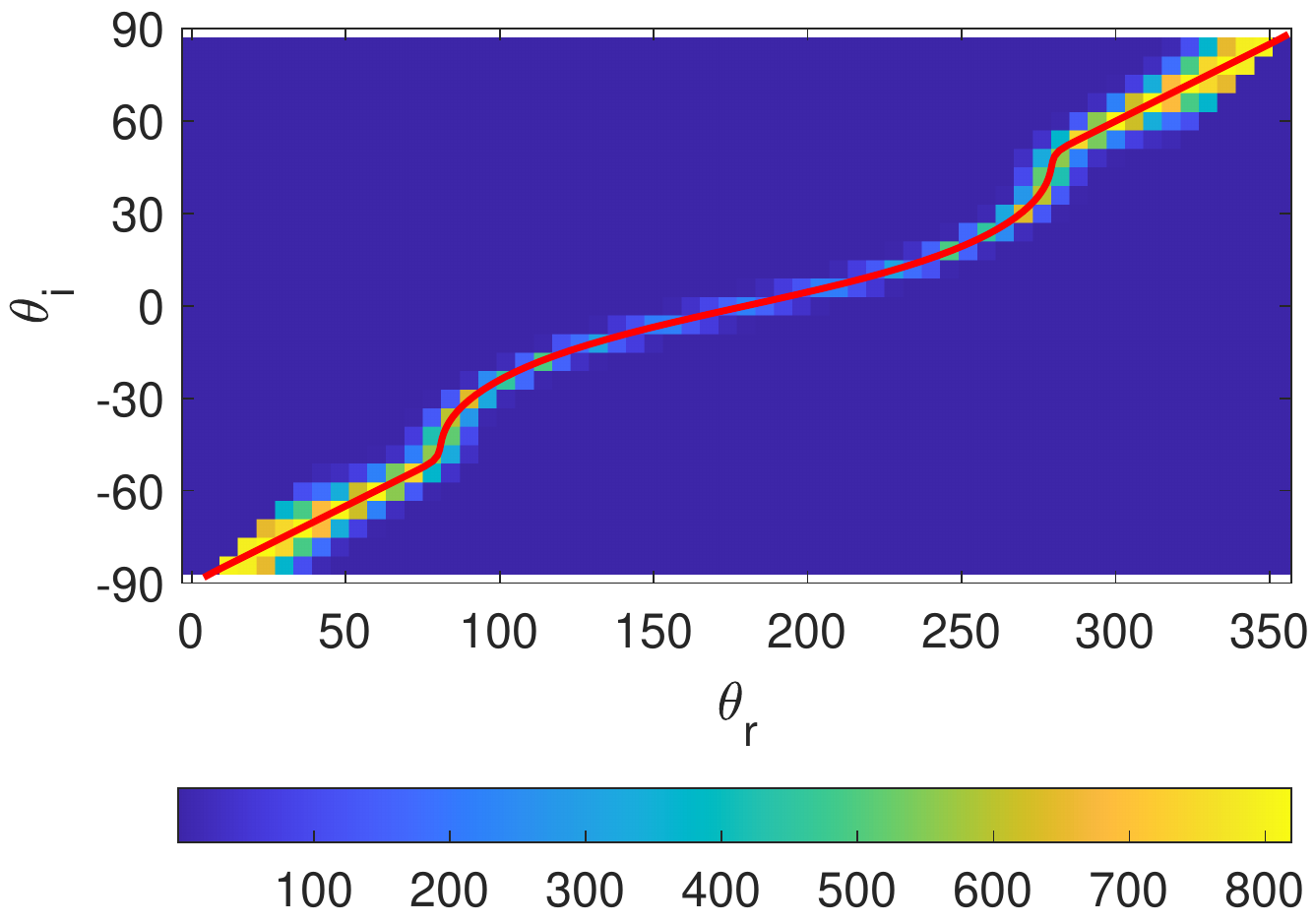}
  \caption{The averaged Husimi transform $M^k_\rmo$ for $k = 2^9$ (left) and $k = 2^{11}$ (right). The red lines show the outgoing position~\eqref{eqn:Liouville_out} of the Liouville trajectory. The medium~\eqref{eqn:medium_ex1} has amplitude $A = -0.5$.}
  \label{fig:Hu_o}
\end{figure}

\begin{figure}[htbp]
  \centering
  \includegraphics[width=0.4\textwidth]{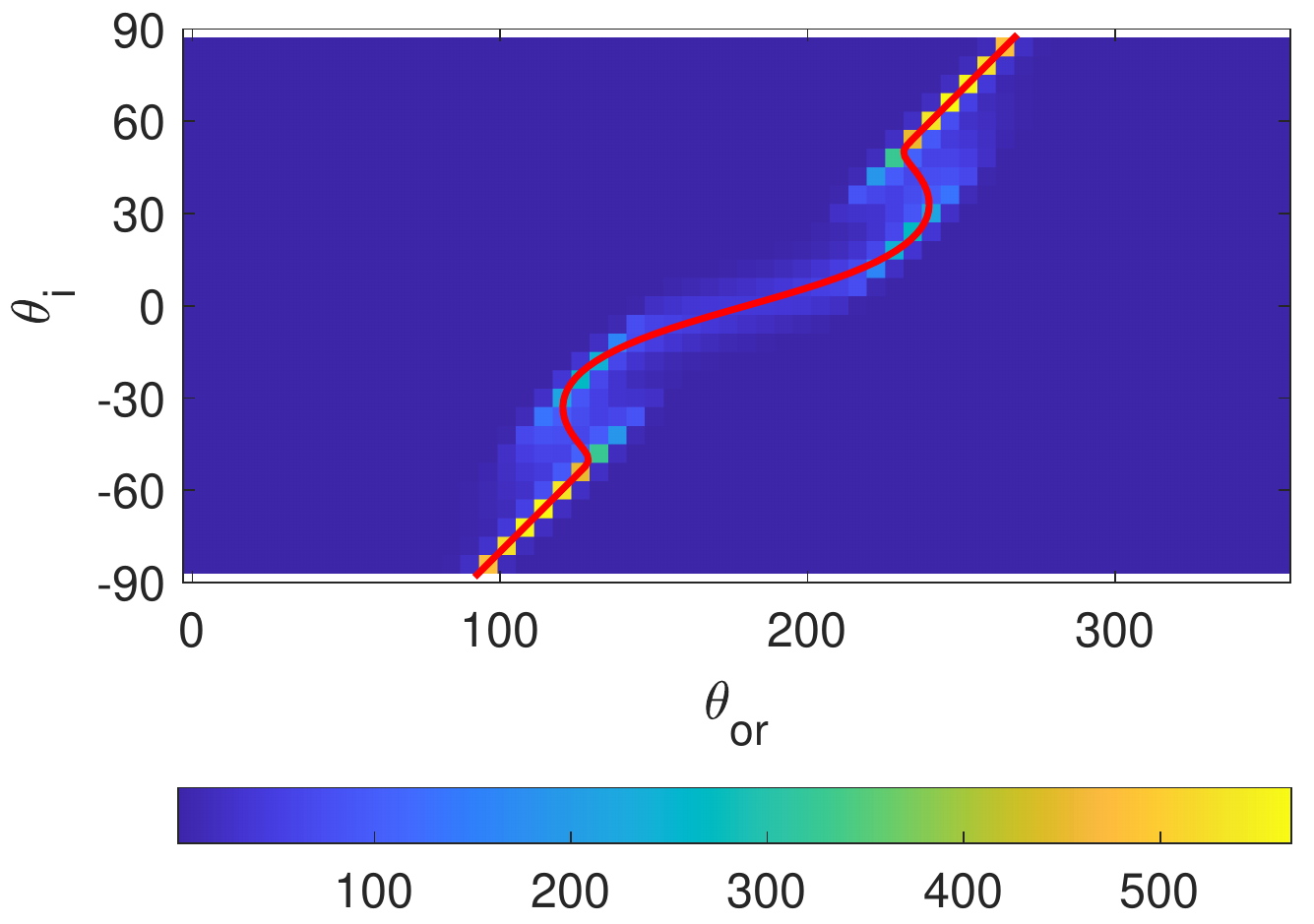}
  \includegraphics[width=0.4\textwidth]{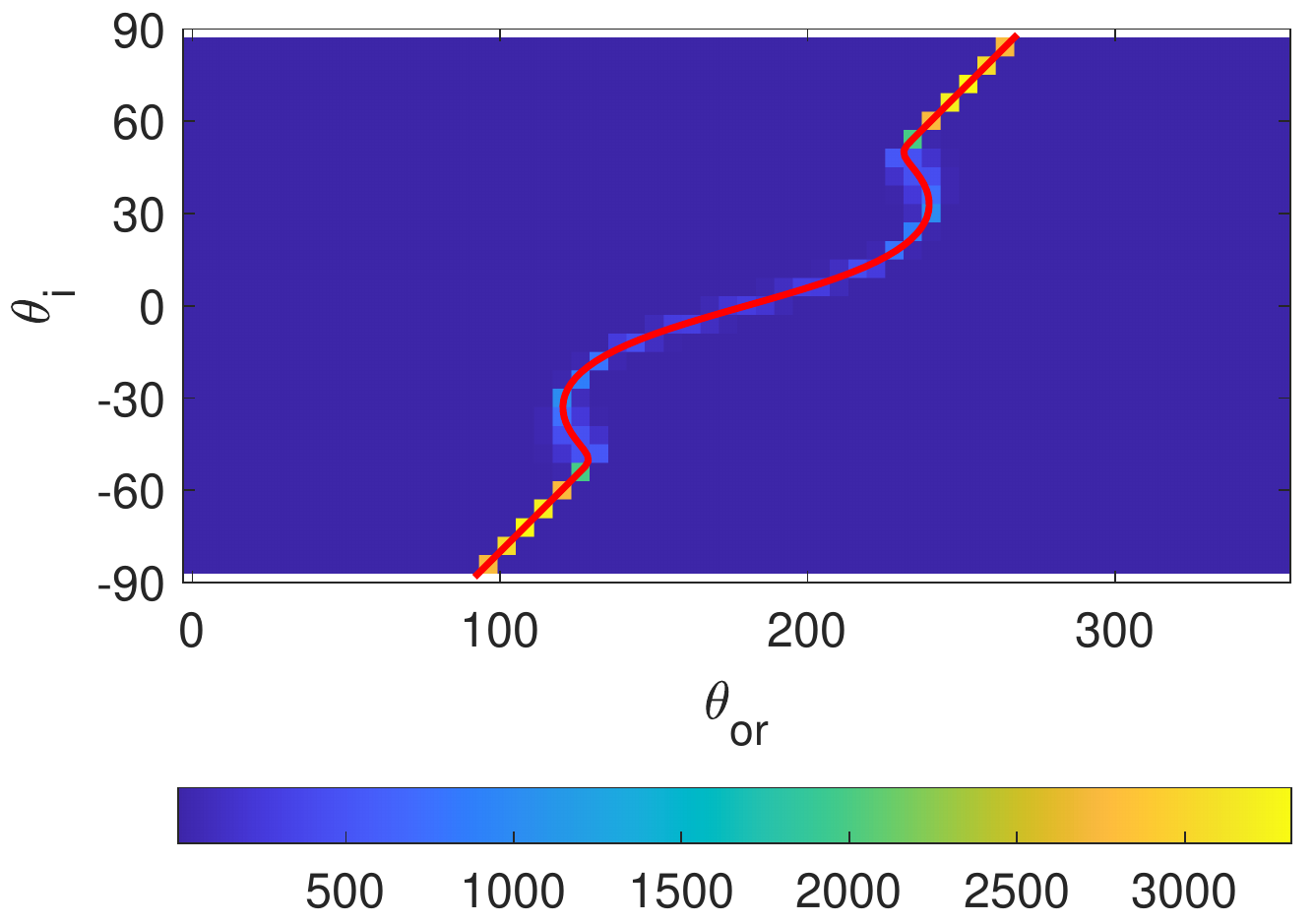}
  \caption{The averaged Husimi transform $M^k_\rmr$ for $k = 2^9$ (left) and $k = 2^{11}$ (right). The red lines show the outgoing direction~\eqref{eqn:Liouville_out} of the Liouville trajectory. The medium~\eqref{eqn:medium_ex1} has amplitude $A = -0.5$.}
  \label{fig:Hu_r}
\end{figure}

\begin{figure}[htbp]
  \centering
  \includegraphics[width=0.4\textwidth]{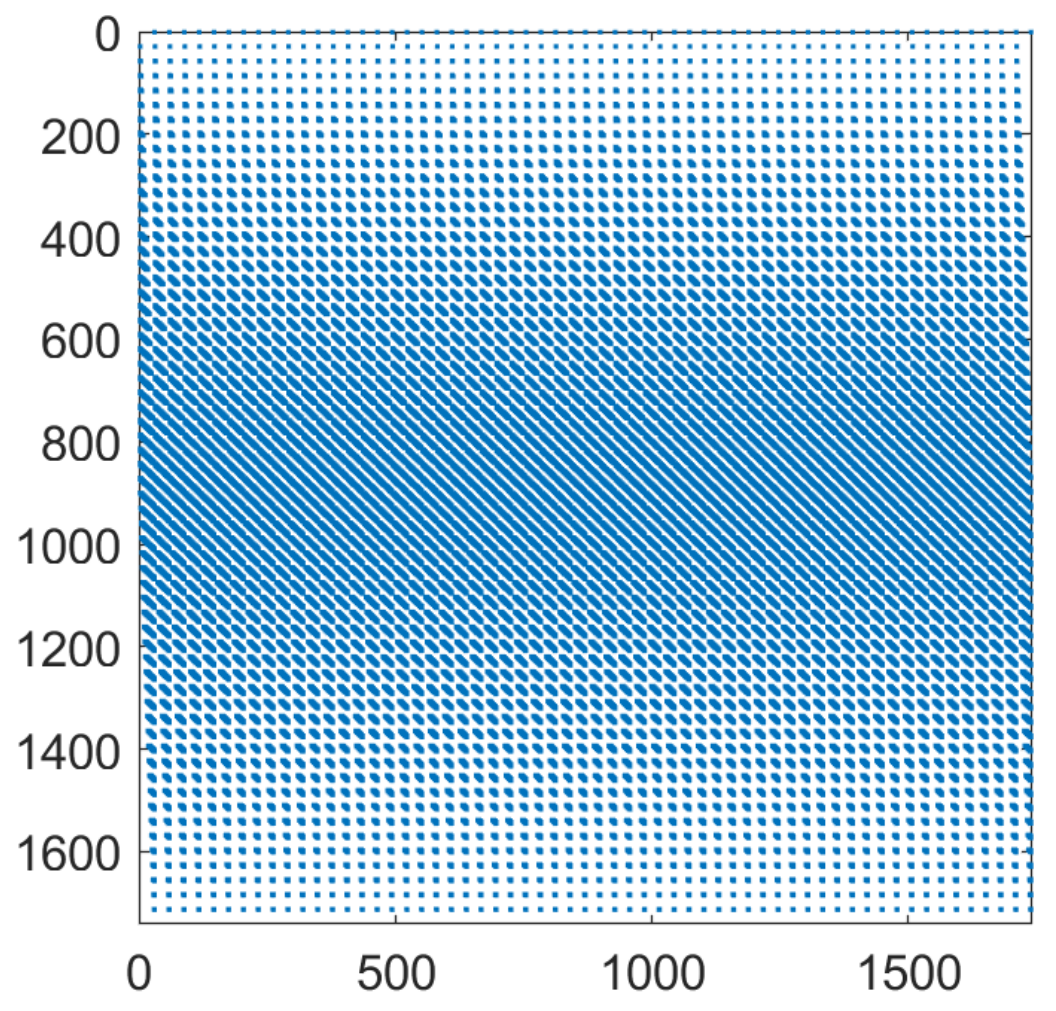}
  \includegraphics[width=0.4\textwidth]{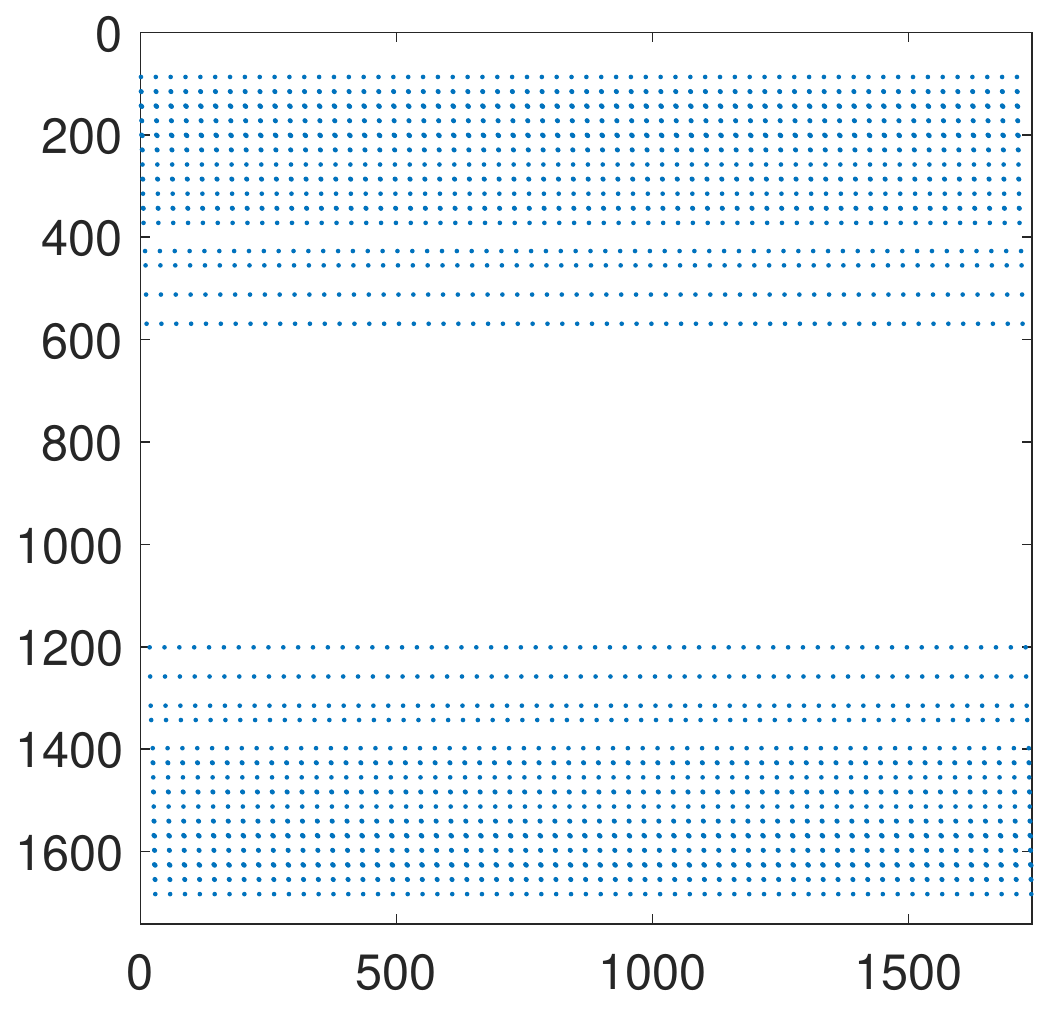}
  \caption{Sparsity of the matrix $\Lambda^k_n$ for $k=2^4$ (left) and $k=2^{11}$ (right). Rows represent different $(\theta_\rmr,\theta_\rmo)$, and columns represent different $(\theta_\rms,\theta_\rmi)$. Elements that are larger than half of the maximal element in $\Lambda^k_n$ are shown. For $k = 2^4$, we use larger computational domain $[-8,8]^2$, and the step size is $h = 2^{-8}$.}
  \label{fig:sparsity}
\end{figure}

Finally we compare the change of $\Lambda^k_n$ as $n$ differs, for different $k$. Let $n_0(x) = 1$ as the background media whose corresponding map is denoted $\Lambda_0^k$, and by adjusting $A$ we design a sequence of $n(x)$.
We measure how the Frobenius norm $\|\Lambda^k_n-\Lambda_0^k\|_\mathrm{F}$ changes with respect to $\|n-n_0\|_{L^\infty}$ for different $k$.
As can be seen in Figure~\ref{fig:lamdba_vs_n}, as $k$ increases, the slope of $\|\Lambda^k-\Lambda_0^k\|_\mathrm{F}$ as $\|n-n_0\|_{L^\infty} \to 0$ increases. This confirms that bigger $k$ sees more sensitivity of the data when $n$ changes, hence the reconstruction is expected to be better for higher $k$.
\begin{figure}[htbp]
  \centering
  \includegraphics[width=0.4\textwidth]{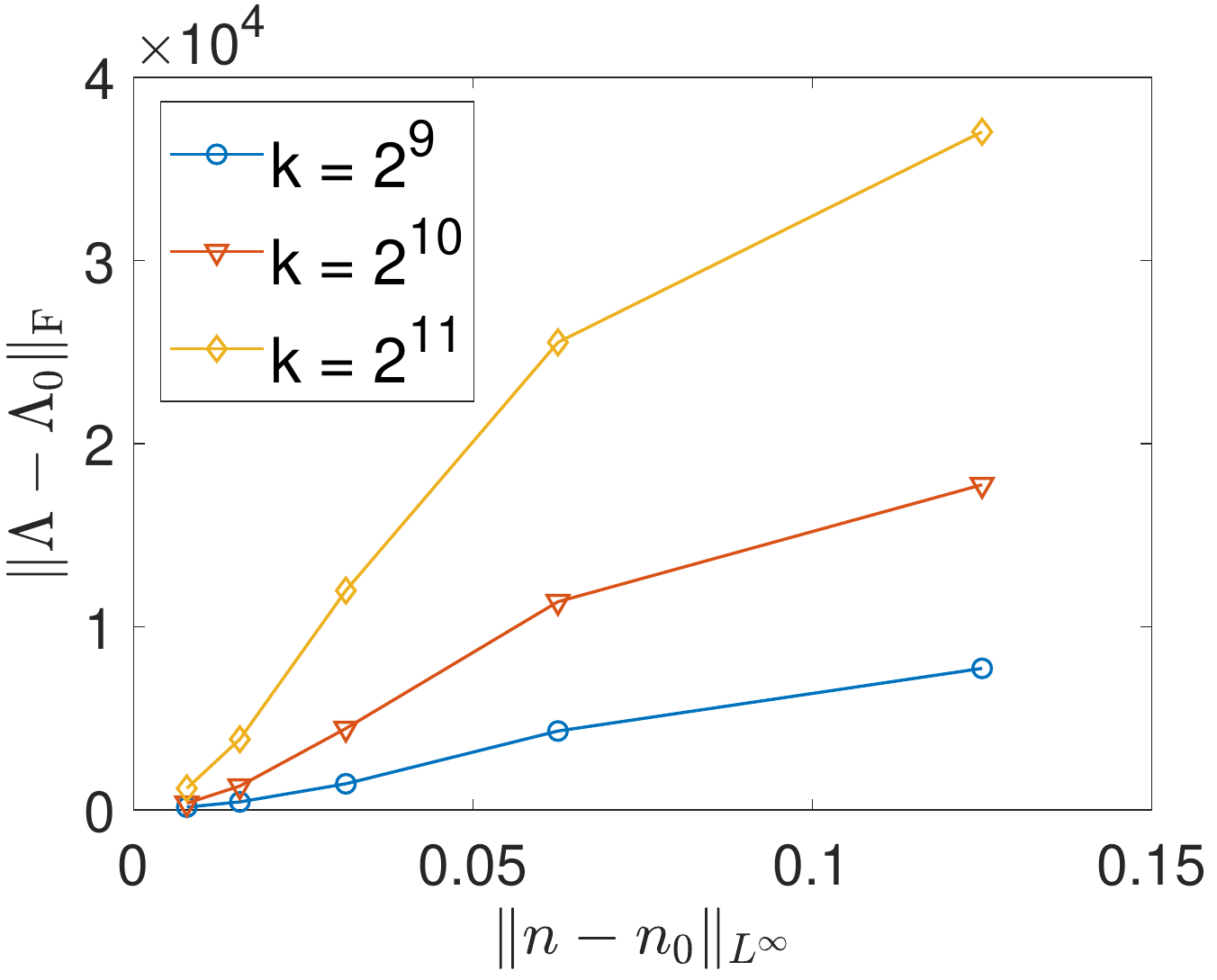}
  \caption{The dependence of $\|\Lambda^k-\Lambda_0^k\|_\mathrm{F}$ on the medium perturbation $\|n-n_0\|_{L^\infty}$. Different $\|n-n_0\|_{L^\infty}$ is obtained by tuning the amplitude $A$ in the medium~\eqref{eqn:medium_ex1}.}
  \label{fig:lamdba_vs_n}
\end{figure}

\section{Inversion Algorithm}\label{sec:numer_inverse}
The inverse problem that we study in this article has a different setup from the conventional one. While the conventional setup has either the concentration in the incoming direction, or in the incoming source location, our experimental setup requires concentration in both direction and source location. Naturally we expect a better stability in the reconstruction process, compared to the traditional formulation. In this section we showcase such stability.

Numerically the reconstruction process is formulated as a PDE-constrained minimization problem, where we seek to minimize the misfit between the data and the forward model:

\begin{equation}
    \min_{n} \left\| D - \Dc^k[n]  \right\|^2_{L^2(\Gamma_{-} \times \Gamma_{+})}\,,
\end{equation}
or equivalently, in the discretized form:
\begin{equation}\label{eq:minimization}
    \min_{n} \mathcal{J}[n], \qquad \text{where} \qquad \mathcal{J}[n]:= \frac{1}{2 n_{\mathtt{rcv}} n_{\mathtt{src}} }\sum_{i=1}^{n_{\mathtt{rcv}}} \sum_{j = 1}^{n_{\mathtt{src}}} \left | D^{i,j} - \left ( \Dc^k[n] \right) ^{i,j} \right|^2.
\end{equation}
In particular, $n_{\mathtt{rcv}}$ and $n_{\mathtt{src}}$ stand for the number of receivers and sources, and each point $(\Dc^k[n] ) ^{i,j}$ is the intensity squared of the impulse response generated by illuminating the medium $n$ with a tight beam  given by~\eqref{eqn:source} originated at $x_{\rm{s}}^{i}$ with direction  $v_{\rm{s}}^{i}$, which is then filtered using~\eqref{eqn:phi_form} centered at $x_{\rm{r}}^{j}$ with direction  $v_{\rm{r}}^{j}$. See definition in~\eqref{eqn:husimi_receive_num}, with $(x_{\rm{r}},v_{\rm{r}})$ replaced by $(x^j_{\rm{r}},v^j_{\rm{r}})$, and $\uk$ solving~\eqref{eqn:Helmholtz_numerics} with $(x_{\rm{s}},v_{\rm{s}})$ replaced by $(x^i_{\rm{s}},v^i_{\rm{s}})$.

We employ quasi-newton methods for finding a local minimum\footnote{Given that the problem is very non-linear, there is not guarantees that we can find the global minimum.}, thus we need to efficiently compute the gradient of the misfit function. In order to provide a fully self-contained exposition we briefly summarize below how to compute the gradient for only one data point using the adjoint-state methods. From there the computation for the full gradient can be easily deduced.

We can readily compute the application of the gradient to a perturbation $\delta n$ by using the chain rule, which results in
\begin{equation}
    \nabla \mathcal{J}[n] \delta n = \left( \frac{k}{2\pi}\right)^d \left ( D - \Hk \uk(x_\rmr,v_\rmr) \right)  \text{Real} \left ( 2 \overline{(\uk*\phi_{v_\rmr}^k(x_\rmr))} (\phi_{v_\rmr}^k(x_\rmr)* F[n] \delta n)\right)\,,
\end{equation}
where $F[n]$ is linearized forward wave-propagation operator, given by the Born approximation of the scattered wave-field~\cite{Bo:1926Quantenmechanik}. Thus the gradient can be easily computed by applying the adjoint of the Born approximation to the residual times the filter function, i.e.,
\[
    \nabla \mathcal{J}[n] =2\left( \frac{k}{2\pi}\right)^d \text{Real} \left ( F[n]^* \left ( \big ( D - \Hk \uk(x_\rmr,v_\rmr) \big )    (\uk*\phi_{v_\rmr}^k(x_\rmr)) \overline{(\phi_{v_\rmr}^k(x_\rmr - x)  )}\right )\right ).
\]
Fortunately, the application of the adjoint of the Born approximation operator is well studied: it can be performed by solving the adjoint equation followed by a multiplication by the solution of the forward wave problem\footnote{We redirect the interest readers to~\cite{BoGiGr:2017high} for a modern self-contained presentation.}. In this case the adjoint equation is the same Helmholtz equation, but with adjoint Sommerfeld radiation conditions, i.e., we solve
\begin{equation}\label{eqn:adj_scattering}
\begin{aligned}
\Delta v + k^2 n(x) v &=  \left ( D - \Hk \uk(x_\rmr,v_\rmr) \right)    (\uk*\phi_{v_\rmr}^k(x_\rmr)) \overline{(\phi_{v_\rmr}^k(x_\rmr - x)  )} \quad x \in \Rb^d \,, \\
\frac{\partial v}{\partial r} + \ri k v &= \Oc(r^{-(d+1)/2}) \text{ as } r=|x|\to\infty \,.
\end{aligned}
\end{equation}
Thus, using~\eqref{eqn:adj_scattering}, we can easily compute the application of the adjoint of the Born approximation
\begin{equation} \label{eq:grad_min_problem}
     F[n]^* \left ( \big ( D - \Hk \uk(x_\rmr,v_\rmr) \big )    (\uk*\phi_{v_\rmr}^k(x_\rmr)) \overline{(\phi_{v_\rmr}^k(x_\rmr - x)  )}\right ) = \overline{\uk} v.
\end{equation}
where $v$ solves~\eqref{eqn:adj_scattering}.

We point out that in \eqref{eq:grad_min_problem}, the source for the adjoint is conjugated, thus following \eqref{eqn:phi_form}, we can see that it means that the $\overline{(\phi_{v_\rmr}^k(x - x_\rmr))}$ is pointing towards the interior of the domain in direction $-v_\rmr$.

We solve~\eqref{eq:minimization} using L-BFGS~\cite{ByLuNoZh:1995limited,ZhByLuNo:1997algorithm}, a quasi-Newton method in Matlab. We consider the initial perturbation equal to zero. We set a first order optimality tolerance of $10^{-5}$ and let the algorithm run for a maximum of $300$ iterations or until the tolerance is achieved.

To avoid the inverse crime~\cite{CoKr:2019inverse}, the data is generated by solving the Lippmann-Schwinger equation discretized by the truncated kernel method~\cite{ViGeFe:2016fast} as in Section~\ref{sec:numer}, and the inversion is performed with an $4$th-order finite difference scheme for both~\eqref{eqn:Helmholtz_numerics} and~\eqref{eqn:adj_scattering}. To generate the data, we set the computational domain to be $K = [-1,1]^2$ with $N_{\mathrm{LS}} = 256^2 = 65536$ grid points so that there are at least $12$ points per wavelength for the largest $k = 2^6$. In the inversion, we discretize the same domain $K$ with $N_{\mathrm{FD}} = 163^2 = 26569$ grid points so that there are at least 8 points per wavelength for $k = 2^6$. We enclose the domain $K$ with perfect matching layer (PML) to avoid reflection. We choose the thickness of PML to be $2.5$ times wavelength.

The measurement is taken on $\partial B(R)$ with $R = 0.4$ in all the examples. To generate the probing ray, we set $\sigma = 2^{-2}$ in~\eqref{eqn:numer_source}. We compute the data with the source position and incident direction
\begin{align*}
&x_\rms^{i_1} = (R\cos\theta_\rms^{i_1},R\sin\theta_\rms^{i_1}) \\
&v_\rms^{i_1,i_2} = (-\cos(\theta_\rms^{i_1}+\theta_\rmi^{i_2}),-\sin(\theta_\rms^{i_1}+\theta_\rmi^{i_2}))
\end{align*}
where $\theta_\rms^{i_1} = \pi + i_1\frac{\pi}{48}$ for all $i_1 = 0,\dots, 95$ and $\theta_\rmi^{i_2} = -\frac{\pi}{2} + i_2\frac{\pi}{49}$ for all $i_2 = 1,\dots,48$, and the receiver position
\begin{align*}
&x_\rmr^{j_1} = (R\cos\theta_\rmr^{j_1},R\sin\theta_\rmr^{j_1}) \\
&v_\rmr^{j_1,j_2} = (\cos(\theta_\rmr^{j_1}+\theta_\rmo^{j_2}),\sin(\theta_\rmr^{j_1}+\theta_\rmo^{j_2}))
\end{align*}
where $\theta_\rmr^{j_1} = j_1\frac{\pi}{48}$ for all $j_1 = 0,\dots, 95$ and $\theta_\rmo^{j_2} = -\frac{\pi}{2} + j_2\frac{\pi}{49}$ for all $j_2 = 1,\dots,48$.

In all the examples, the scattered data is perturbed with the noise in the form
\begin{equation}\label{eqn:data_noise}
    \widetilde{D}^{i,j}  = D^{i,j} + 0.05 \eps \frac{D^{i,j}}{|D^{i,j}|}
\end{equation}
where $\eps$ is symmetric Bernoulli random variable that takes the values $\pm1$.

All the experiments are reported on a server with 64-core Intel Xeon CPU and 256 Gigabytes RAM. The code accompanying this manuscript are publicly available~\cite{ChDiLiZe:2021Husimi}.

In order to illustrate the reconstruction using Husimi data, we choose three examples of increasing complexity. The exact contrast function $q(x)$'s are shown in Figure~\ref{fig:inverse_truth}.
\begin{figure}[htbp]
  \centering
  \includegraphics[width=0.3\textwidth]{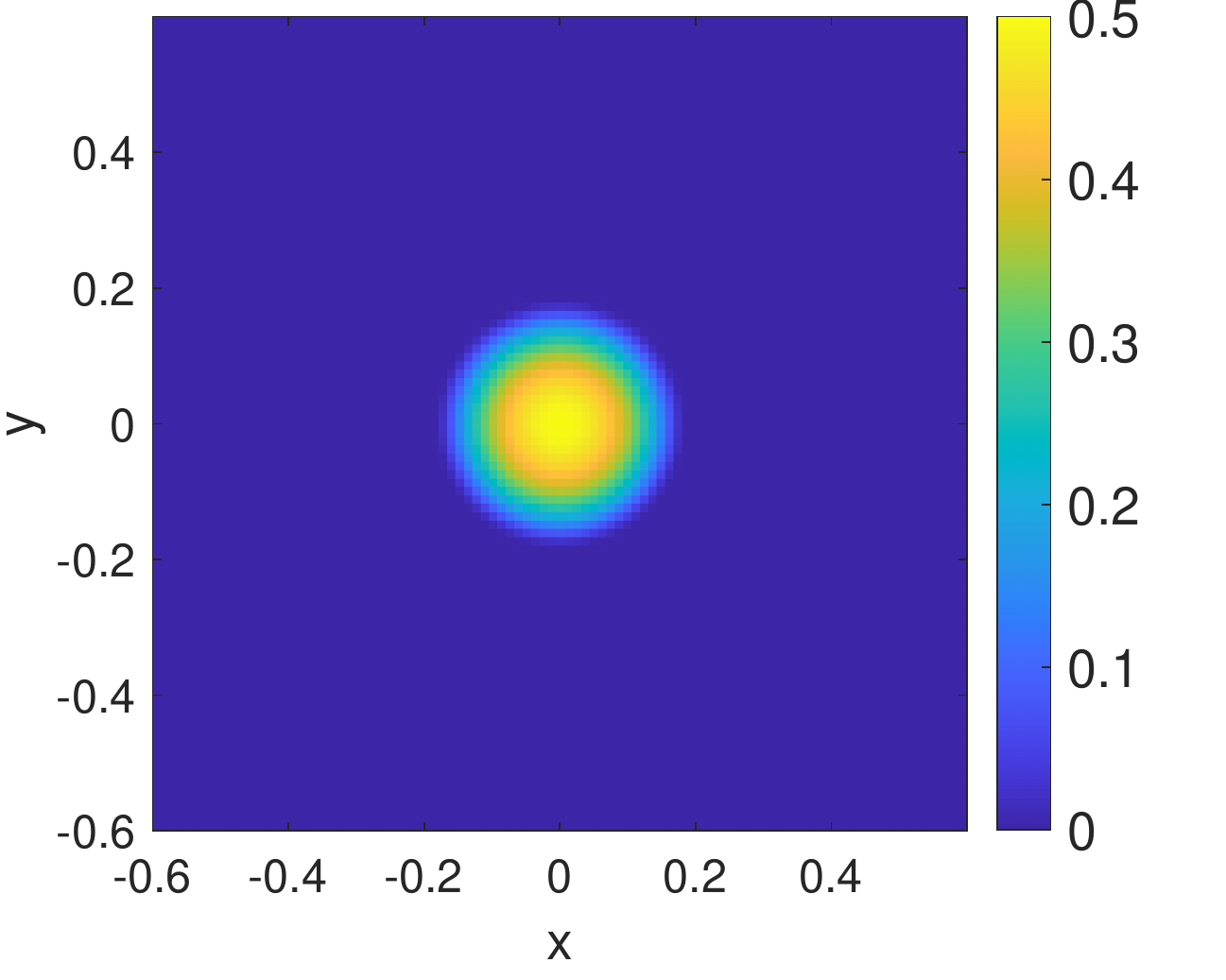}
  \includegraphics[width=0.3\textwidth]{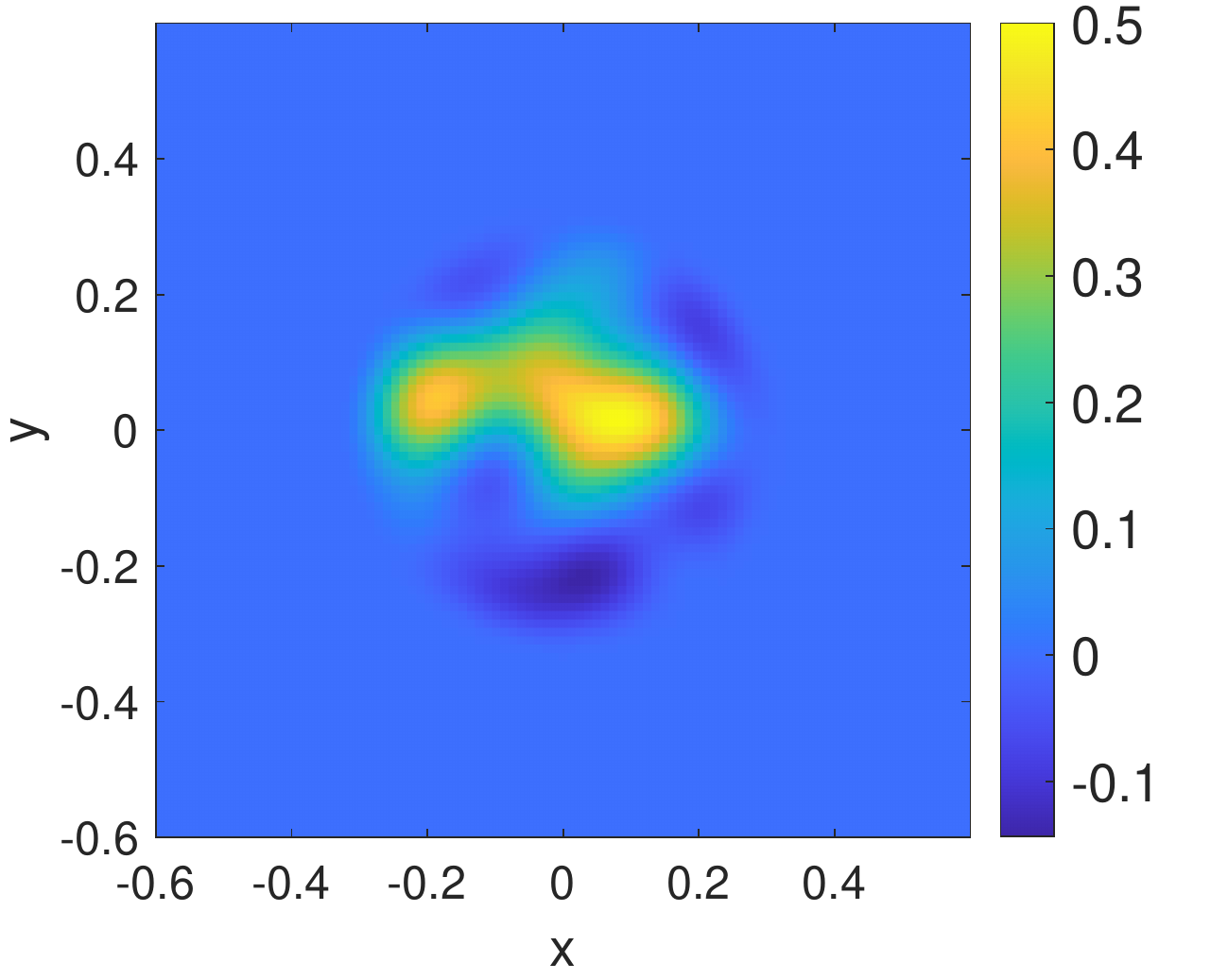}
  \includegraphics[width=0.3\textwidth]{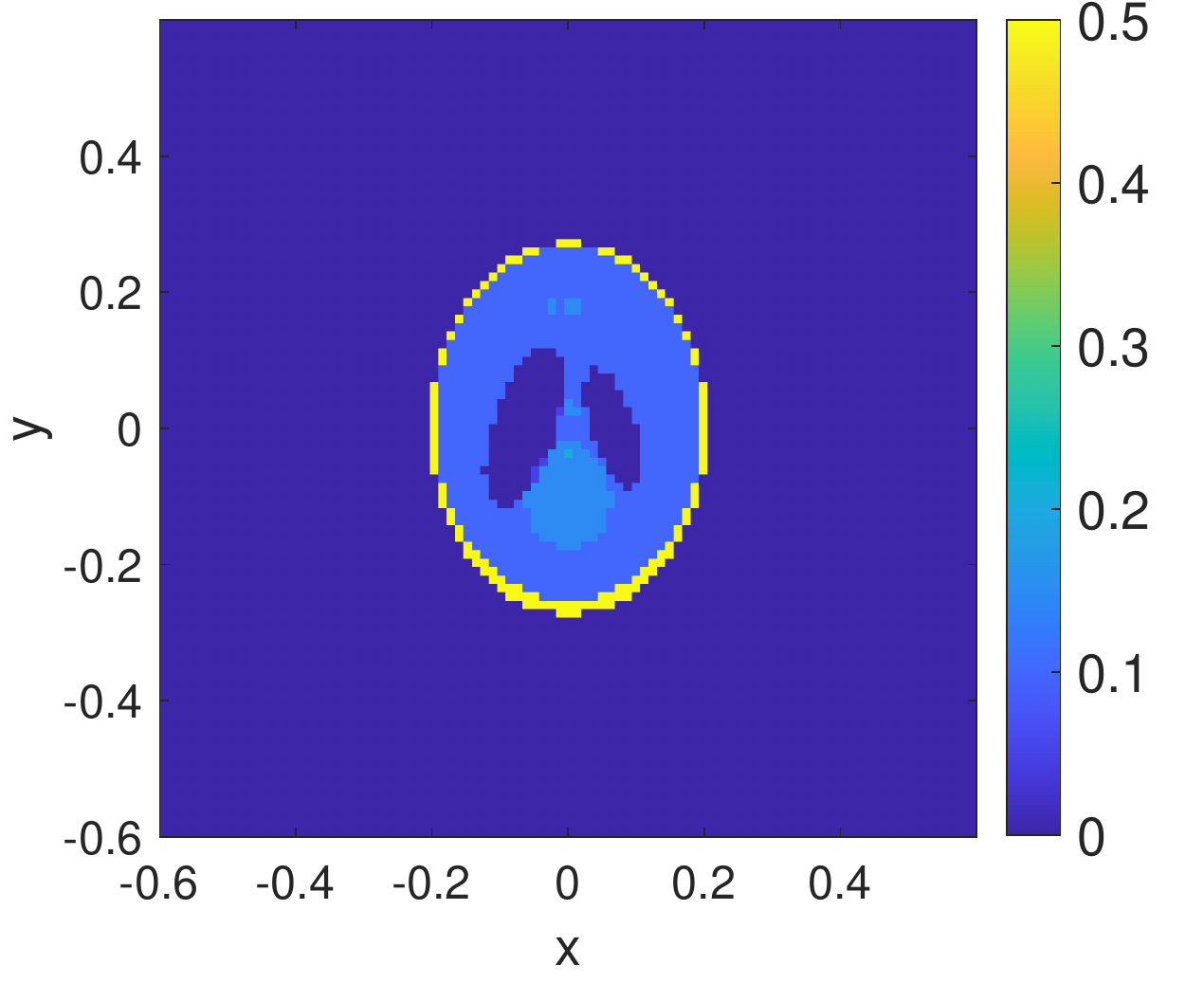}
  \caption{The contrast function $q(x)$ for our three examples: a bump function (left), a delocalized function (middle) and the Shepp-Logan phantom (right).
  }
  \label{fig:inverse_truth}
\end{figure}

In the first example, we consider a single bump in the form~\eqref{eqn:medium_ex1} with $A = 0.5$ and $r = 0.2$, which is shown in Figure~\ref{fig:inverse_truth} (left). We run the minimization loop as described above using $k = 2^4$ and $k = 2^6$, and the resulting reconstruction are shown in Figure~\ref{fig:inverse_bump_LS}. From Figure~\ref{fig:inverse_bump_LS} we can clearly see that as $k$ becomes larger, the reconstruction becomes closer to the true medium. The solution time for $k=2^6$ is 15787.1 seconds.
\begin{figure}[htbp]
  \centering
  \includegraphics[width=0.4\textwidth]{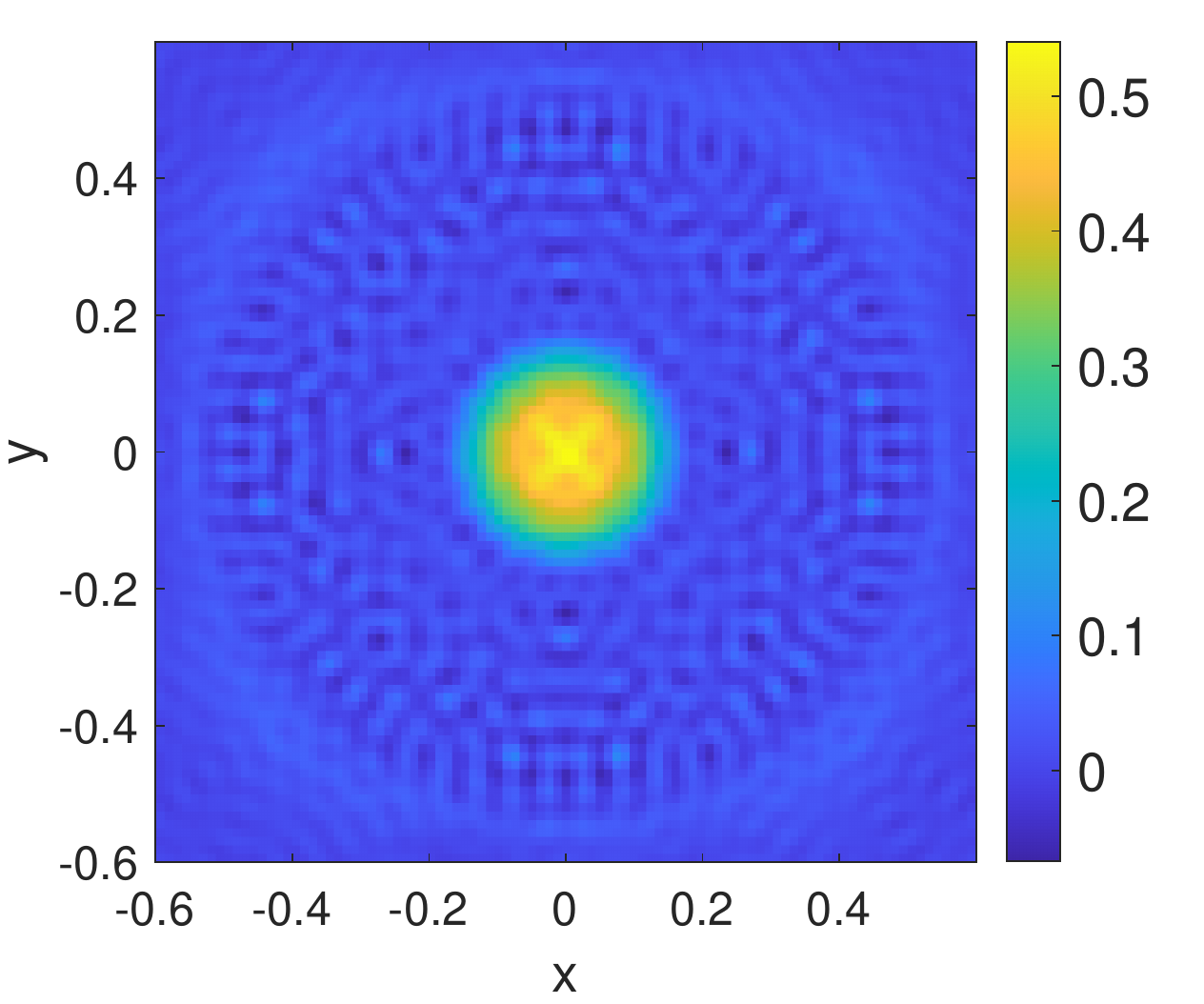}
  \includegraphics[width=0.4\textwidth]{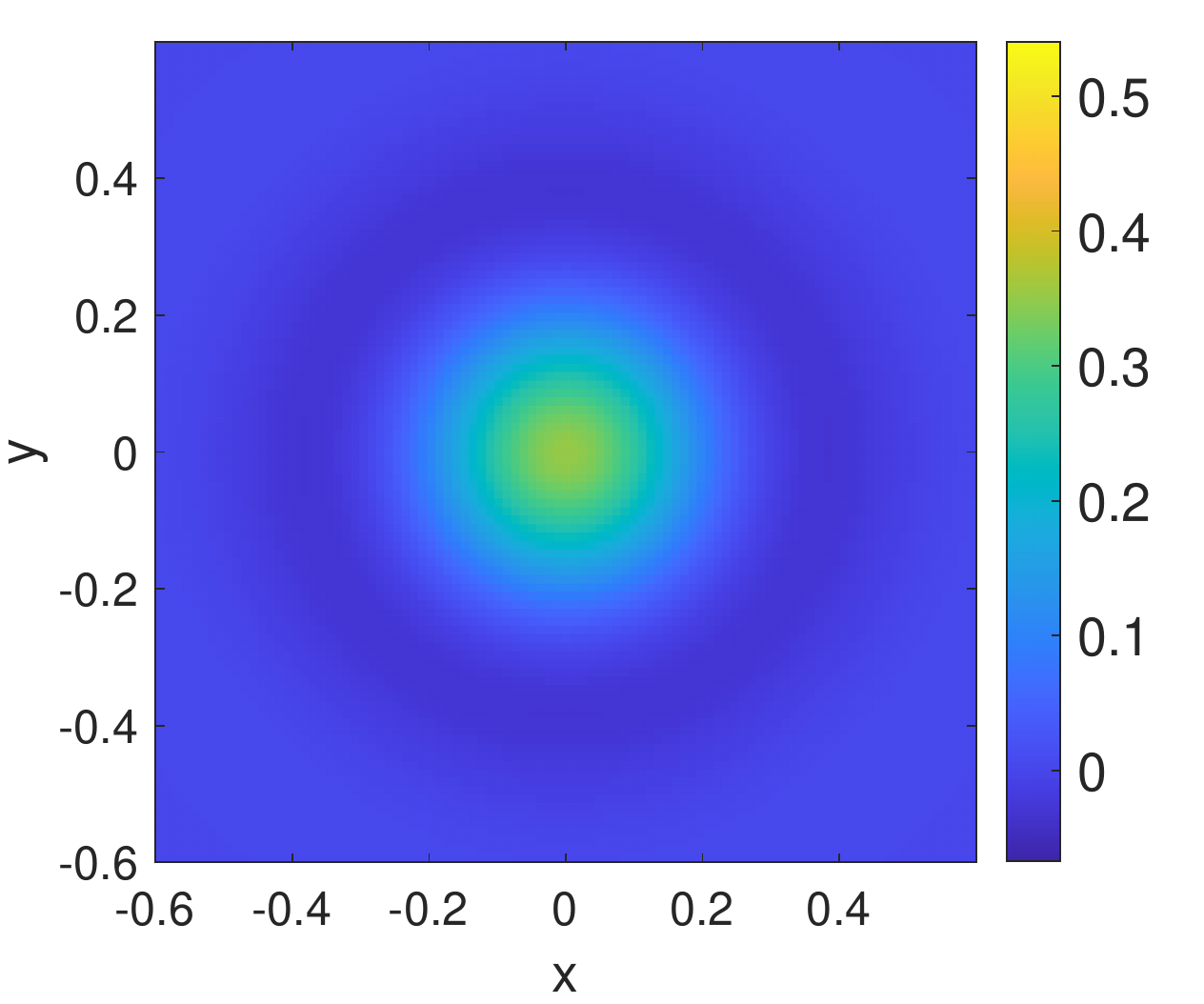}
  \\
  \includegraphics[width=0.4\textwidth]{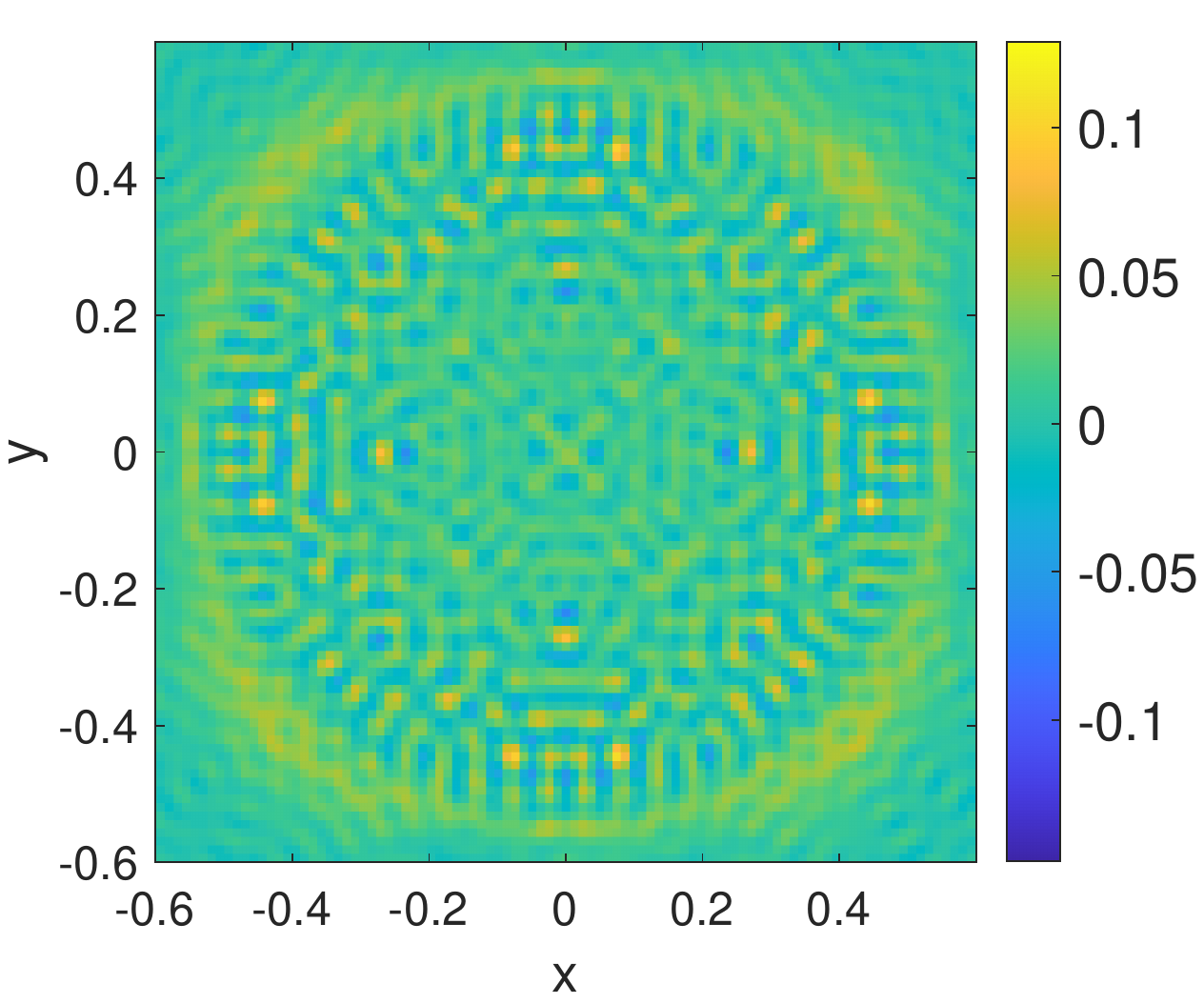}
  \includegraphics[width=0.4\textwidth]{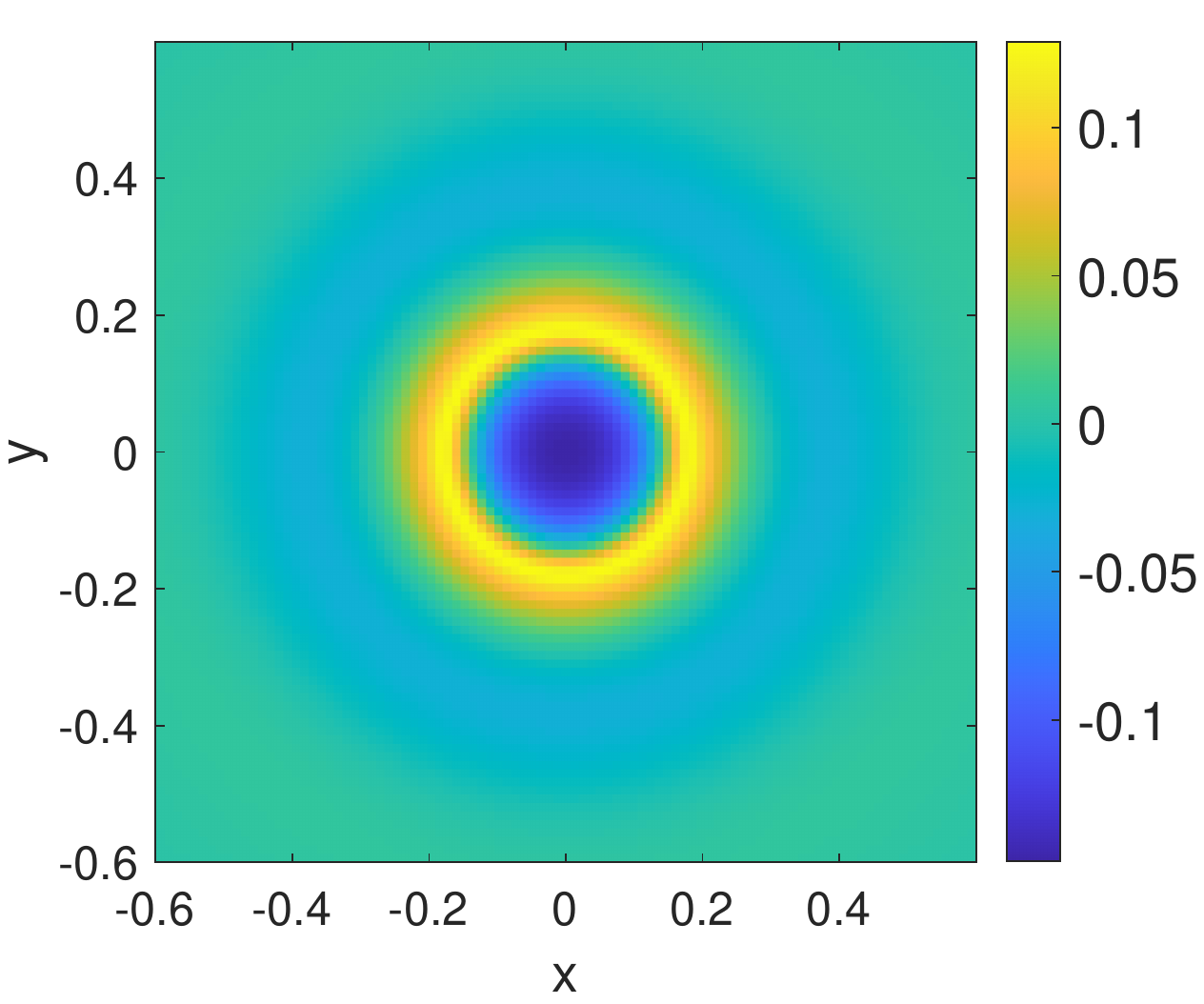}

  \caption{Recovering a single bump contrast function. The upper row shows the estimated contrast function and the lower row shows the reconstruction error at $k=2^6$ (left) and $k=2^4$ (right).
  }
  \label{fig:inverse_bump_LS}
\end{figure}

In the second example, we consider a delocalized medium. The delocalized contrast function $q(x)$ is obtained by convolving a pointwise independent Gaussian random field with a Gaussian mollifier. The main difference with the single bump example is that the refractive index, can be smaller than the background one, thus allowing for more complex ray paths as shown in Figure~\ref{fig:inverse_truth} (center). We repeat the same experiments, whose results are shown in Figure~\ref{fig:inverse_delocal_LS}. The solution time required for $k=2^6$ is 13185.3 seconds.
\begin{figure}[htbp]
  \centering
  \includegraphics[width=0.4\textwidth]{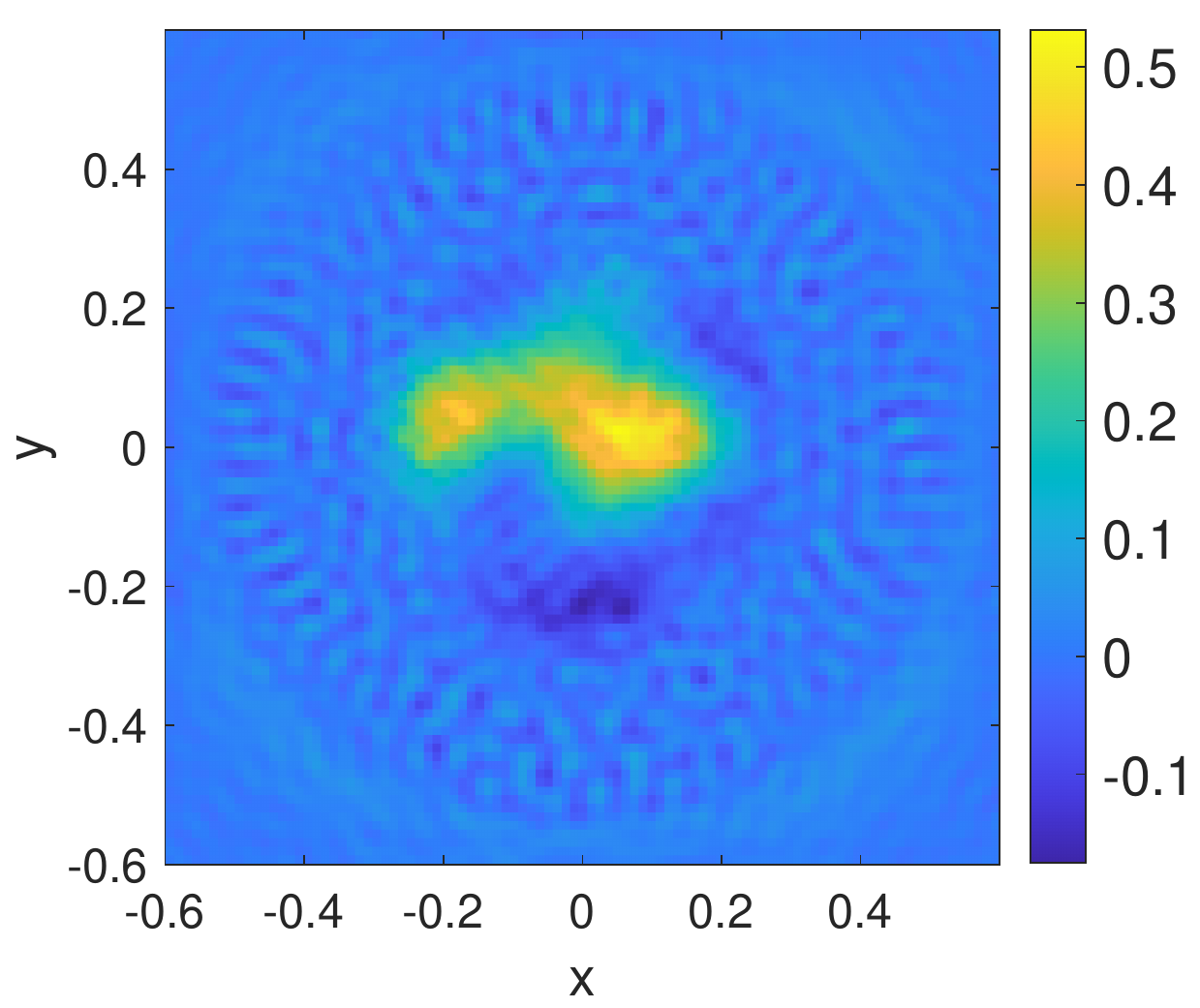}
  \includegraphics[width=0.4\textwidth]{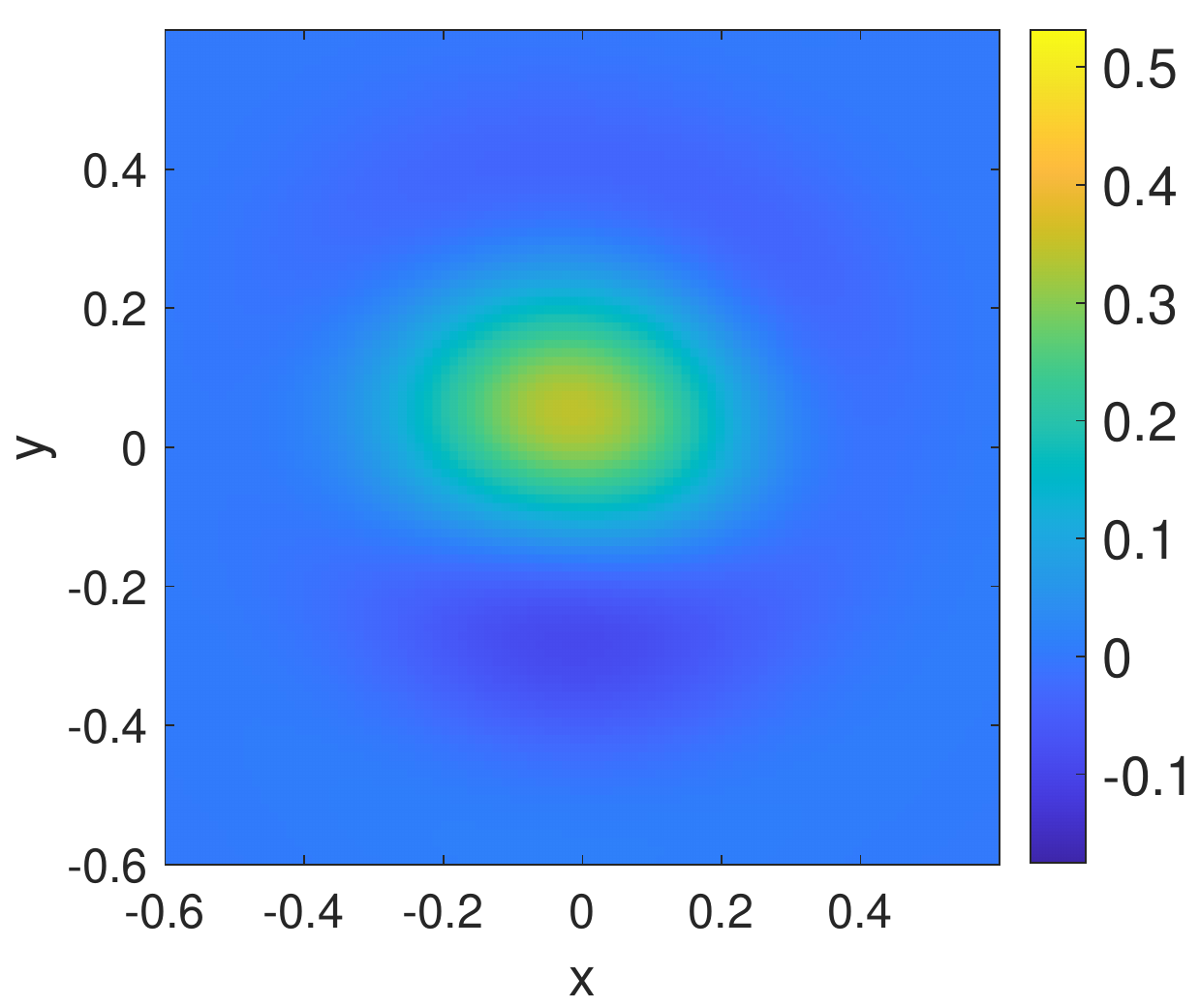}
  \\
  \includegraphics[width=0.4\textwidth]{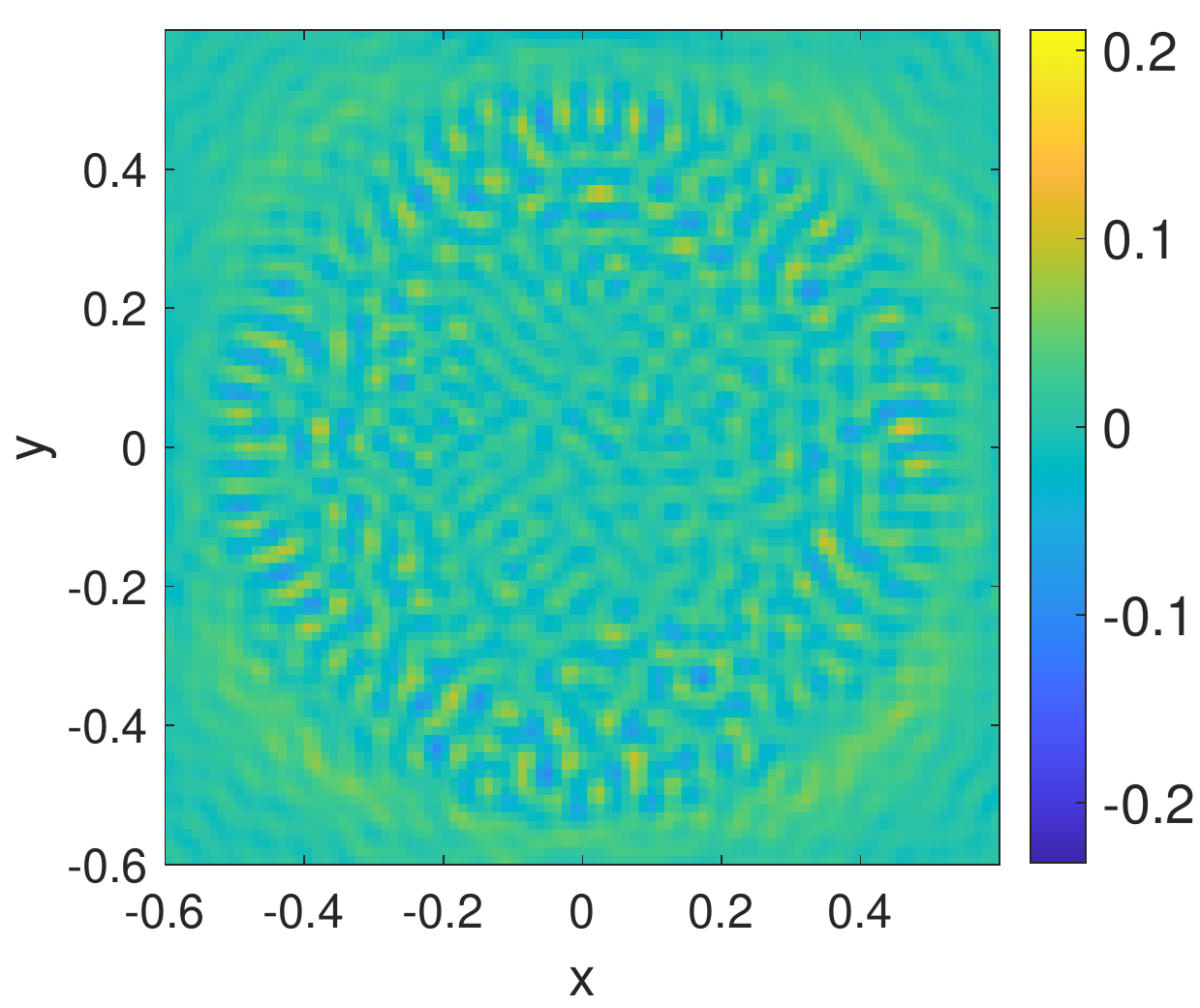}
  \includegraphics[width=0.4\textwidth]{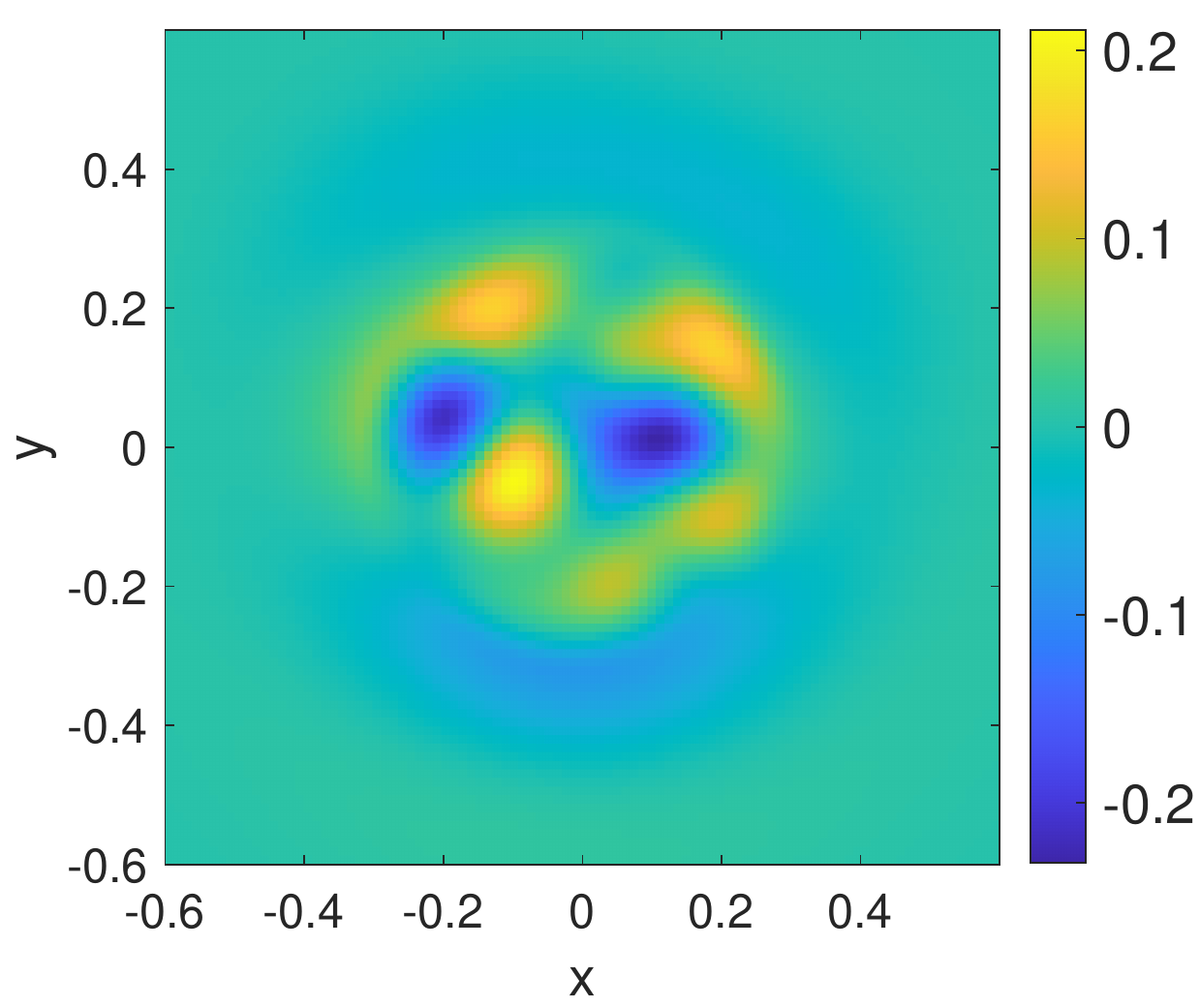}

  \caption{Recovering a delocalized contrast function. The upper row shows the estimated contrast function and the lower row shows the reconstruction error at $k=2^6$ (left) and $k=2^4$ (right).
  }
  \label{fig:inverse_delocal_LS}
\end{figure}

Finally, for the third example, we consider the more challenging, and more practical, problem of recovering the Shepp-Logan phantom, depicted in Figure~\ref{fig:inverse_truth} (right). In this case we have very sharp transitions of the refractive index, which will generate a strong reflection, compared to the refraction-dominated media considered before. In addition, the interior of the still acts as a resonant cavity, thus creating a large amount of interior reflections, which are exacerbated as the frequency increases. We perform the same experiments as above, whose results are depicted in in Figure~\ref{fig:inverse_shepplogan_LS}. The solution time required for $k=2^6$ is 14640.4 seconds. In this case, the reconstruction is qualitative worse than before. We can still see the shape of the phantom, but with a large amount of artifacts. These artifacts are common to the three examples, but are somewhat more notorious for the Shepp-Logan phantom. Indeed, these artifacts can be in part explained by the large difference in the dispersion relation between the forward and backwards discretizations. The Lippmann-Schwinger discretization used for the forward problem is known to be highly accurate if the media is smooth. In the cases before, the data generated by the Lippmann-Schwinger solver is close to the analytical solution, and the artifacts seems to come mostly for the phase errors in the Finite-Difference discretization. However, in this case the phantom is discontinuous thus creating large phase errors in the solution of the equation, and therefore the forward map, which in return produce more notorious artifacts.
\begin{figure}[htbp]
  \centering
  \includegraphics[width=0.4\textwidth]{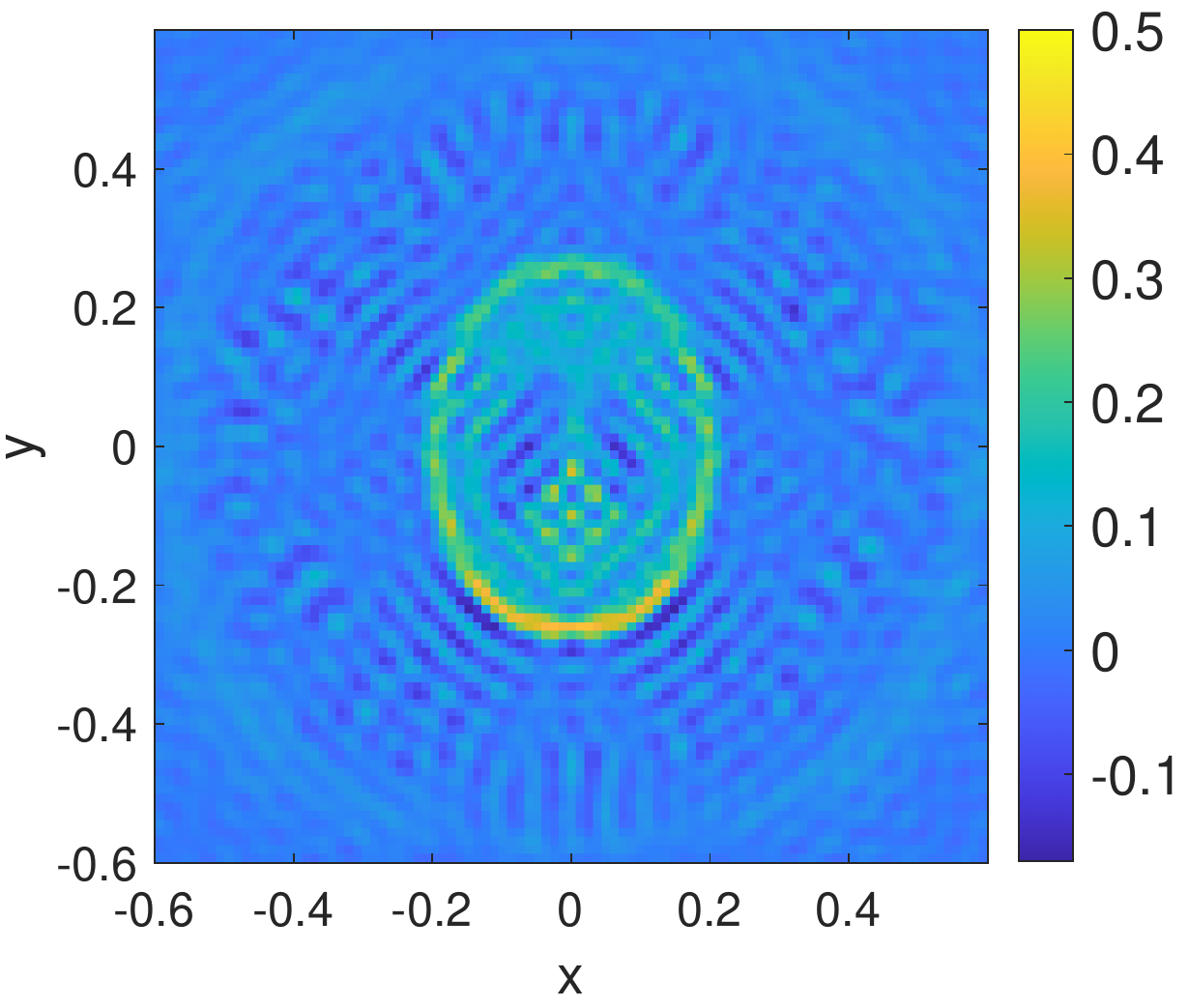}
  \includegraphics[width=0.4\textwidth]{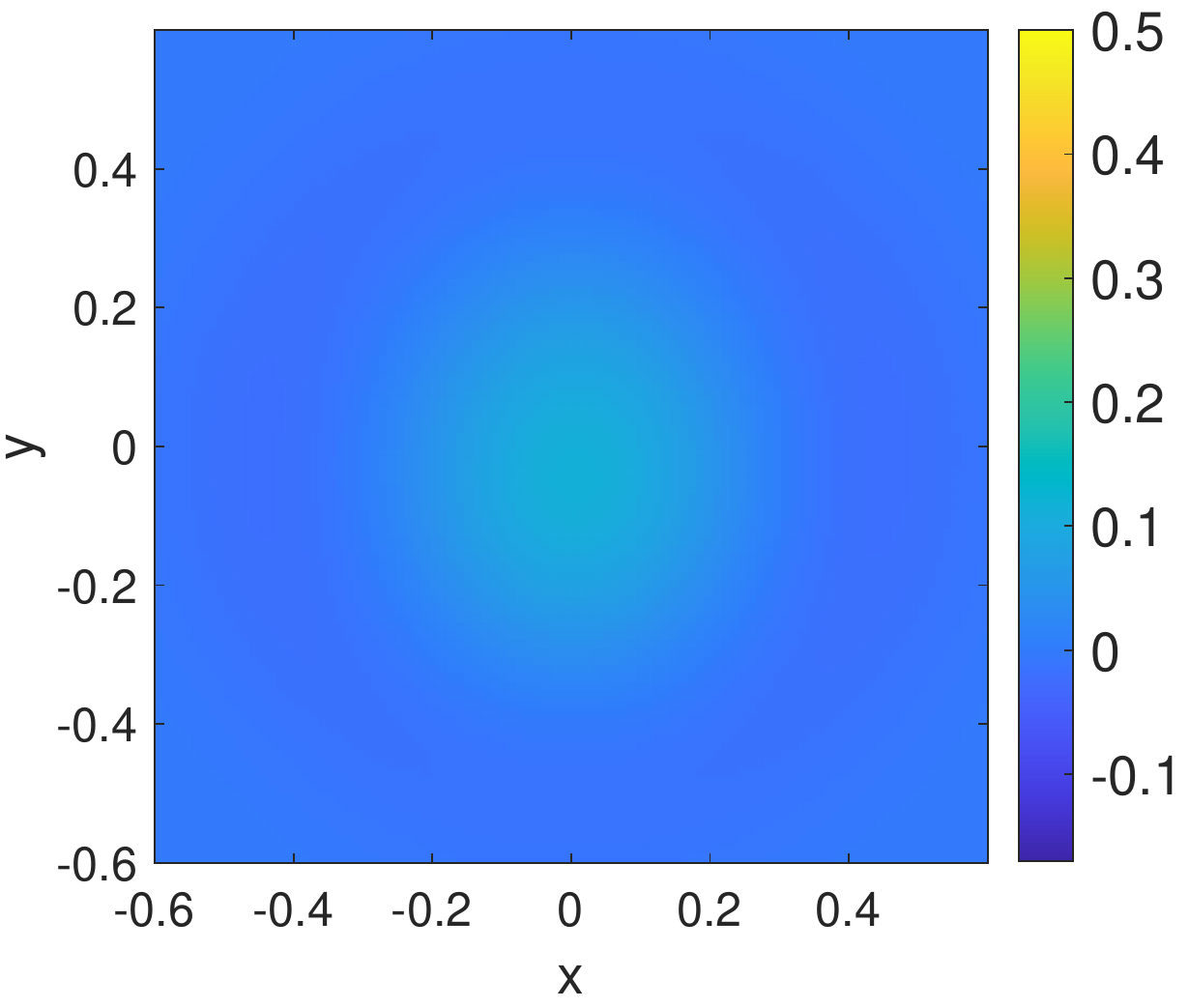}
  \caption{Recovering the Shepp-Logan phantom. The estimated contrast functions are shown for $k=2^6$ (left) and $k=2^4$ (right).
  }
  \label{fig:inverse_shepplogan_LS}
\end{figure}

To avoid inverse crime, we have used two different solvers for computing the equation. The two solvers produce relatively large phase errors that propagate in the reconstruction. The reconstruction can be significantly improved if we use the same PDE solvers in generating data and reconstructing the media. In Figure~\ref{fig:inverse_bump_FD}, we show the reconstructions of the same single bump medium as in Figure~\ref{fig:inverse_bump_LS} but with the $4$th-order finite difference for both data generation and inversion. It can be seen that the artifacts in the estimated medium is much smaller for larger $k$ and the reconstructed medium achieves a relative $L^2$ error of $0.0389$ for $k=2^6$. Better reconstruction can also be seen in Figure~\ref{fig:inverse_delocal_FD} for the reconstructed delocalized medium, whose relative $L^2$ error is $0.0341$ for $k=2^6$. In Figure~\ref{fig:inverse_shepplogan_FD}, we show the reconstruction for the Shepp-Logan phantom. We can observe that as the frequency increased the reconstruction becomes better, although due to computational limitations induced by the current implementation we were unable to test with a higher frequency, however, we would expect to obtain even a better reconstruction.
\begin{figure}[htbp]
  \centering
  \includegraphics[width=0.4\textwidth]{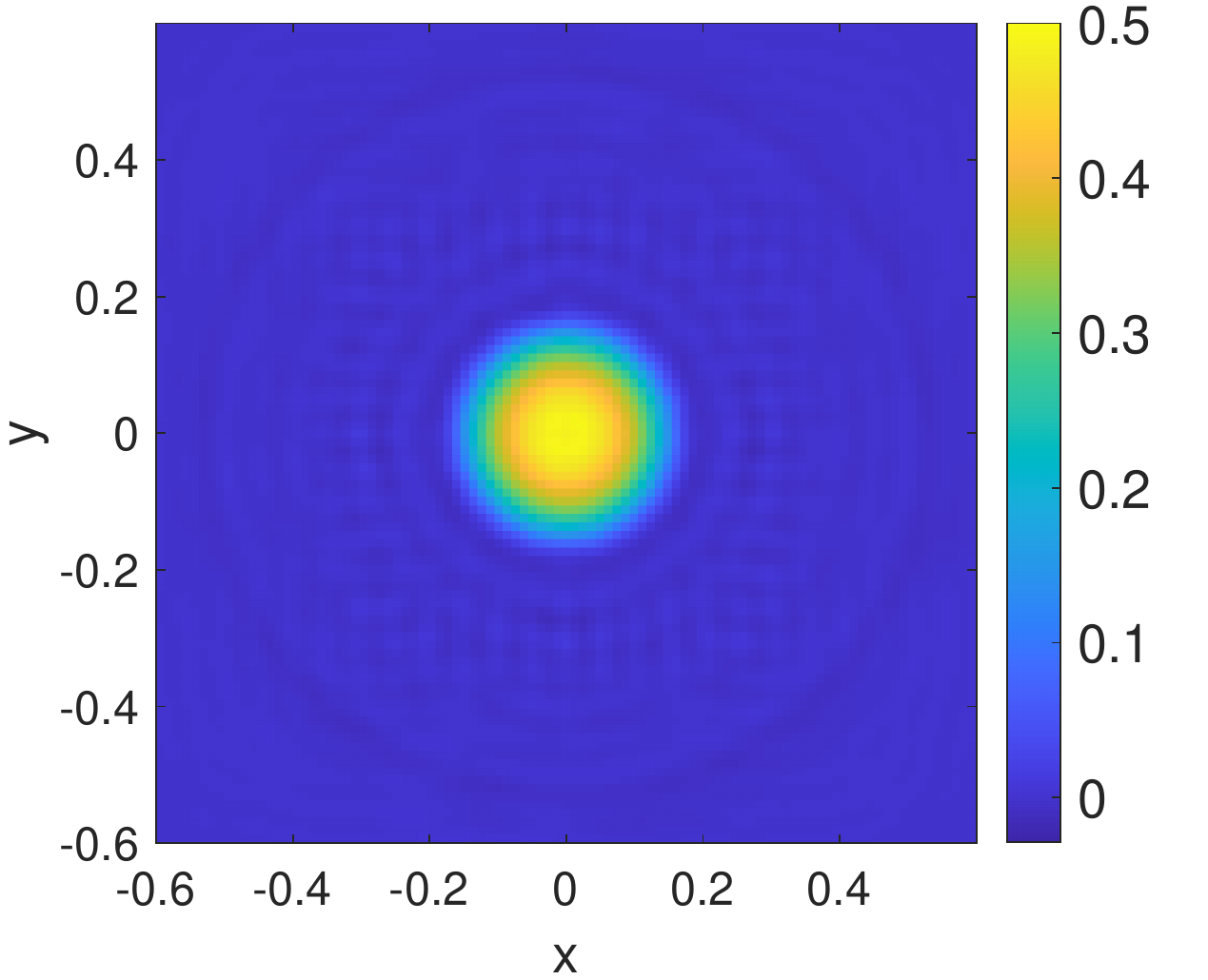}
  \includegraphics[width=0.4\textwidth]{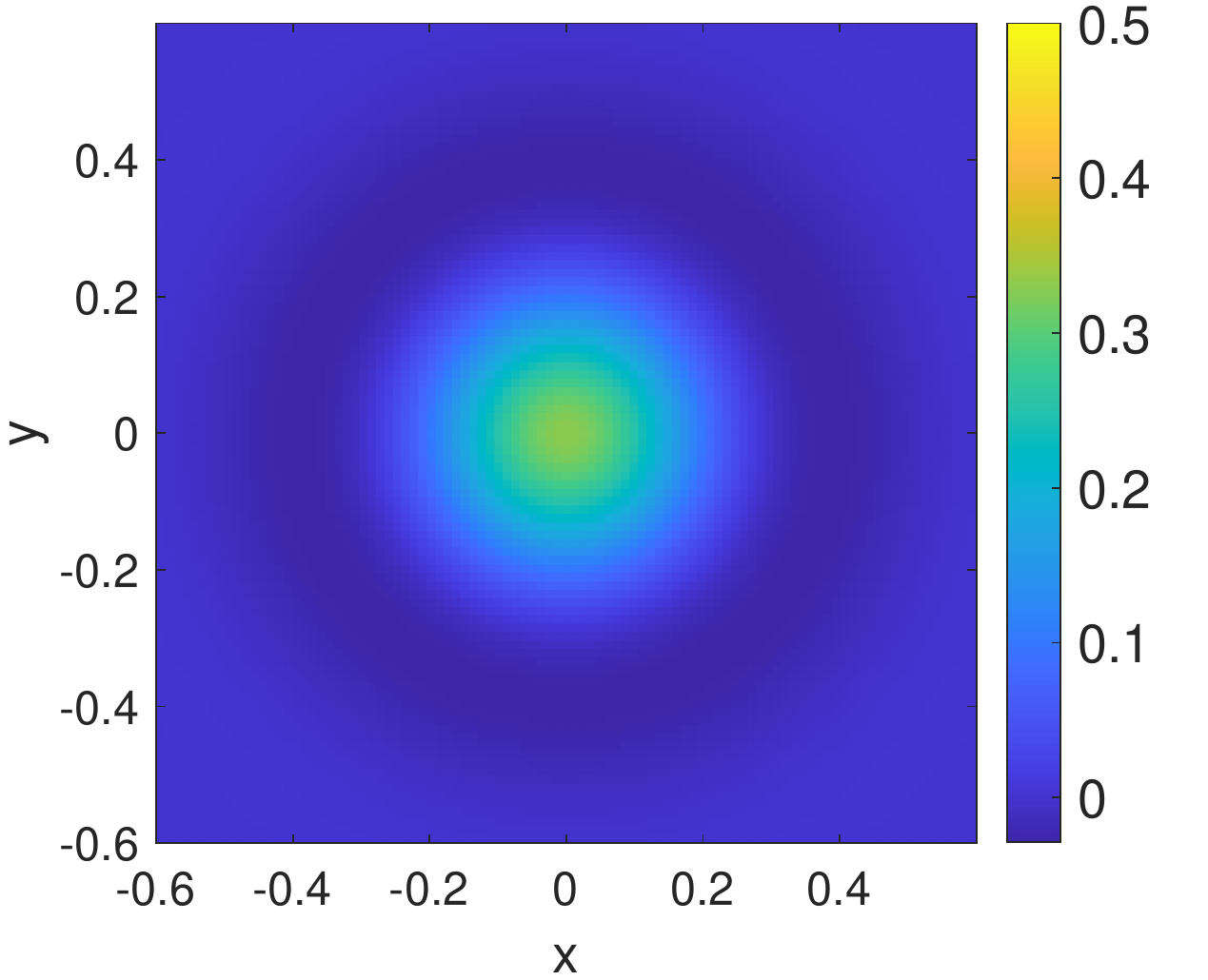}
  \\
  \includegraphics[width=0.4\textwidth]{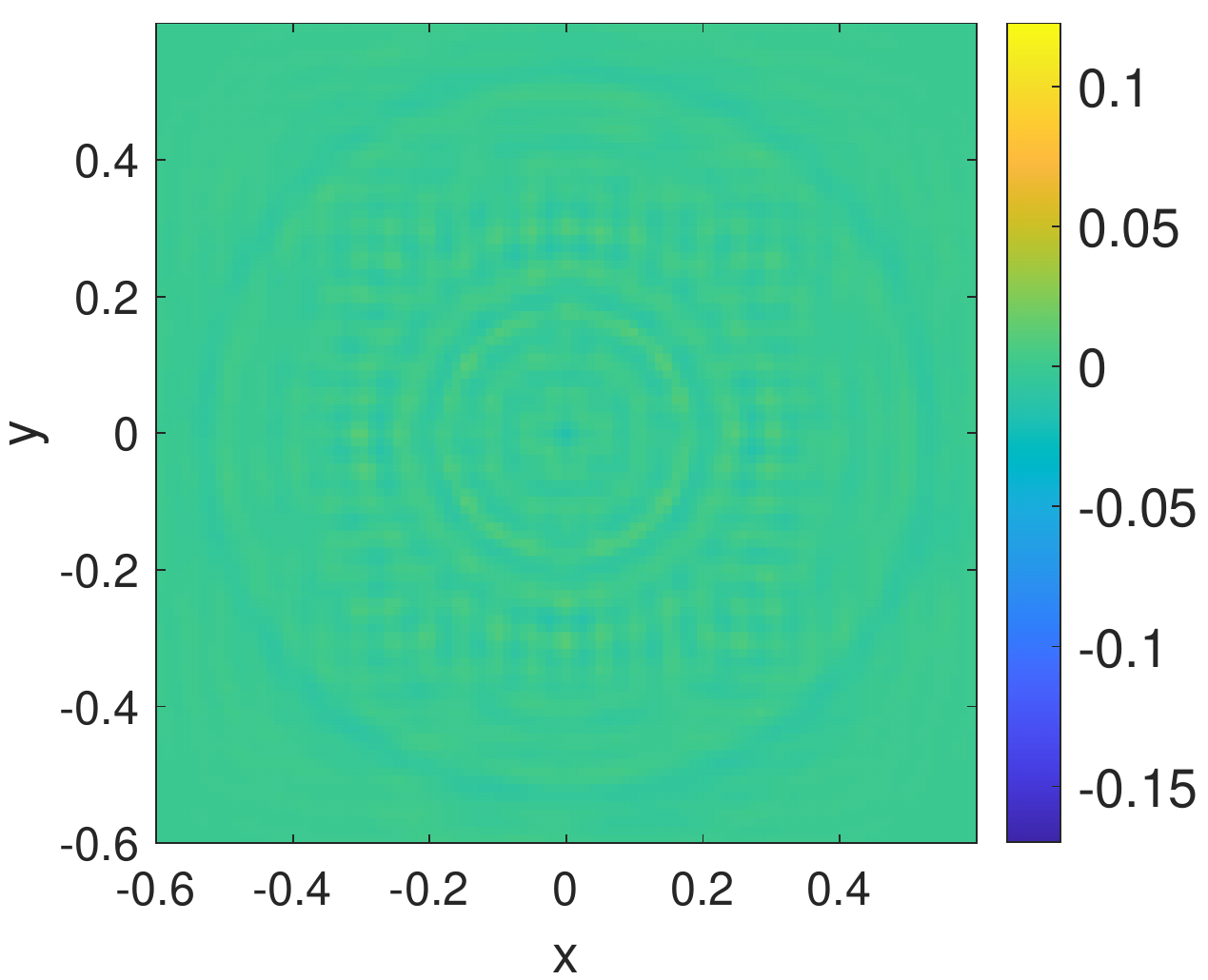}
  \includegraphics[width=0.4\textwidth]{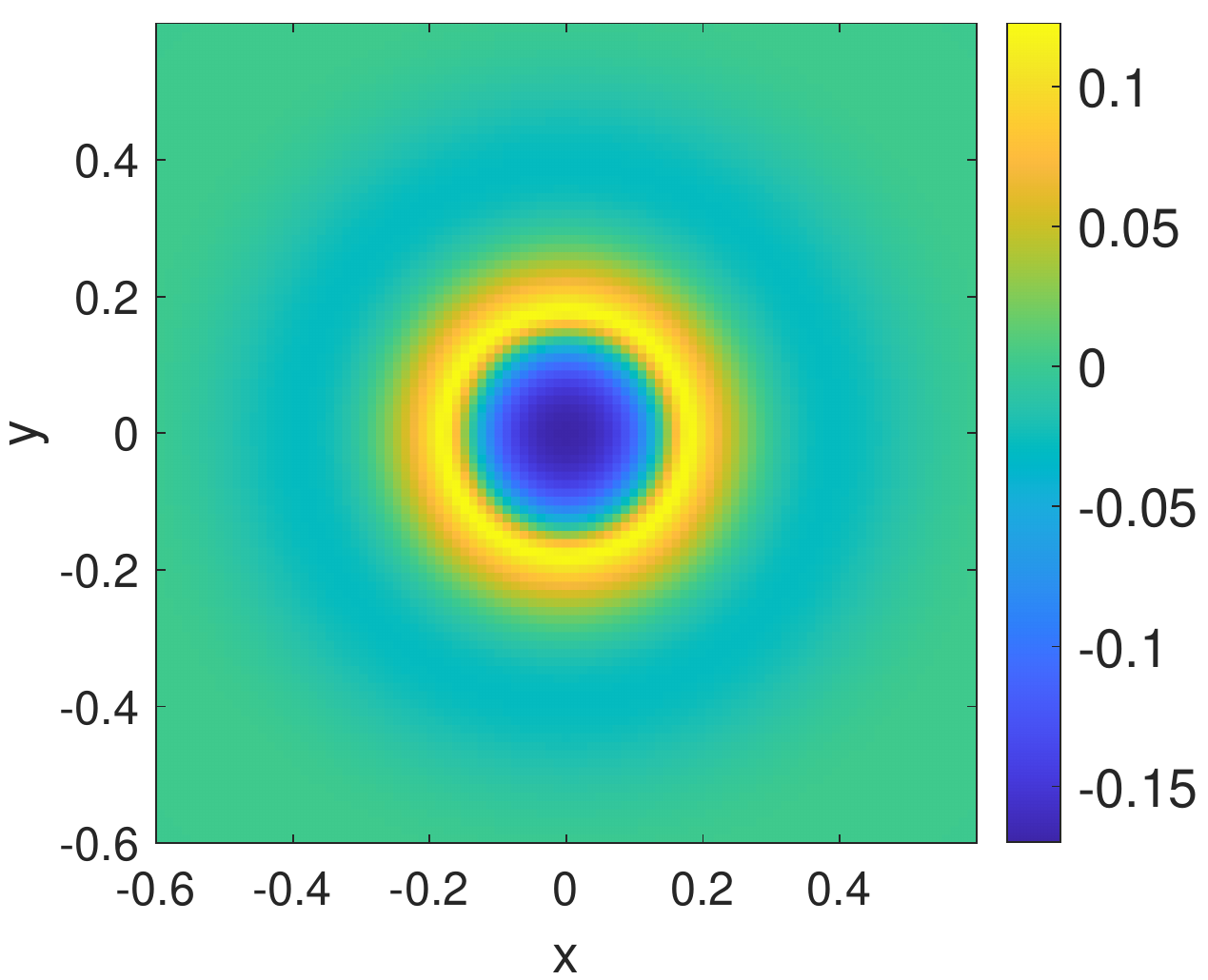}

  \caption{Recovering a single bump contrast function with $4$th-order finite difference solver for both data and inversion. The upper row shows the estimated contrast function and the lower row shows the reconstruction error at $k=2^6$ (left) and $k=2^4$ (right).
  }
  \label{fig:inverse_bump_FD}
\end{figure}

\begin{figure}[htbp]
  \centering
  \includegraphics[width=0.4\textwidth]{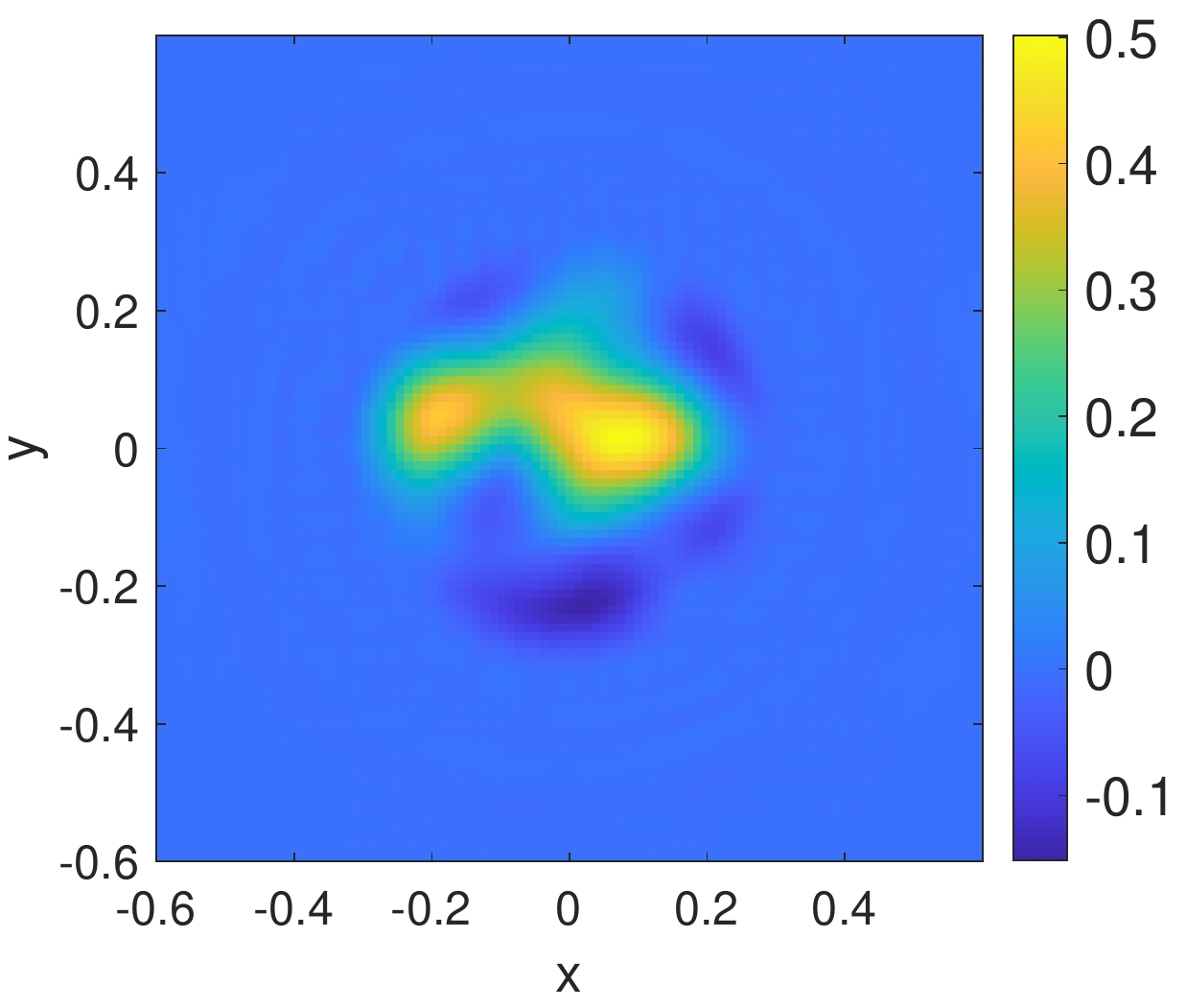}
  \includegraphics[width=0.4\textwidth]{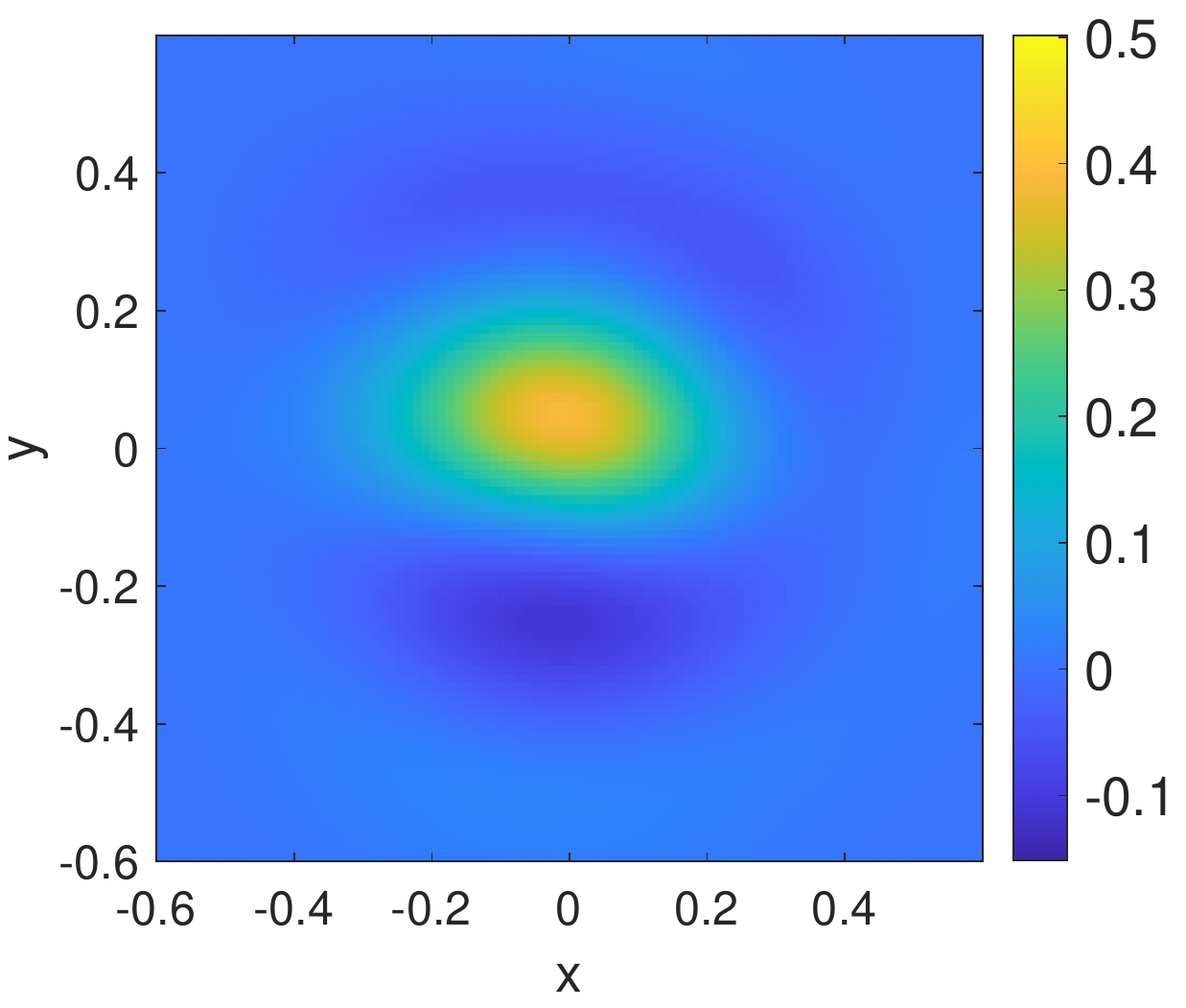}
  \\
  \includegraphics[width=0.4\textwidth]{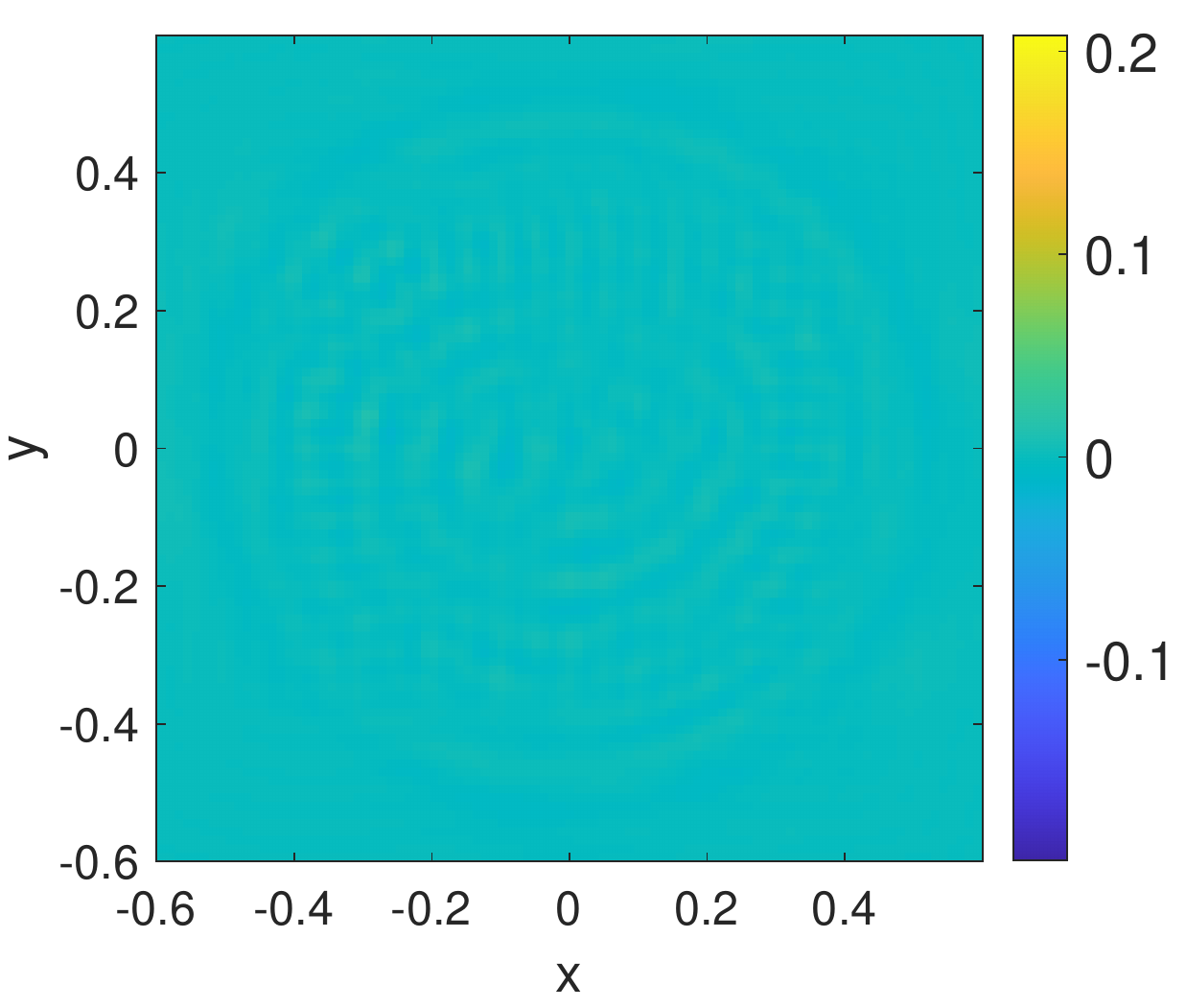}
  \includegraphics[width=0.4\textwidth]{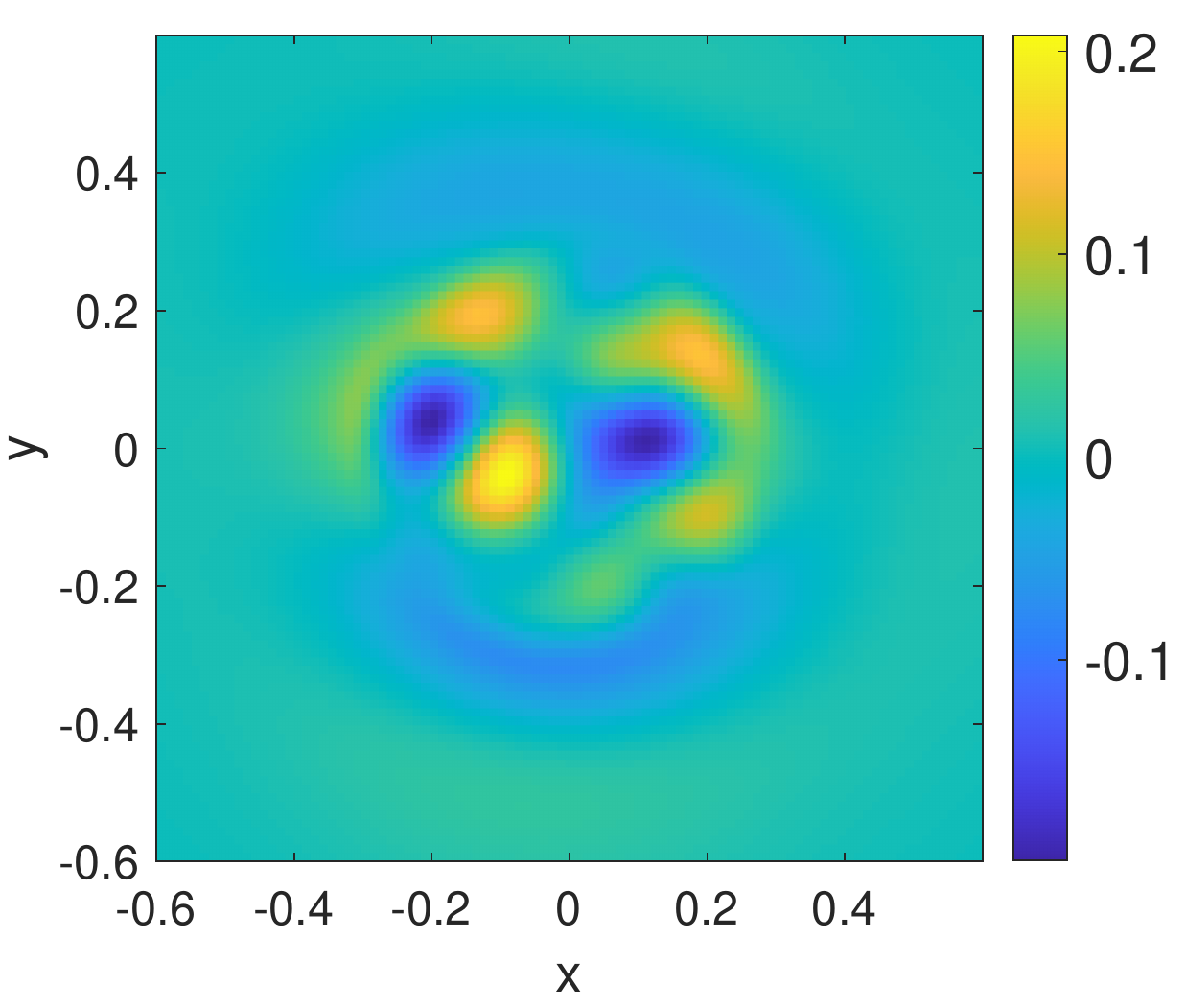}

  \caption{Recovering a delocalized contrast function with $4$th-order finite difference solver for both data and inversion. The upper row shows the estimated contrast function and the lower row shows the reconstruction error at $k=2^6$ (left) and $k=2^4$ (right).
  }
  \label{fig:inverse_delocal_FD}
\end{figure}

\begin{figure}[htbp]
  \centering
  \includegraphics[width=0.4\textwidth]{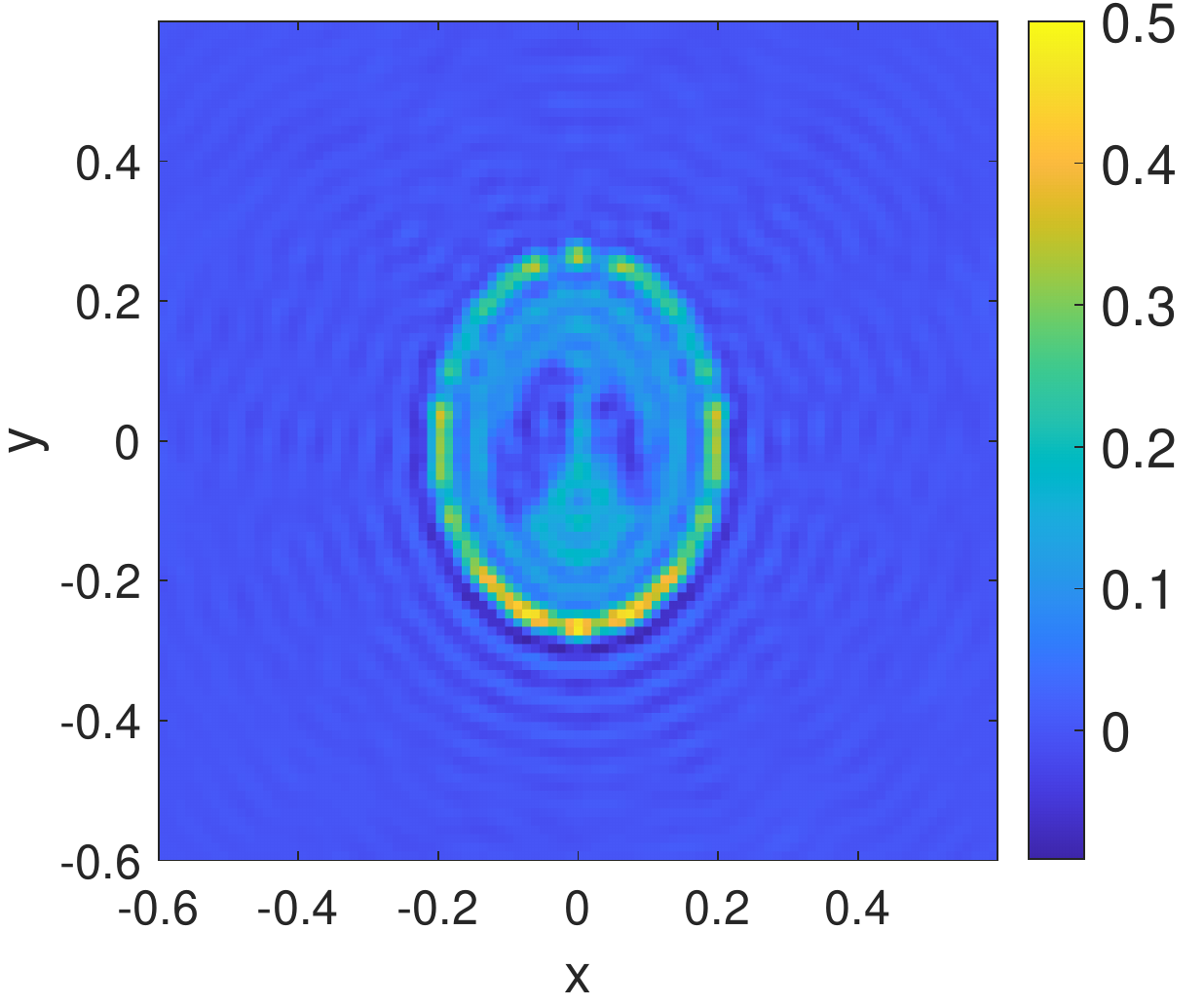}
  \includegraphics[width=0.4\textwidth]{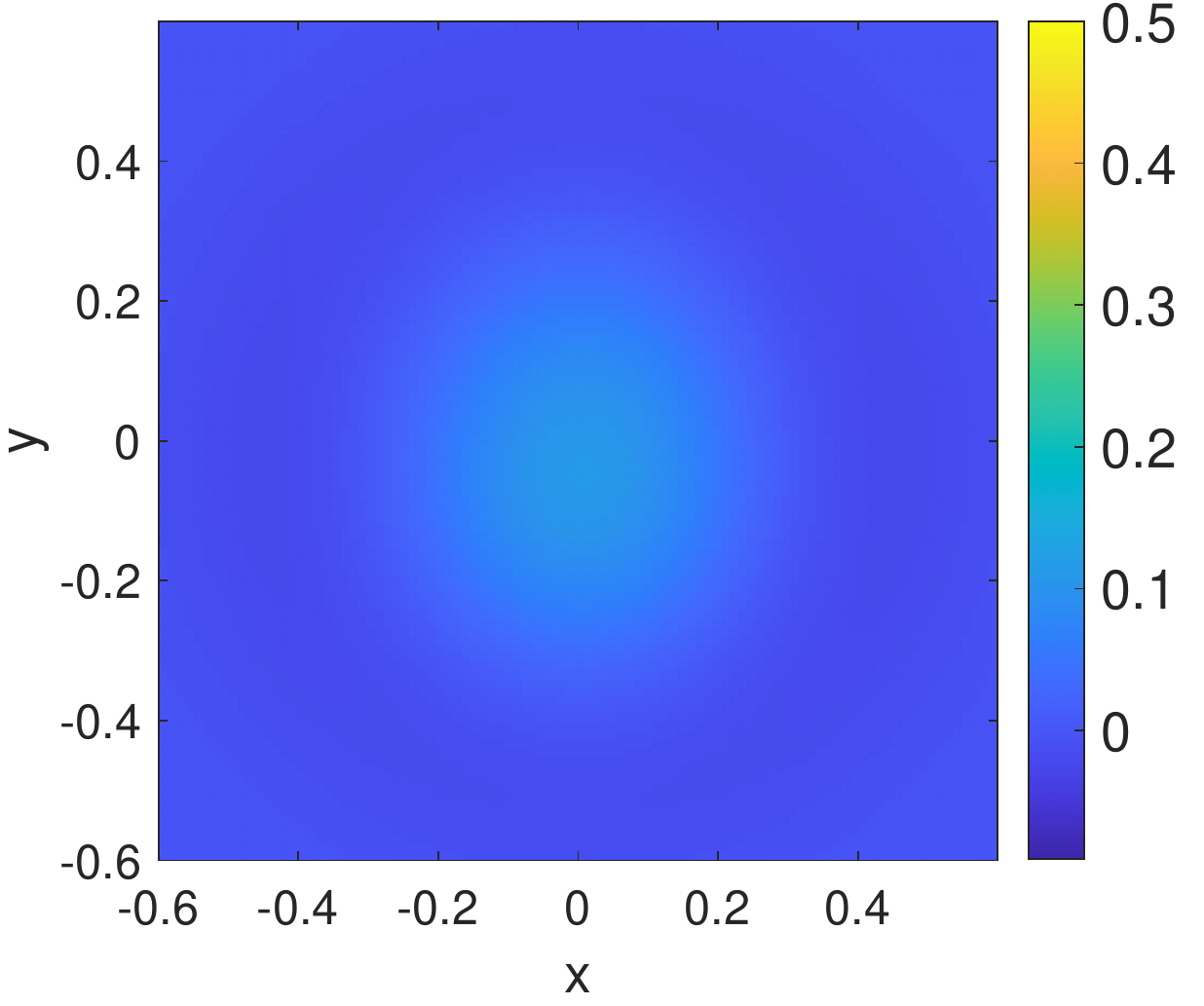}

  \caption{Recovering the Shepp-Logan phantom with $4$th-order finite difference solver for both data and inversion. The estimated contrast functions are shown for $k=2^6$ (left) and $k=2^4$ (right).
  }
  \label{fig:inverse_shepplogan_FD}
\end{figure}

Lastly, we compare the conventional inverse scattering problem and our new inverse problem using the Husimi data. We choose the incident wave $\uik = e^{\ri \omega \hat{\theta} \cdot x}$ with $\hat{\theta} \in \Sb^1$ in~\eqref{eqn:Helmholtz_incident}, and measure the scattered far field data $\usk$. Again we cast the problem as a nonlinear least square problem, and solve it using L-BFGS. We consider the initial perturbation equal to zero, and set a first order optimality tolerance of $10^{-5}$.

For simplicity, we use $4$th-order finite difference for both data generation and inversion. The setup of the computational domain and the discretization are the same as in the previous examples.

The far field measurement is taken on the boundary $\partial B(\widetilde{R})$ with $\widetilde{R}=1$. We compute the data with $180$ incident directions $\hat{\theta}$ that are equally distributed on $\Sb^1$ and $180$ receivers that are equally distributed on $\partial B(\widetilde{R})$. We add $5\%$ noise to the scattered data in the form of~\eqref{eqn:data_noise}.

\begin{figure}[htbp]
  \centering
  \includegraphics[width=0.4\textwidth]{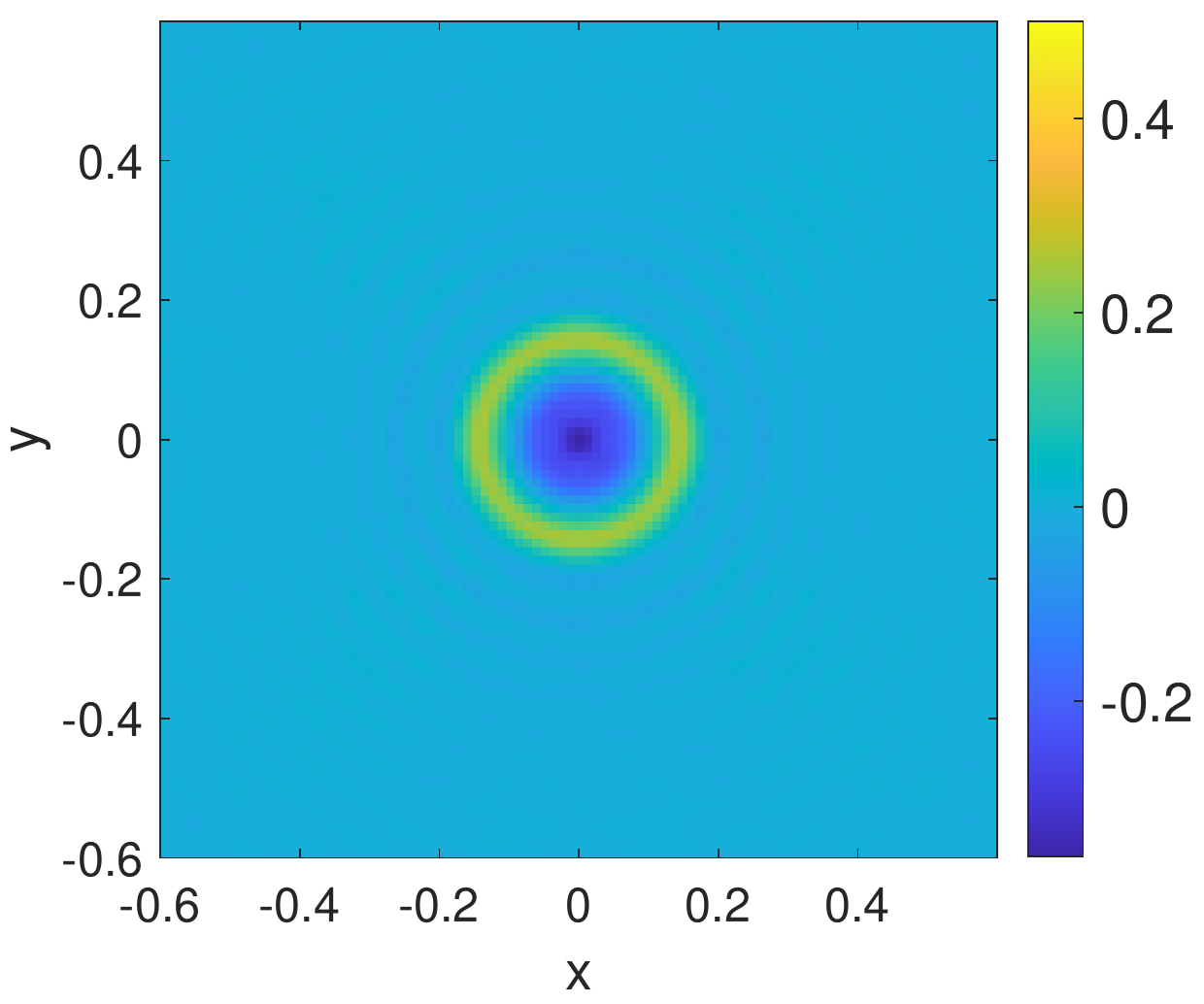}
  \includegraphics[width=0.4\textwidth]{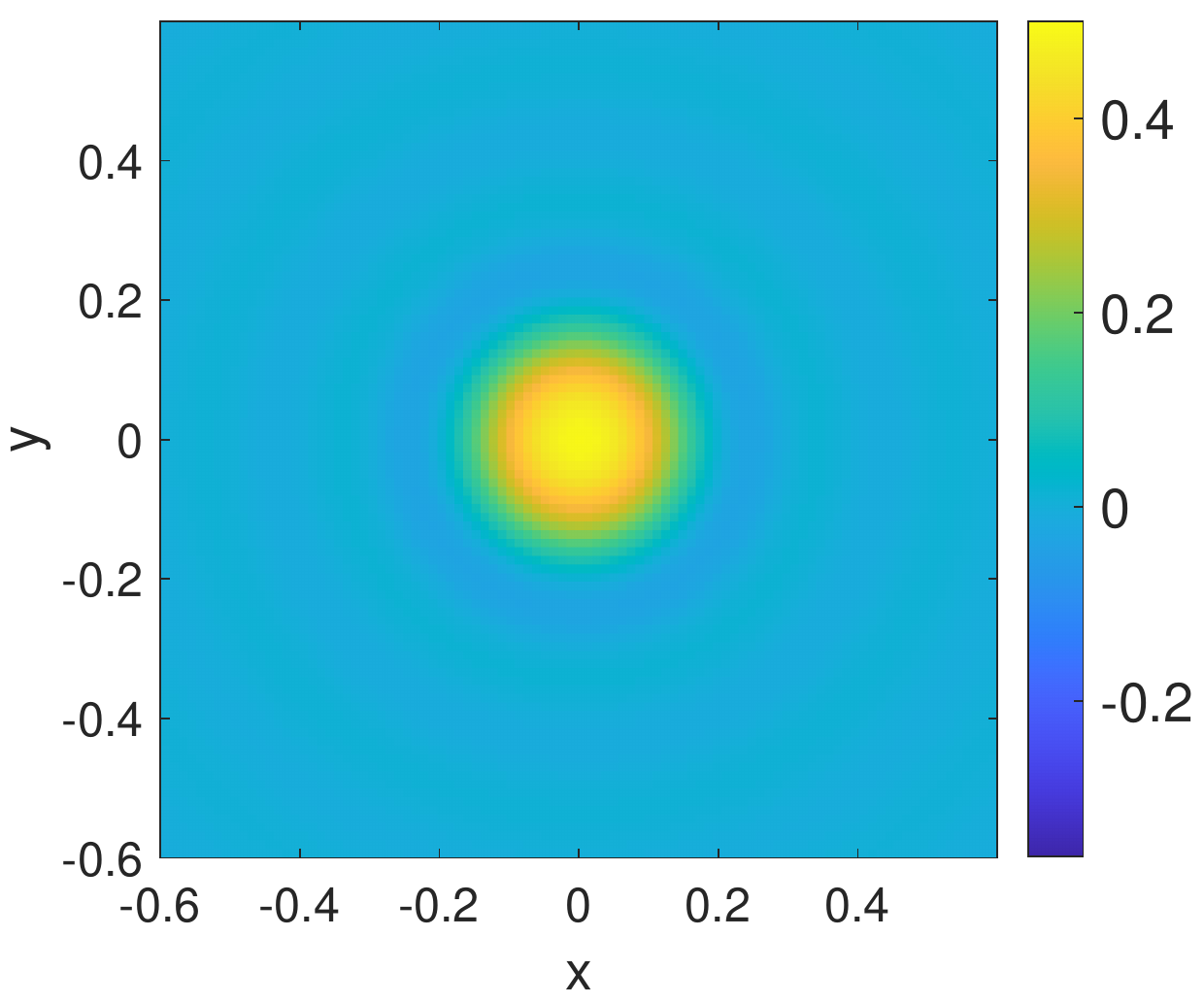}

  \caption{Recovering a single bump contrast function by plane waves. The estimated contrast function at $k=2^6$ (left) and $k=2^4$ (right) are shown. 
  $4$th-order finite difference solver is used for both data and inversion.
  }
  \label{fig:inverse_IS_bump_FD}
\end{figure}

Finally, we test the robustness of the new formulation with respect to the non-convexity of the loss function. The ill-posedness of the inverse scattering problem is often manifested as a very non-convex loss function with a myriad of local minima. As a consequence, any PDE constrained optimization-based reconstruction has a higher chance of converging to a non-physical minimum, a process that is often called cycle-skipping~\cite{ViAsBrMeRiZh:2017introduction}. For comparing the new formulation and the traditional one we also run the classical full-wave form inversion in frequency domain, using data at a single frequency, using the delocalized media in Figure~\ref{fig:inverse_truth}.
As discussed in Section~\ref{rmk:classical}, in the classical formulation one probes the medium with plane waves, and the measurement operator samples the wavefield directly on the boundary of the domain of interest. Numerically, we minimize the $\ell^2$ misfit of the wavefield at the boundary, using the same L-BFGS solver as before. Initial guess is zero. We repeat the experiments for two different wave numbers that are used in the new formulation as well. The results are shown in Figures~ \ref{fig:inverse_IS_bump_FD}, \ref{fig:inverse_IS_delocal_FD}, and \ref{fig:inverse_IS_shepplogan_FD}, respectively. In the plots we can observe that at low-frequencies we recover a smoothed version of the medium, but as the frequency increases we encounter cycle-skipping, i.e., the algorithm converges to a spurious medium. This is an stark contrast with the inversion results of the new formulation shown in Figures~\ref{fig:inverse_bump_FD},~\ref{fig:inverse_delocal_FD}, and \ref{fig:inverse_shepplogan_FD}, where at low-frequency the reconstruction does not perform as well, but it is more stable at high-frequencies, providing an accurate reconstruction.

In summary the numerical experiments seems to indicate that the new inverse formulation is far more robust to cycle skipping than its traditional counterpart.

\begin{figure}[htbp]
  \centering
  \includegraphics[width=0.4\textwidth]{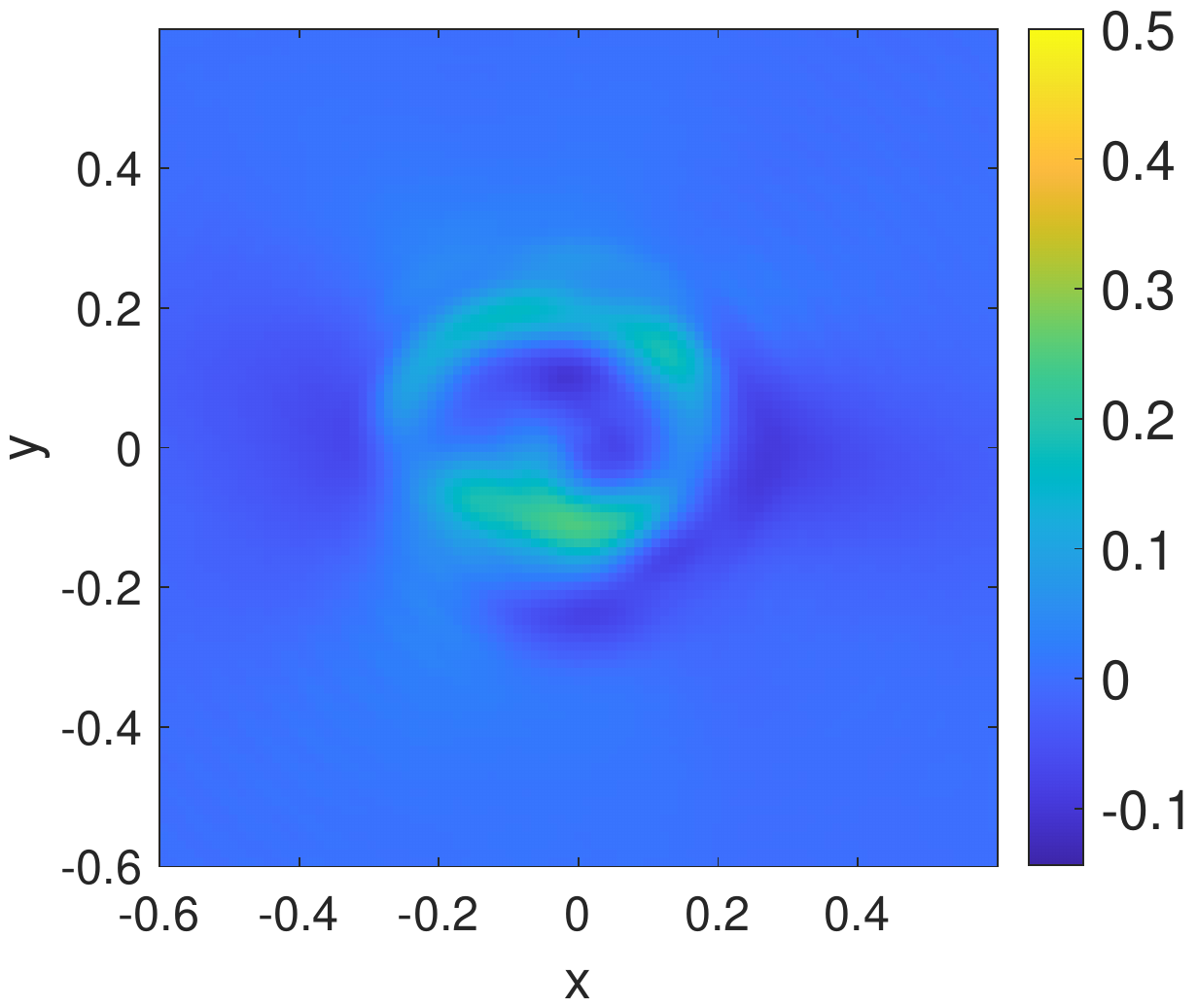}
  \includegraphics[width=0.4\textwidth]{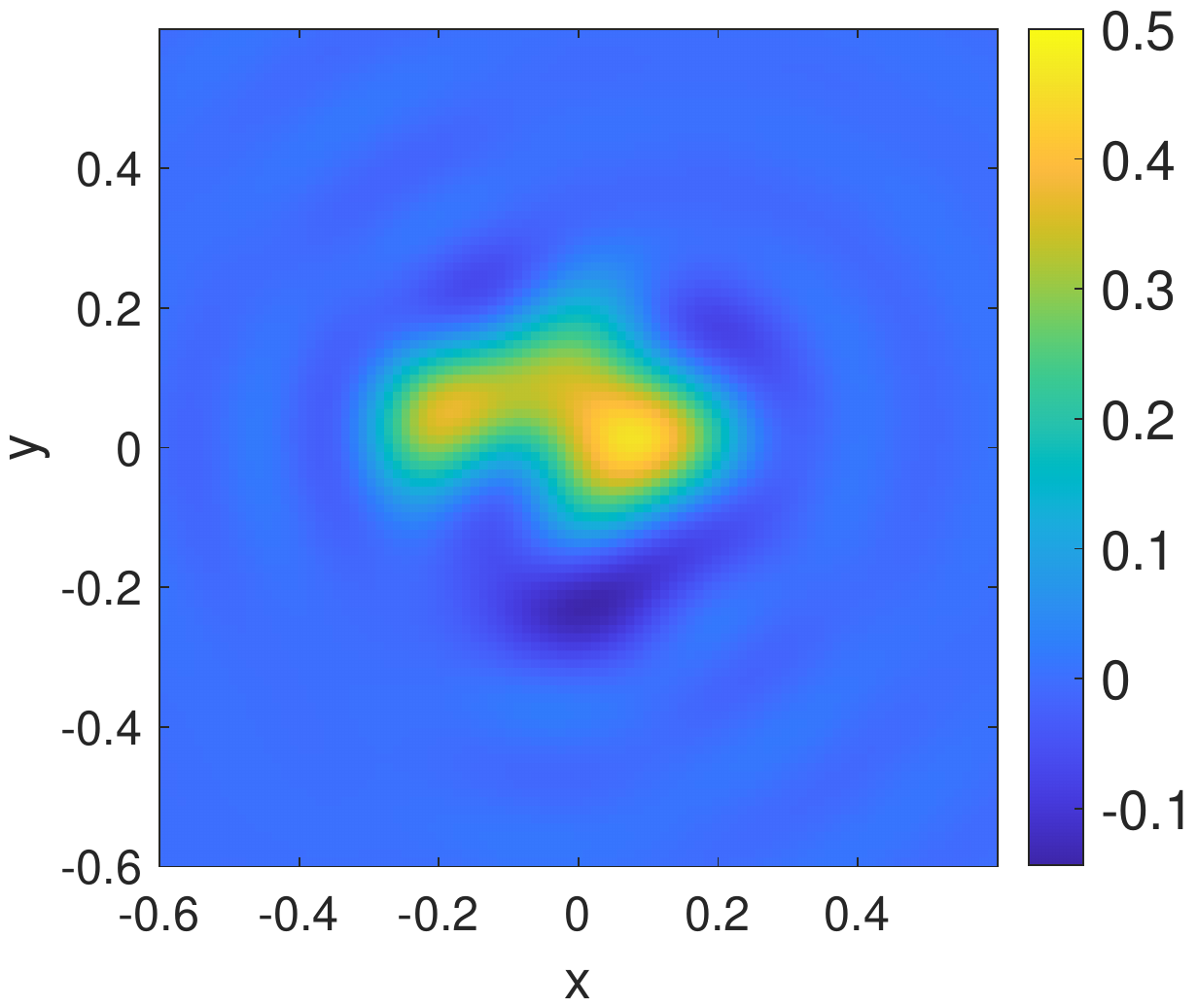}

  \caption{Recovering a delocalized contrast function by plane wave. The estimated contrast function at $k=2^6$ (left) and $k=2^4$ (right) are shown.
  $4$th-order finite difference solver is used for both data and inversion.
  }
  \label{fig:inverse_IS_delocal_FD}
\end{figure}

\begin{figure}[htbp]
  \centering
  \includegraphics[width=0.4\textwidth]{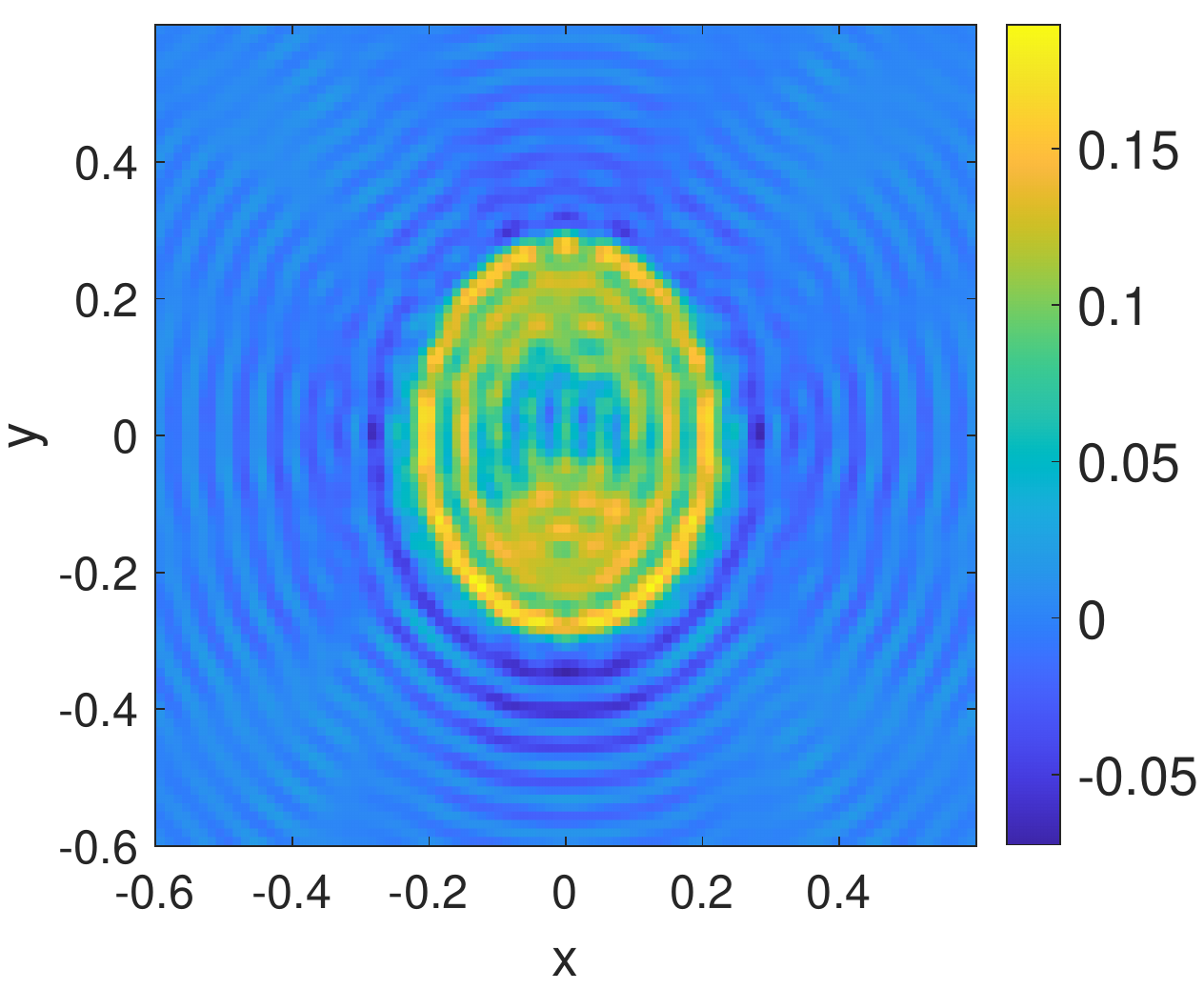}
  \includegraphics[width=0.4\textwidth]{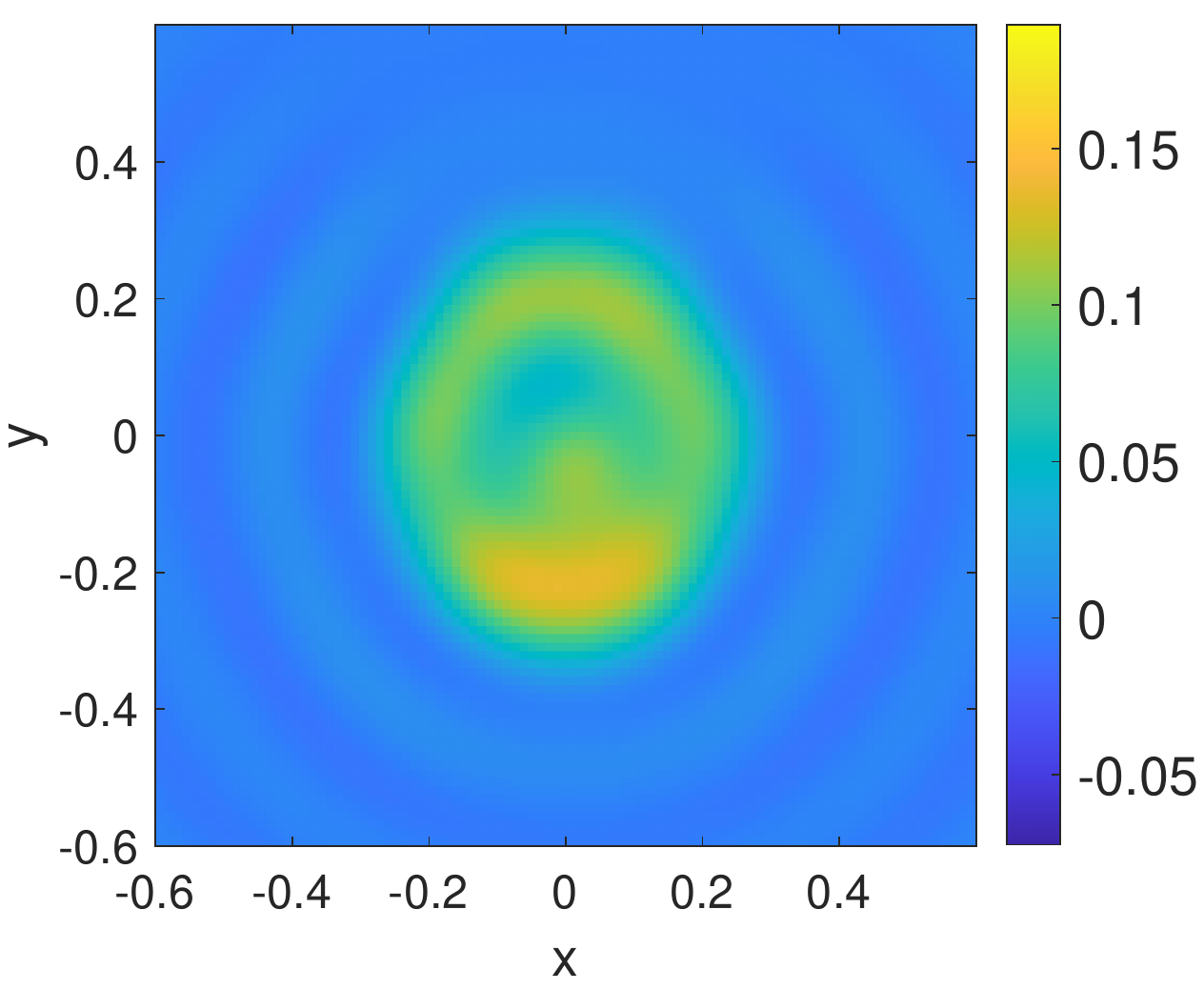}

  \caption{Recovering the Shepp-Logan phantom by plane wave. The estimated contrast function at $k=2^6$ (left) and $k=2^4$ (right) are shown.
  $4$th-order finite difference solver is used for both data and inversion.
  }
  \label{fig:inverse_IS_shepplogan_FD}
\end{figure}

\section{Conclusions}
To reconstruct an unknown medium, the generalized Helmholtz inverse scattering problem uses data pairs consisting of the impinging and scattered wave fields, while Liouville inverse scattering problems uses data pairs consisting of incoming and outgoing wave location and direction. The former is regarded ill-posed in the high-frequency regime, while the latter is well-posed. This is intuitively contradicting to the fact that Liouville equation is the asymptotic limit of the Helmholtz equation.

We investigate this issue in this paper. In particular, we develop a new formulation for studying the Helmholtz inverse scattering problem with a new data collection process, and we show that this new formulation, in the high-frequency limit, becomes the Liouville inverse scattering problem, and thus inherits the well-posedness nature. This discovery bares the conceptual merit of providing the mathematical description of the wave-particle duality for light propagation in the inverse setting. In addition, this discovery also suggests a more stable numerical reconstruction process for studying the Helmholtz inverse scattering problem, which we showcase using several numerical experiments.

\bibliographystyle{siamplain}
\bibliography{ref}

\begin{thebibliography}{10}

\bibitem{AlDaGrYi:2019stable}
{\sc R.~Alaifari, I.~Daubechies, P.~Grohs, and R.~Yin}, {\em Stable phase
  retrieval in infinite dimensions}, Foundations of Computational Mathematics,
  19 (2019), pp.~869--900.

\bibitem{AtAp:1997inverse}
{\sc D.~Atkinson and N.~D. Aparicio}, {\em An inverse problem method for crack
  detection in viscoelastic materials under anti-plane strain}, Int. J. Eng.
  Sci., 35 (1997), pp.~841 -- 849.

\bibitem{BaPaRy:2002radiative}
{\sc G.~Bal, G.~Papanicolaou, and L.~Ryzhik}, {\em Radiative transport limit
  for the random {S}chr{\"o}dinger equation}, Nonlinearity, 15 (2002), p.~513.

\bibitem{BaUh:2010inverse}
{\sc G.~Bal and G.~Uhlmann}, {\em Inverse diffusion theory of photoacoustics},
  Inverse Problems, 26 (2010), p.~085010.

\bibitem{BaLiLiTr:2015inverse}
{\sc G.~Bao, P.~Li, J.~Lin, and F.~Triki}, {\em Inverse scattering problems
  with multi-frequencies}, Inverse Problems, 31 (2015), p.~093001.

\bibitem{BaZh:2014sensitivity}
{\sc G.~Bao and H.~Zhang}, {\em Sensitivity analysis of an inverse problem for
  the wave equation with caustics}, Journal of the American Mathematical
  Society, 27 (2014), pp.~953--981.

\bibitem{BaLeRa:1992sharp}
{\sc C.~Bardos, G.~Lebeau, and J.~Rauch}, {\em Sharp sufficient conditions for
  the observation, control, and stabilization of waves from the boundary}, SIAM
  journal on control and optimization, 30 (1992), pp.~1024--1065.

\bibitem{BeCaKaPe:2002high}
{\sc J.-D. Benamou, F.~Castella, T.~Katsaounis, and B.~Perthame}, {\em High
  frequency limit of the {H}elmholtz equations}, Rev. Mat. Iberoam., 18 (2002),
  pp.~187--209.

\bibitem{BoGiGr:2017high}
{\sc C.~Borges, A.~Gillman, and L.~Greengard}, {\em High resolution inverse
  scattering in two dimensions using recursive linearization}, SIAM J. Imaging
  Sci., 10 (2017), pp.~641--664.

\bibitem{Bo:1926Quantenmechanik}
{\sc M.~Born}, {\em Quantenmechanik der sto{\ss}vorg{\"a}nge}, Zeitschrift
  f{\"u}r Physik, 38 (1926), pp.~803--827.

\bibitem{ByLuNoZh:1995limited}
{\sc R.~H. Byrd, P.~Lu, J.~Nocedal, and C.~Zhu}, {\em A limited memory
  algorithm for bound constrained optimization}, §SIAM J. Sci. Comput., 16
  (1995), pp.~1190--1208.

\bibitem{CaPeRu:2002high}
{\sc F.~Castella, B.~Perthame, and O.~Runborg}, {\em High frequency limit of
  the {H}elmholtz equation. {II}. {S}ource on a general smooth manifold}, Comm.
  Partial Differential Equations,  (2002).

\bibitem{ChDiLiZe:2021Husimi}
{\sc S.~Chen, Z.~Ding, Q.~Li, and L.~Zepeda-N{\'u}nez}, {\em {Inverse
  scattering with Husimi data}}, 12 2021,
  \url{https://github.com/Forgotten/inverse_scattering/tree/husimi}.

\bibitem{ChLi:2021semiclassical}
{\sc S.~Chen and Q.~Li}, {\em Semiclassical limit of an inverse problem for the
  {S}chr{\"o}dinger equation}, Res. Math. Sci., 8 (2021), pp.~1--18.

\bibitem{ChLiYa:2021classical}
{\sc S.~Chen, Q.~Li, and X.~Yang}, {\em Classical limit for the varying-mass
  {S}chr{\"o}dinger equation with random inhomogeneities}, J. Comput. Phys.,
  438 (2021), p.~110365.

\bibitem{Ch:1997inverse}
{\sc Y.~Chen}, {\em Inverse scattering via {H}eisenberg's uncertainty
  principle}, Inverse Problems, 13 (1997), p.~253.

\bibitem{Ch:2001mathematical}
{\sc M.~Cheney}, {\em A mathematical tutorial on synthetic aperture radar},
  SIAM Rev., 43 (2001), pp.~301--312.

\bibitem{ChQiUhZh:2007new}
{\sc E.~Chung, J.~Qian, G.~Uhlmann, and H.~Zhao}, {\em A new phase space method
  for recovering index of refraction from travel times}, Inverse Problems, 23
  (2007), p.~309.

\bibitem{CoPrFrSeVeCoBaDeCa:2015use}
{\sc A.~Colli, D.~Prati, M.~Fraquelli, S.~Segato, P.~P. Vescovi, F.~Colombo,
  C.~Balduini, S.~Della~Valle, and G.~Casazza}, {\em The use of a pocket-sized
  ultrasound device improves physical examination: Results of an in-and
  outpatient cohort study}, PLOS ONE, 10 (2015), pp.~1--10.

\bibitem{CoKr:2019inverse}
{\sc D.~L. Colton and R.~Kress}, {\em Inverse {A}coustic and {E}lectromagnetic
  {S}cattering {T}heory}, vol.~93, Springer, 2019.

\bibitem{CoRo:2012coherent}
{\sc M.~Combescure and D.~Robert}, {\em Coherent {S}tates and {A}pplications in
  {M}athematical Physics}, Springer, 2012.

\bibitem{DeDa:2017numerical}
{\sc M.~de~Buhan and M.~Darbas}, {\em Numerical resolution of an
  electromagnetic inverse medium problem at fixed frequency}, Comput. Math.
  Appl., 74 (2017), pp.~3111 -- 3128.

\bibitem{DeKr:2013new}
{\sc M.~de~Buhan and M.~Kray}, {\em A new approach to solve the inverse
  scattering problem for waves: combining the {TRAC} and the adaptive inversion
  methods}, Inverse Problems, 29 (2013), p.~085009.

\bibitem{EnRu:2003computational}
{\sc B.~Engquist and O.~Runborg}, {\em Computational high frequency wave
  propagation}, Acta Numer., 12 (2003), p.~181–266.

\bibitem{GeMaMaPo:1997homogenization}
{\sc P.~G{\'e}rard, P.~A. Markowich, N.~J. Mauser, and F.~Poupaud}, {\em
  Homogenization limits and {W}igner transforms}, Comm. Pure Appl. Math., 50
  (1997), pp.~323--379.

\bibitem{GrKoRa:2020phase}
{\sc P.~Grohs, S.~Koppensteiner, and M.~Rathmair}, {\em Phase retrieval:
  Uniqueness and stability}, SIAM Review, 62 (2020), pp.~301--350.

\bibitem{GrRa:2019stable}
{\sc P.~Grohs and M.~Rathmair}, {\em Stable gabor phase retrieval and spectral
  clustering}, Communications on Pure and Applied Mathematics, 72 (2019),
  pp.~981--1043.

\bibitem{HaHo:2001new}
{\sc P.~H{\"a}hner and T.~Hohage}, {\em New stability estimates for the inverse
  acoustic inhomogeneous medium problem and applications}, SIAM J. Math. Anal.,
  33 (2001), pp.~670--685.

\bibitem{Is:2017inverse}
{\sc V.~Isakov}, {\em Inverse {P}roblems for {P}artial {D}ifferential
  {E}quations}, vol.~127 of Applied Mathematical Sciences, Springer, Cham,
  third~ed., 2017.

\bibitem{King:1981principles}
{\sc M.~C. KING}, {\em Chapter 2 - principles of optical lithography}, in VLSI
  Electronics: Microstructure Science, N.~G. Einspruch, ed., vol.~1 of VLSI
  Electronics Microstructure Science, Elsevier, 1981, pp.~41--81.

\bibitem{Ki:2011introduction}
{\sc A.~Kirsch}, {\em An {I}ntroduction to the {M}athematical {T}heory of
  {I}nverse {P}roblems}, vol.~120, Springer, 01 2011.

\bibitem{LaLiUh:2019inverse}
{\sc R.-Y. Lai, Q.~Li, and G.~Uhlmann}, {\em Inverse problems for the
  stationary transport equation in the diffusion scaling}, SIAM J. Appl. Math.,
  79 (2019), pp.~2340--2358.

\bibitem{LiDe:2016full}
{\sc Y.~E. Li and L.~Demanet}, {\em Full-waveform inversion with extrapolated
  low-frequency data}, Geophysics, 81 (2016), pp.~R339--R348.

\bibitem{LuYa:2010frozen}
{\sc J.~Lu and X.~Yang}, {\em Frozen gaussian approximation for high frequency
  wave propagation}, Commun. Math. Sci., 9 (2010), pp.~663--683.

\bibitem{MonardStefanovUhlmann_geodesic_ray}
{\sc F.~Monard, P.~Stefanov, and G.~Uhlmann}, {\em The geodesic ray transform
  on riemannian surfaces with conjugate points}, Communications in Mathematical
  Physics, 337 (2015).

\bibitem{NaUhWa:2013increasing}
{\sc S.~Nagayasu, G.~Uhlmann, and J.-N. Wang}, {\em Increasing stability in an
  inverse problem for the acoustic equation}, Inverse Problems, 29 (2013),
  p.~025012.

\bibitem{Na:2001mathematics}
{\sc F.~Natterer}, {\em The {M}athematics of {C}omputerized {T}omography},
  SIAM, 2001.

\bibitem{No:1999small}
{\sc R.~G. Novikov}, {\em Small angle scattering and {X}-ray transform in
  classical mechanics}, Ark. Mat., 37 (1999), pp.~141--169.

\bibitem{Ol:1906construction}
{\sc R.~D. Oldham}, {\em The constitution of the interior of the {E}arth, as
  revealed by earthquakes}, Quarterly Journal of the Geological Society, 62
  (1906), pp.~456--475.

\bibitem{PeWaLo:2015improved}
{\sc J.~R. Pettit, A.~E. Walker, and M.~J.~S. Lowe}, {\em Improved detection of
  rough defects for ultrasonic nondestructive evaluation inspections based on
  finite element modeling of elastic wave scattering}, IEEE T. Ultrason. Ferr.,
  62 (2015), pp.~1797--1808.

\bibitem{Pr:1999seismic}
{\sc R.~G. Pratt}, {\em Seismic waveform inversion in the frequency domain;
  part 1: {T}heory and verification in a physical scale model}, Geophysics, 64
  (1999), pp.~888--901.

\bibitem{QiYi:2010gaussian}
{\sc J.~Qian and L.~Ying}, {\em Fast multiscale {G}aussian wavepacket
  transforms and multiscale {G}aussian beams for the wave equation}, Multiscale
  Model. Simul., 8 (2010).

\bibitem{RaPoFi:2010seismic}
{\sc N.~Rawlinson, S.~Pozgay, and S.~Fishwick}, {\em Seismic tomography: a
  window into deep {E}arth}, Phys. Earth Planet. Int., 178 (2010),
  pp.~101--135.

\bibitem{RyPaKe:1996transport}
{\sc L.~Ryzhik, G.~Papanicolaou, and J.~B. Keller}, {\em Transport equations
  for elastic and other waves in random media}, Wave motion, 24 (1996),
  pp.~327--370.

\bibitem{Sc:1978improved}
{\sc H.~Schomberg}, {\em An improved approach to reconstructive ultrasound
  tomography}, J. of Phys. D: Appl. Phys., 11 (1978), pp.~L181--L185.

\bibitem{StUhVaZh:2019travel}
{\sc P.~Stefanov, G.~Uhlmann, A.~Vasy, and H.~Zhou}, {\em Travel time
  tomography}, Acta Mathematica Sinica, English Series, 35 (2019),
  pp.~1085--1114.

\bibitem{Su:1990continuous}
{\sc Z.~Sun}, {\em On continuous dependence for an inverse initial boundary
  value problem for the wave equation}, Journal of Mathematical Analysis and
  Applications, 150 (1990), pp.~188--204.

\bibitem{TaQiRa:2007mountain}
{\sc N.~M. Tanushev, J.~Qian, and J.~V. Ralston}, {\em Mountain waves and
  gaussian beams}, Multiscale Model. Simul., 6 (2007), pp.~688--709.

\bibitem{Ta:1984inversion}
{\sc A.~Tarantola}, {\em Inversion of seismic reflection data in the acoustic
  approximation}, Geophysics, 49 (1984), pp.~1259--1266.

\bibitem{ViGeFe:2016fast}
{\sc F.~Vico, L.~Greengard, and M.~Ferrando}, {\em Fast convolution with
  free-space {G}reen's functions}, J. Comput. Phys., 323 (2016), pp.~191--203.

\bibitem{ViAsBrMeRiZh:2017introduction}
{\sc J.~Virieux, A.~Asnaashari, R.~Brossier, L.~M\'etivier, A.~Ribodetti, and
  W.~Zhou}, {\em 6. An introduction to full waveform inversion}, Society of
  Exploration Geophysicists, 2017, pp.~R1--1--R1--40.

\bibitem{ViOp:2009overview}
{\sc J.~Virieux and S.~Operto}, {\em An overview of full-waveform inversion in
  exploration geophysics}, Geophysics, 74 (2009), pp.~WCC1--WCC26.

\bibitem{Wh:2001electromagnetic}
{\sc A.~D. Wheelon}, {\em Electromagnetic {S}cintillation}, vol.~1, Cambridge
  University Press, 2001.

\bibitem{Wh:2003electromagnetic}
{\sc A.~D. Wheelon}, {\em Electromagnetic {S}cintillation}, vol.~2, Cambridge
  University Press, 2003.

\bibitem{YuOlYa:2001global}
{\sc O.~Yu.~Imanuvilov and M.~Yamamoto}, {\em Global uniqueness and stability
  in determining coefficients of wave equations}, Communications in Partial
  Differential Equations, 26 (2001), pp.~1409--1425.

\bibitem{ZhByLuNo:1997algorithm}
{\sc C.~Zhu, R.~H. Byrd, P.~Lu, and J.~Nocedal}, {\em Algorithm 778: L-bfgs-b:
  Fortran subroutines for large-scale bound-constrained optimization}, ACM
  Trans. Math. Softw., 23 (1997), p.~550–560.

\end{thebibliography}

\appendix
\section{Formal derivation of Theorem \ref{thm:formal}}\label{appendix:formalderivation}

We start from the equation
\begin{equation}\label{eqn:app_schr}
\ri k \alpha^k \uk + \Delta \uk + k^2 n(x) \uk = -S^k(x) = -k^{\frac{d+3}{2}} S(k(x-x_s)) \,, \quad x\in\Rb^d \,,
\end{equation}
and assume that $\alpha^k\to\alpha\geq0$ in the limit $k\to\infty$. We denote the density matrix of $\uk$ satisfying~\eqref{eqn:app_schr} by
\begin{equation}
g^k(x,y) = \uk\left( x-\frac{y}{2k} \right) \overline{\uk} \left( x+\frac{y}{2k} \right) \,,
\end{equation}
and the Fourier transform of a generic $u$ by
\begin{equation}
\widehat{u}(v) = \Fc_{y\to v} u(y) = \frac{1}{(2\pi)^d} \int_{\Rb^d} e^{-\ri y v} u(y) \rmd y \,.
\end{equation}
The inverse Fourier transform is then
\begin{equation}
\Fc^{-1}_{v\to x} u(v) = \int_{\Rb^d} e^{\ri x v} u(v) \rmd v \,.
\end{equation}
Now we compute the equation satisfied by the Wigner transform. The first step is to compute the derivatives of $g^k$
\begin{equation}
\nabla_y\cdot\nabla_x g^k(x,y) = -\frac{1}{2k}\left[ \Delta \uk\left( x-\frac{y}{2k} \right) \overline{\uk} \left( x + \frac{y}{2k} \right) - \uk\left( x-\frac{y}{2k} \right) \Delta \overline{\uk} \left( x + \frac{y}{2k} \right)\right] \,,
\end{equation}
and thus we have
\begin{equation}
\begin{aligned}
\alpha^k g^k + \ri \nabla_y\cdot\nabla_x g^k(x,y)
&+ \frac{\ri k}{2} \left[ n\left( x+\frac{y}{2k} \right) - n\left( x-\frac{y}{2k} \right)\right] g^k(x,y) = \\
& = \sigma^k(x,y) \\
& := \frac{\ri}{2k} \left[ S^k\left( x-\frac{y}{2k} \right) \overline{\uk}\left( x+\frac{y}{2k} \right) - \overline{S^k}\left( x+\frac{y}{2k} \right) \uk\left( x-\frac{y}{2k} \right)\right] \,.
\end{aligned}
\end{equation}
Therefore, after a Fourier transform, we obtain the following transport equation on the Wigner transform $\fk$
\begin{equation}
\alpha^k \fk(x,v) + v\cdot \nabla_x \fk(x,v) + Z^k(x,v) \ast_v \fk(x,v) = Q^k(x,v) \,,
\end{equation}
where the last term denotes the convolution in $v$
\[
Z^k(x,v) \ast_v \fk(x,v) = \int_{\Rb^d} Z^k(x,v-p) \fk(x,p) \rmd p
\]
and the quantities $Z^k$, $Q^k$ arising in this equation are given by
\begin{equation}
\begin{aligned}&Z^k(x,v) = \frac{1}{(2\pi)^d}\frac{\ri k}{2} \Fc_{y\to v}^{-1} \left[ n\left( x+\frac{y}{2k} \right) - n\left( x-\frac{y}{2k} \right) \right] \,,\\
&Q^k(x,v) = \frac{1}{(2\pi)^d}\Fc_{y\to v}^{-1} \sigma^k(x,y) \,.
\end{aligned}
\end{equation}
From this equation we can compute the formally compute the limits. For $Z^k$ we have that
\begin{equation}
Z^k(x,v) \xrightarrow{k\to\infty} \frac{1}{(2\pi)^d} \frac{\ri}{2} (\Fc_{y\to v}^{-1} y) \cdot \nabla_x n(x) = -\frac{1}{2} \nabla_x n(x) \cdot \nabla_v \delta(v) \,.
\end{equation}
The limit of the source term $Q^k$ is slightly more involved. First, we define the complex valued function
\begin{equation}
w^k(y) = \frac{1}{k^{\frac{d-1}{2}}} \uk\left( x_s + \frac{y}{k} \right) \,,
\end{equation}
which after a change of variable can be rewritten as
\begin{equation}
\uk(x) = k^{\frac{d-1}{2}} w^k(k(x-x_s)) \,,
\end{equation}
where function $w^k$ satisfies the rescaled Helmholtz equation
\begin{equation}
\ri \frac{\alpha^k}{k} w^k + \Delta w^k + n\left( x_s+\frac{y}{k} \right) w^k = -S(y) \,.
\end{equation}
In the high-frequency limit, $w^k$ converges towards a solution $w$ of
\begin{equation}\label{eqn:app_rescaled_w}
\Delta w + n(x_s)w = -S(y) \,.
\end{equation}
The second step is to compute the Fourier transform of $w$. To do so, we add an absorption term to the equation above, resulting in
\begin{equation}
\ri \beta w + \Delta w + n(x_s)w = -S(y) \,.
\end{equation}
where $\beta>0$. This new term, is used as a broadening factor, which helps to smooth the Fourier transform. We perform a Fourier transform on both sides, which leads to
\begin{equation}
\widehat{w}(v) = \frac{-\hat{S}(v)}{n(x_s)-|v|^2 + \ri\beta}
= \hat{S}(v) \hat{G}(v;\beta)\,.
\end{equation}
where $\hat{G}(v;\beta)$ denotes the Fourier transform of the outgoing Green's function that vanishes at infinity
\begin{equation}
\hat{G}(v;\beta)
\equiv -\frac{1}{n(x_s)-|v|^2 + \ri\beta}
= -\frac{n(x_s)-|v|^2}{(n(x_s)-|v|^2)^2 + \beta^2} + \frac{\ri\beta}{(n(x_s)-|v|^2)^2 + \beta^2} \,, \quad \beta>0 \,.
\end{equation}
As usual, we take the limit $\beta\to0+$. The first term converges weakly to the principal value
\begin{equation}
-\frac{n(x_s)-|v|^2}{(n(x_s)-|v|^2)^2 + \beta^2} \xrightarrow{\beta\to0+} -\text{P.V.} \left( \frac{1}{n(x_s)-|v|^2} \right)\,.
\end{equation}
The second term converges to a delta function on the sphere $\{|v|^2 = n(x_s)\}$ as $\beta\to0+$
\begin{equation}
\frac{\ri\beta}{(n(x_s)-|v|^2)^2 + \beta^2} \xrightarrow{\beta\to0+}
\frac{\ri \pi}{2}\delta(|v|^2 = n(x_s))\,.
\end{equation}
In summary, we obtain the Fourier transform of the outgoing solution to~\eqref{eqn:app_rescaled_w}
\begin{equation}\label{eqn:app_fourier_w_limit}
\widehat{w}(v) = \lim_{\beta\to 0+}\hat{S}(v) \hat{G}(v;\beta)
= \hat{S}(v) \left[ \frac{\ri \pi}{2}\delta(|v|^2 = n(x_s)) - \text{P.V.} \left( \frac{1}{n(x_s)-|v|^2} \right) \right]\,.
\end{equation}
Now we are ready to compute $Q^k$. We take two test functions $\phi(x)$ and $\psi(y)$
\begin{equation}
\begin{aligned}
&\quad \int_{\Rb^{2d}} \sigma^k(x,y)\phi(x)\psi(y) \, \rmd x \, \rmd y \\
&=\frac{\ri}{2k} \int_{\Rb^{2d}} \left[ S^k\left( x-\frac{y}{2k} \right) \overline{\uk}\left( x+\frac{y}{2k} \right) - \overline{S^k}\left( x+\frac{y}{2k} \right) \uk\left( x-\frac{y}{2k} \right)\right] \phi(x)\psi(y)\rmd x \rmd y \\
&=\frac{\ri k^d}{2} \int_{\Rb^d} \bigg[S\left( k\left(x-\frac{y}{2k} -x_s\right)\right) \overline{w^k}\left( k\left(x+\frac{y}{2k}-x_s\right) \right)\\
& \qquad \qquad - \overline{S}\left( k\left( x+\frac{y}{2k}-x_s \right) \right) w^k\left( k\left( x-\frac{y}{2k} - x_s \right) \right) \bigg] \phi(x)\psi(y)\rmd x \rmd y \\
&=\frac{\ri}{2} \int_{\Rb^{2d}} \left[ S(z)\overline{w^k}(z+y)\phi\left(\frac{z}{k}+\frac{y}{2k}+x_s\right)
 - \overline{S}(z)w^k(z-y)\phi\left(\frac{z}{k}-\frac{y}{2k}+x_s\right) \right] \psi(y) \rmd z \rmd y \\
&\xrightarrow{k\to\infty} \frac{\ri}{2} \phi(x_s) \int_{\Rb^{2d}} \left[ S(z)\overline{w}(z+y)
 - \overline{S}(z)w(z-y)\right] \psi(y) \rmd z \rmd y \,.
\end{aligned}
\end{equation}
In other words, we have formally obtained that
\begin{equation}
\sigma^k(x,y) \xrightarrow{k\to\infty} \frac{\ri}{2} \delta(x-x_s)\int_{\Rb^d} \left[ S(z)\overline{w}(z+y)
 - \overline{S}(z)w(z-y) \right] \rmd z \,,
\end{equation}
which after a Fourier transform gives
\begin{equation}\label{eqn:app_Qk}
\begin{aligned}
&\qquad Q^k(x,v) = \frac{1}{(2\pi)^d}\Fc_{y\to v}^{-1} \sigma^k(x,y) \\
&\xrightarrow{k\to\infty} \frac{1}{(2\pi)^d}\frac{\ri}{2} \delta(x-x_s) \Fc_{y\to v}^{-1} \left\{\int_{\Rb^d} \left[ S(z)\overline{w}(z+y)
 - \overline{S}(z)w(z-y) \right] \rmd z \right\} \\
&= \frac{\ri}{2} \delta(x-x_s) (2\pi)^d \left[\hat{S}(v) \overline{\hat{w}(v)} - \overline{\hat{S}(v)} \hat{w}(v)\right] \\
&= (2\pi)^d\delta(x-x_s) \im\left[\overline{\hat{S}(v)} \hat{w}(v) \right] \,.
\end{aligned}
\end{equation}
We finally obtain
\begin{equation}
Q^k(x,v) \xrightarrow{k\to\infty} (2\pi)^d\frac{\pi}{2}\delta(x-x_s) |\hat{S}(v)|^2 \delta(|v|^2 = n(x_s)) \,.
\end{equation}
by substituting~\eqref{eqn:app_fourier_w_limit} in~\eqref{eqn:app_Qk}.
\end{document}